\documentclass[12pt]{article}
\usepackage{amsmath} 
\usepackage{amssymb}  
\usepackage{amsfonts}
\usepackage{graphicx} 
\usepackage{times} 
\usepackage{framed}

\newtheorem{definition}{Definition}
\newtheorem{lemma}{Lemma}
\newtheorem{theorem}{Theorem}

\def\wh{\widehat}
\def\wt{\widetilde}

\def\x{\times}

\def\op{\oplus}
\def\Span{\mathrm{span}}        
\def\mod{\mathrm{mod}}        
\def\re{\mathrm{e}}        
\def\<{\leqslant}           
\def\>{\geqslant}           

\def\rT{\mathrm{T}}
\def\rd{\mathrm{d}}

 \newcommand{\bone}{{\bf 1}}

\def\mC{\mathbb{C}}
\def\bF{\mathbf{F}}

\def\bE{\mathbf{E}}
\def\bP{\mathbf{P}}
\def\cQ{\mathcal{Q}}
\def\cW{\mathcal{W}}
\def\cD{\mathcal{D}}
\def\cN{\mathcal{N}}
\def\cE{\mathcal{E}}
\def\cB{\mathcal{B}}
\def\cM{\mathcal{M}}
\def\cL{\mathcal{L}}
\def\bA{\mathbf{A}}
\def\bD{\mathbf{D}}
\def\cI{\mathcal{I}}
\def\cK{\mathcal{K}}

\def\Re{\mathrm{Re}}
\def\Im{\mathrm{Im}}
\def\mN{\mathbb{N}}

\def\cP{\mathcal{P}}

\def\mR{\mathbb{R}}
\def\cS{\mathcal{S}}

\def\mZ{\mathbb{Z}}
\def\eps{\varepsilon}
\def\phi{\varphi}

\def\mes{\mathrm{mes}}

\def\Tr{\mathrm{Tr}}

\def\lfp{\{\!\!\{}
\def\rfp{\}\!\!\}}
\def\sn{| \! | \! |}

\begin{document}
\title{\bf
Quantized linear systems on integer lattices: \\ a frequency-based approach\footnote{This work is a slightly edited version of two research reports: I.Vladimirov, ``Quantized linear systems on integer lattices: frequency-based approach'', Parts I, II, Centre for Applied Dynamical Systems and Environmental Modelling, CADSEM Reports 96--032, 96--033, October 1996, Deakin University, Geelong, Victoria, Australia, which were issued while the author was with the Institute for Information Transmission Problems, the Russian Academy of Sciences, Moscow, 127994, GSP--4, Bolshoi Karetny Lane 19. None of the original results have been removed, nor have new results been added in the present version except for a numerical example on p.~58 of the last section.
}}
\author{Igor G. Vladimirov\thanks{Present address: UNSW Canberra, ACT 2600, Australia. E-mail: {\tt
         igor.g.vladimirov@gmail.com}.}
}
\date{}

\maketitle
\thispagestyle{empty}

\centerline{\emph{In loving memory of my father, Gennadiy Ivanovich Vladimirov (21.04.1933 -- 24.12.2014)}}

\begin{abstract}
The roundoff errors in computer simulations of continuous dynamical systems, which are caused by finiteness of machine arithmetic, can lead to qualitative discrepancies between phase portraits of the resulting spatially discretized systems and the original systems. These effects can be modelled on a multidimensional integer lattice by using a dynamical system    obtained by composing the transition operator of the original system with a quantizer representing the computer discretization procedure. Such model systems manifest pseudorandomness which can be studied using a rigorous  probability theoretic approach.  To this end, the lattice $\mathbb{Z}^n$ is endowed with a class of frequency measurable subsets and a spatial frequency functional as a finitely additive probability measure describing the limit fractions of such sets in large rectangular fragments of the lattice. Using a multivariate version of
Weyl's equidistribution  criterion and a related nonresonance condition,  we introduce an algebra of frequency  measurable quasiperiodic subsets of the lattice. The frequency-based approach is applied to quantized linear systems  with the transition operator $R \circ L$, where $L$ is a nonsingular matrix of the original linear system in $\mathbb{R}^n$, and $R:\mR^n\to \mZ^n$ is a quantizer (in an idealized fixed-point arithmetic with no overflow) which commutes with the additive group of translations of the lattice. It is shown that, for almost every $L$, the events associated with the deviation of trajectories of the quantized  and original systems are frequency measurable quasiperiodic subsets of the lattice whose frequencies are amenable to computation involving  geometric probabilities on finite-dimensional tori.  Using the skew products of measure preserving toral automorphisms, we prove
mutual independence and uniform distribution of the quantization errors and investigate statistical properties of invertibility loss for the quantized linear system, thus extending the results obtained by V.V.Voevodin in the 1960s. In the case when $L$ is similar to an orthogonal matrix, we establish a functional central limit theorem for the deviations of trajectories of the quantized and original systems. As an example, these results are demonstrated for rounded-off planar rotations.
\end{abstract}




\newpage
\tableofcontents

\section{Introduction}

The roundoff errors in computer simulations of continuous dynamical systems, which are  caused by finiteness of machine arithmetic, can lead to dramatic discrepancies between phase portraits of the resulting spatially discrete systems and the original systems \cite{Blank,Kozyakin,Kuznetsov}. A model class of these spatially discretized systems is provided by dynamical systems on multidimensional integer lattices,  obtained by composing the transition operator of the original system with a quantizer  which represents the computer discretization procedure \cite{Diamond1,Vivaldi1,Vivaldi2}. Trajectories of such model systems exhibit pseudorandomness. This  suggests that their statistical properties can be studied from a probability theoretic point of view  by equipping the integer lattice  with a probability measure so as to quantify the discretization effects for a sufficiently wide class of dynamical systems.

The present study is concerned with those ``events'' on the lattice which are associated with the deviation of trajectories of the discretized and original systems. We develop a rigorous probabilistic approach to quantized linear systems on the $n$-dimensional integer lattice $\mZ^n$ with the transition operator $R \circ L$, where $L$ is a nonsingular matrix of the original linear system in $\mR^n$, and the map $R$ is a quantizer in an idealized fixed-point arithmetic with no overflow.
To this end, the lattice $\mZ^n$ is endowed with a class of frequency measurable subsets and a spatial frequency functional as a finitely additive probability measure on such sets. We introduce an algebra of quasiperiodic subsets of the lattice, which are frequency  measurable under a nonresonance  condition, and apply  this frequency-based  approach to the quantized linear systems.   We show that almost every matrix $L$ satisfies the nonresonance condition, and the events associated with the deviation of trajectories of the quantized  and original systems over a bounded time interval are representable as frequency measurable quasiperiodic subsets of the lattice. The frequencies of these sets are amenable to computation which involves geometric probabilities on finite-dimensional tori.  In particular, for the generic matrices $L$, satisfying the nonresonance condition, we prove the mutual independence and uniform distribution of quantization errors, and investigate the statistical properties of the invertibility loss for the transition operator. In the case where $L$ is similar to an orthogonal matrix, this allows a functional central limit theorem to be established for the deviations of trajectories of the quantized and original systems. As an illustration, we apply these results to a two-dimensional quantized linear system associated with the problem of rounded-off planar rotations.

Given the length of this report, we will now outline its structure and main results. In Section~\ref{PSIL}, the $n$-dimensional integer lattice $\mZ^n$ is endowed with a \emph{frequency} functional $\bF$. The frequency $\bF(A)$ of a set $A \subset \mZ^n$ is defined as the limit fraction of  points of the set in unboundedly increasing rectangular fragments of the lattice. The class $\cS$ of \textit{frequency measurable sets}, for which this limit exists, is closed under the union of disjoint sets and the complement of a set to the lattice, and $\bF$ is a \emph{finitely} additive probability measure on the class $\cS$ os such subsets of the lattice. The frequency $\bF$ is defined in Section \ref{FMS} as an \emph{average value} of the indicator function of a set, with the average value functional  $\bA$ on a class of \emph{averageable functions} being introduced in Section~\ref{AF}. Since $\cS$ is not a $\sigma$-algebra, and $\bF$ is not a \emph{countably} additive measure, the triple $(\mZ^n, \cS, \bF)$ can be regarded as  an unusual probability space (where the average value functional $\bA$ plays the role of expectation) which does not satisfy Kolmogorov's axioms \cite{Shiryayev}.  Nevertheless, this  probability space allows a random element with values in a metric space $X$ to be defined as a map $g: \mZ^n\to X$ for which there exists a countably additive probability measure $D$ such that the preimage $g^{-1}(B)$ of every $D$-continuous \cite{Billingsley} Borel set $B \subset X$ is frequency measurable with frequency $\bF(g^{-1}(B)) = D(B)$.  Any such map $g$ is called $D$-\emph{distributed} and generates an algebra $\cS_g \subset \cS$ of events related to the behaviour of $g$.  The distributed maps are discussed in Section~\ref{DM}. The averagability of a function, the frequency measurability of a set and the property of a map to be distributed are closely related to each other and can be formulated in terms of the weak convergence of probability measures in metric spaces \cite{Billingsley}.

Section~\ref{QOIL} introduces a class of \emph{quasiperiodic} objects which are nontrivial examples of averageable functions, frequency measurable sets and  distributed maps. Section~\ref{QS} describes quasiperiodic subsets of the lattice $\mZ^n$.  Any such set is specified by a Jordan measurable set $G \subset [0,1)^m$ and a matrix $\Lambda \in \mR^{m \x n}$, and is denoted by $Q_m(G,\Lambda)$. The set $Q_m(G,\Lambda)$ is formed by those points $x \in \mZ^n$ for which the $m$-dimensional vector $\Lambda x$, whose entries are considered modulo one, belongs to $G$. The resulting class $\cQ_m(\Lambda)$ of $\Lambda$-quasiperiodic sets is an algebra with respect to set theoretic operations. Furthermore, Section \ref{FMQS}  shows that if the matrix $\Lambda$ is \emph{nonresonant} in the sense that the rows of the matrix $\begin{bmatrix}I_n \\ \Lambda\end{bmatrix}$  are rationally independent (with $I_n$ the identity matrix of order $n$), then every $\Lambda$-quasiperiodic set $Q_m(G,\Lambda)$ is frequency measurable and its frequency coincides with the $m$-dimensional Lebesgue measure $\mes_m G$ of the set $G$. Similarly, a $\Lambda$-quasiperiodic function on  $\mZ^n$ is defined as the composition $f \circ \Lambda$ of the linear map, specified by a matrix $\Lambda \in \mR^{m \x n}$, and a $\mes_m$-continuous bounded function $f: \mR^m \to \mR$, unit periodic in its $m$ variables. In a similar vein, Section~\ref{DQM} defines a $\Lambda$-quasiperiodic map from $\mZ^n$ to a metric space $X$ as the composition $g \circ \Lambda$ of the linear map $\Lambda$ with a unit periodic $\mes_m$-continuous map $g: \mR^m \to X$. It is shown that if the matrix $\Lambda$ is nonresonant,  then every $\Lambda$-quasiperiodic function $f \circ \Lambda$ is averageable, with $\bA(f \circ \Lambda) = \int_{[0,1]^m} f(u)\rd u$, and any $\Lambda$-quasiperiodic map $g \circ \Lambda$ is distributed, with the algebra $\cS_{g \circ \Lambda}$ consisting of of $\Lambda$-quasiperiodic sets. It turns out that $\mes_{mn}$-almost all matrices $\Lambda \in \mR^{m \x n}$ are nonresonant, and the corresponding algebras $\cQ_m(\Lambda)$ of $\Lambda$-quasiperiodic sets are parameterized by Jordan measurable subsets of $[0,1)^m$. Therefore, a typical matrix $\Lambda$ leads to sufficiently rich classes of $\Lambda$-quasiperiodic frequency measurable sets, averageable functions and distributed maps, which can be studied in the framework of the probability space $(\mZ^n, \cS, \bF)$. The above properties of quasiperiodic objects are established by using a multivariate version of what is known as the method of trigonometric sums in number theory \cite{Vinogradov} and Weyl's equidistribution criterion in the theory of weak convergence of probability measures \cite{Billingsley}. The quasiperiodic sets are then extended to infinite dimensional matrices $\Lambda$. More precisely, for a matrix $\cL \in \mR^{\infty \x n}$, the class $\cQ_{\infty}(\cL)$ of $\cL$-quasiperiodic sets is defined as the union of  the algebras $\cQ _m(\Lambda)$ over all positive integers $m$ and all submatrices $\Lambda \in \mR^{m \x n}$ of the matrix $\cL$. The class $\cQ _{\infty}(\cL)$ is also an algebra and consists of frequency measurable subsets of the lattice $\mZ^n$, provided  $\cL\in \mR^{\infty\x n}$ is nonresonant in the sense that all its submatrices $\Lambda\in \mR^{m\x n}$ are nonresonant. The discussion of quasiperiodic objects in Sections~\ref{QS}--\ref{DQM} employs the  notion of a \emph{cell} in $\mR^m$, which extends that of a space-filling polytope~\cite{Conway} and is given in Section~\ref{Cells}.

Section~\ref{QLSGC}   applies the results of Section~\ref{QOIL} to a frequency-based probabilistic analysis  of \emph{quantized linear systems}, which is the main theme of the present report. A quantized linear system is defined in Section~\ref{DQLS} as a dynamical system in $\mZ^n$  with the transition operator $T := R\circ L$, where $L \in \mR^{n \x n}$ is a nonsingular matrix, and $R: \mR^n \to \mZ^n$ is a \emph{quantizer} which commutes with the additive group of translations of $\mZ^n$ and such that  $R^{-1}(0)$ is a Jordan measurable set. An example of a quantizer is the \emph{roundoff operator} $R_*$ with $R_*^{-1}(0) = [-1/2, 1/2)^n$. The quantized linear $(R,L)$-system is a model of the spatial discretization of a linear dynamical system in fixed-point arithmetic with no overflow. Section~\ref{ABFA} defines an \emph{associated algebra} $\cQ $ of $\cL$-quasiperiodic subsets of $\mZ^n$, with the matrix $\cL \in \mR^{\infty \x n}$ formed from nonzero integer powers of $L$. By splitting $\cL$ into two submatrices $\cL^-$ and $\cL^+$ formed from negative  and positive integer powers of $L$, respectively, we define a \emph{backward algebra}  $\cQ ^-$ of $\cL^-$-quasiperiodic subsets of $\mZ^n$ and a \emph{forward algebra} $\cQ ^+$ of $\cL^+$-quasiperiodic sets. It turns out that the associated algebra $\cQ $ is invariant with respect to both the transition operator $T$ and its set-valued inverse $T^{-1}$, and that the backward and forward algebras $\cQ ^-$ and $\cQ ^+$ are invariant with respect to $T$ and $T^{-1}$, respectively. Section~\ref{PFFA} shows that a $\mes_{n^2}$-almost
every nonsingular matrix $L \in \mR^{n \x n}$ is \emph{iteratively nonresonant} in the sense that its positive and negative integer powers form a nonresonant matrix
$\cL:= (L^k)_{k\in\mZ\setminus \{0\}}
\in \mR^{\infty \x n}.
$ Hence, for any such matrix $L$, the associated algebra $\cQ $ (including its subalgebras
$\cQ ^-$ and $\cQ ^+$) consists of frequency measurable subsets of the lattice $\mZ^n$. Moreover, it turns out that for an iteratively  nonresonant matrix $L$,   the corresponding
transition operator $T$  is not only $\cQ ^+$-measurable,
 but also preserves the frequency $\bF$ on the forward algebra $\cQ^+$,
thus allowing the quadruple $(\mZ^n, \cQ ^+, \bF,
T)$ to be regarded as a dynamical system with an invariant finitely additive probability measure $\bF$.
Section~\ref{IUDQE} is concerned with \emph{quantization errors} which are defined as the compositions $E_k := E \circ T^{k-1}$ for $k = 1, 2, 3, \ldots$,  where the first quantization error $E$ maps a point $x \in \mZ^n$ to the vector $E(x)  := Lx - T(x)
\in R^{-1}(0)$. It is shown that, under the condition that the matrix $L$ is iteratively nonresonant,
the quantization errors $E_1, E_2, E_3, \ldots$ are mutually independent,
uniformly distributed on the set $R^{-1}(0)$ and are $\cQ ^+$-measurable in the sense that for any positive integer $N$ and any
Jordan measurable set $G \subset (R^{-1}(0))^N$, the following set is an element of the forward algebra $\cQ^+$ and its frequency is $\mes_{Nn} G$:
\begin{equation}
    \label{intro1}
    \big\{
        x \in \mZ^n:\
        (
            E_k(x)
        )_{1 \< k \< N}
        \in
        G
    \big\}
    \in \cQ ^+.
\end{equation}
These properties of the quantization errors, which extend V.V.Voevodin's results of \cite{V_1967}, are closely related to the property that the transition operator $T$ is not only $\cQ ^+$-measurable and preserves the frequency $\bF$ on the forward algebra $\cQ ^+$,  but is also mixing. Hence, the quadruple $(\mZ^n, \cQ ^+, \bF, T)$ can be regarded as an ergodic dynamical system. Using the property that the quantization errors are uniformly distributed over $R^{-1}(0)$, we define a \emph{supporting} dynamical system in $\mR^n$ with the affine transition operator $T_*(x) := Lx - \mu$,  where $\mu := \int_{R^{-1}(0)} u\rd u$ is the  mean vector of the uniform distribution.  It is more convenient to study the deviation of trajectories of the quantized linear system from those of the supporting system whose phase portrait is qualitatively similar to that of the original linear system. Section~\ref{DDPSQLSS} shows that the vector $ (T_*^k(x) - T^k(x))_{1 \< k \< N} \in \mR^{Nn} $, which describes  the deviation of \emph{positive} semitrajectories of the quantized linear and the supporting systems during the first $N$ steps of their evolution from a common starting point $x \in \mZ^n$, is  an affine function of the quantization errors $E_1(x), \ldots, E_N(x)$, thus clarifying the role of the latter.  Moreover, the sequence $(T_*^k(x) - T^k(x))_{k \> 1}$, driven by the quantization errors, is  a homogeneous Markov chain on the probability space $(\mZ^n, \cQ ^+, \bF)$.
Therefore, in view of the results of Sections~\ref{IUDQE}
and~\ref{DDPSQLSS}, for any iteratively nonresonant matrix $L$, the events, which pertain to the deviation of
positive semitrajectories of the discretized system and the original linear
system in a finite number of steps of their evolution, are representable as
frequency measurable quasiperiodic subsets of the lattice in (\ref{intro1})   which
belong to the forward algebra $\cQ ^+$.

In contrast to the original linear system, the nonlinear transition operator $T$ is, in general,  neither surjective nor injective. There are ``holes'' $x \in \mZ^n$ in the lattice, which are not reachable for $T$, that is, $T^{-1}(x) = \emptyset$, and there also exist points $x \in \mZ^n$ for which the set $T^{-1}(x)$ consists of more than two elements. Therefore, the negative semitrajectory $(T^{-k}(x))_{k \> 1}$ of the quantized linear system  is a set-valued sequence with values from the class of finite (possibly, empty) subsets of the lattice. Section~\ref{DDNSQLSS} defines a \emph{compensating system}
as the quantized linear $(\wt{R}, L^{-1})$-system
with the transition operator $\wt{T}:= \wt{R}\circ L^{-1}$, where the quantizer
$\wt{R}$ is given by $\wt{R}(u) := R(u + (I_n
+ L^{-1}) \mu)$. This allows the preimage
$T^{-N}(x)$ to be represented as the Minkowski sum $\wt{T}^N(x) + \Sigma_N(x)$
of the singleton $\wt{T}^N(x)$ and a finite (possibly, empty)
set $\Sigma_N(x) \subset \mZ^n$. The set $\Sigma_N(x)$ is completely determined
by the quantization errors $\wt{E}_1(x), \ldots, \wt{E}_N(x)$
for the compensating system which are defined similarly to the
quantization errors of the quantized linear $(R,L)$-system.
Under the assumption that the matrix $L$ is iteratively nonresonant, the
quantization errors of the compensating system are  mutually
independent, uniformly distributed on the set
$\wt{R}^{-1}(0)$ and $\cQ ^-$-measurable in the sense that, for any
positive integer $N$ and any Jordan measurable set $H
\subset (\wt{R}^{-1}(0))^N$, the following set is an element of the backward algebra $\cQ^-$ and its frequency is $\mes_{Nn} H$:
\begin{equation}
\label{intro2}
    \big\{
        x \in \mZ^n:\
        (
            \wt{E}_k(x)
        )_{1 \< k \< N}
        \in
        H
    \big\}
    \in
    \cQ ^-.
\end{equation}
Due to these properties, the set-valued sequence $(T_*^{-k}(x) - T^{-k}(x))_{k \> 1}$ consists of the Min\-kow\-ski sums of the corresponding elements of the  sequence $(T_*^{-k}(x) - \wt{T}^k(x))_{k \> 1}$ and the set-valued sequence $(- \Sigma_k(x))_{k \> 1}$, with the last two sequences being homogeneous Markov chains on the probability space $(\mZ^n, \cQ ^-, \bF)$. As a corollary, a martingale property \cite{Shiryayev} is established for the sequence $(|\det L|^k \# T^{-k}(x))_{k \> 1}$ which describes the cardinality structure of the set-valued negative semitrajectory of the quantized linear system. In view of the results of Section~\ref{DDNSQLSS}, under the assumption that the matrix $L$ is iteratively nonresonant, any event, pertaining to the deviation of negative semitrajectories of the discretized and linear systems in a finite number of steps of their evolution, is representable as a frequency measurable quasiperiodic subset of the lattice in (\ref{intro2}) which belongs to the backward algebra $\cQ ^-$.

Therefore, Section~\ref{QLSGC} can be summarized as follows. The phase portrait of the original linear system with a nonsingular matrix $L \in \mR^{n \x n}$ is distorted when the system is replaced with the quantized linear $(R,L)$-system. This distortion can be described in terms of the deviations of positive and negative semitrajectories of the quantized linear system from those of the original linear system. For an iteratively nonresonant matrix $L$ (such matrices are typical), these deviations can be studied in the framework of the probability space $(\mZ^n, \cQ , \bF)$, since all the events, related to the deviations in a finite number of steps of system evolution, are representable as frequency measurable quasiperiodic  subsets of the lattice belonging to the associated algebra $\cQ$.

Section~\ref{QLSNC} considers a particular class of \emph{neutral} quantized linear $(R,L)$-systems on the lattice $\mZ^n$ of even dimension $n = 2r$, with a matrix $L$ being similar to an orthogonal matrix with a nondegenerate spectrum. This class of systems is specified in Section~\ref{CQLSBC}. The special structure of the matrix $L$ is used in Section~\ref{FCLTDPSQLSS} in order to establish a functional central limit theorem for the deviation of positive semitrajectories of the quantized linear system and the supporting system. Section~\ref{ERPR} applies the above results to quantitative analysis of the phase portrait of a two-dimensional neutral quantized linear system with which the problem on rounded-off planar rotations is concerned \cite{Blank,Diamond1,Kuznetsov}.  We compute the frequencies of some events pertaining to the phase portrait of the system and compare these theoretical predictions with a numerical experiment on a moderately large fragment of the lattice.

The results of this report (some of them were partially announced in \cite{Kozyakin}) can be used for rigorous justification  of ad hoc models \cite{Diamond2} for spatial discretizations of dynamical systems.  The apparatus of frequency measurable quasiperiodic sets, proposed in the present study,  is also applicable to quantitative analysis of aliasing structures \cite{Izmailov1,Izmailov2}.

\section{A probability structure on the integer lattice}
\label{PSIL}
\setcounter{equation}{0}

\subsection{Averageable functions}
\label{AF}

In what follows,  $\mZ^n$ denotes the integer lattice in the real
Euclidean space $\mR^n$ of $n$-dimensional column-vectors with the
standard inner product $x^{\rT} y$ and Euclidean norm $|x| :=
\sqrt{x^{\rT} x}$, where $(\cdot)^{\rT}$ is the transpose. Let $M$ denote
the linear space of bounded functions $f: \mZ^n \to \mR$ with
the uniform norm
$$
    \| f \| :=
    \sup_{x \in \mZ^n} |f(x)|.
$$
Denoting the class of nonempty finite subsets of the lattice $\mZ^n$ by
\begin{equation}
\label{totPi}
    \Pi
    :=
    \left\{
        P \subset \mZ^n:\
        0 <  \# P < +\infty
    \right\},
\end{equation}
with $\#(\cdot)$ the number of elements in a set, we define a functional $W: M \x \Pi \to \mR$  which maps a function $f\in M$ to its average value over a set $P \in \Pi$:
\begin{equation}
\label{WfP}
    W(f, P)
    :=
    \frac{1}{\# P}
    \sum_{x \in P} f(x).
\end{equation}
This quantity is also representable as the expectation $\bE f(\xi)$, where $\xi$ is a random vector with uniform probability distribution over the set $P$. For any vectors $a := (a_k)_{1 \< k \< n} \in \mZ^n$ and $\ell := (\ell_k)_{1 \< k \< n} \in \mN^n$, with $\mN$ the set of positive integers, let
\begin{equation}
\label{Pal}
    P_{a,\ell}
    :=
    \big\{
        x
        :=
        (x_k)_{1 \< k \< n} \in \mZ^n:\
        a_k \< x_k < a_k+\ell_k\
        {\rm for\ all}\
        1 \< k \< n
    \big\}
\end{equation}
denote a discrete parallelepiped which consists of  $\prod_{k=1}^{n} \ell_k$ points of the lattice. For any given $N \in \mN$, we also define a class
\begin{equation}
\label{calPN}
    \cP_N
    :=
    \left\{
        P_{a,\ell}:\
        a, \ell \in \mZ^n,\
        \min_{1 \< k \< n} \ell_k \> N
    \right\}
\end{equation}
of sufficiently large parallelepipeds  whose edge lengths are bounded below by $N$. These sets form a decreasing sequence: $\cP_1 \supset \cP_2 \supset \ldots$, with the class
\begin{equation}
\label{Pfilter}
    \bP
    :=
    \{
        \cP_N:\ N \in \mN
    \}
\end{equation}
specifying a topological filter base on the set $\Pi$ in (\ref{totPi}). For any function $f\in M$,  consider the lower and upper
limits of the average $W(f,P)$ in (\ref{WfP}) as the set $P$ tends to infinity in the sense of (\ref{Pfilter}):
\begin{align}
\label{lowerlim}
    \bA_*(f)
    & :=
    \liminf_{P \nearrow \infty}
    W(f,P)
    :=
    \sup_{N \> 1} \inf_{P \in \scriptsize{\cP}_N} W(f,P),\\
\label{upperlim}
    \bA^*(f)
    & :=
    \limsup_{P \nearrow \infty}
    W(f,P)
    :=
    \inf_{N \> 1} \sup_{P \in \scriptsize{\cP}_N} W(f,P).
\end{align}

\begin{definition}
\label{averfun}
For a given function $f: \mZ^n \to \mR$, the limits (\ref{lowerlim}) and (\ref{upperlim}) are called the lower and upper average values of the function, respectively. If these limits coincide, their common value is written as $\bA(f)$ and is called the  average value of $f$, and the function is called averageable.
\end{definition}
Therefore, the average value of an averageable function $f$ is the limit $
\bA(f) := \lim_{P \nearrow \infty} W(f,P) $. Constants are trivial
examples of averageable functions. In particular, the identically unit function
$\bone: \mZ^n \to \{1\}$ is averageable, and its average value is $\bA(\bone) = 1$. Less
trivial averageable functions will be studied in Section~\ref{QOIL}.

\begin{lemma}
\label{averfunprop}
\begin{itemize}
\item[{}]
\item[{\bf (a)}]
The functionals $\bA_*, \bA^*: M \to \mR$ are concave and convex on
the space $M$, are monotone with respect to the set $M^+$ of
nonnegative-valued bounded functions, are positively homogeneous and
satisfy a Lipschitz condition, \begin{equation}\label{Lipschitz}
    |\bA_*(f)-\bA_*(g)| ,\,
    |\bA^*(f)-\bA^*(g)|
    \< \bA^*(|f-g|)
    \< \|f-g\|;
\end{equation}
\item[{\bf (b)}]
the class of averageable functions  $\cM$ is a linear subspace of
$M$;
\item[{\bf (c)}]
the functional $\bA:  \cM \to \mR$ is linear, positive~\cite{Kras} with respect to the set $\cM^+$ of nonnegative-valued averageable functions and satisfies $ |\bA(f)| \< \bA^*(|f|) \< \|f\| $.
\end{itemize}
\end{lemma}

\begin{proof}  From (\ref{WfP}), it follows that for any given nonempty finite set $P
\subset \mZ^n$, the functional $W(\cdot, P): M \to \mR$ is linear,
positive with respect to the cone $M^+$ and satisfies $ |W(f,P)| \<
W(|f|, P) \< \|f\| $. These properties  imply the assertion of the
lemma. \end{proof}

The average value functional $\bA$ is extended in a standard fashion  to the complex
linear space $\cM + i \cM$ as $ \bA(f) := \bA(\Re f) + i \bA(\Im f)$.
Furthermore, if $f: \mZ^n \to X$ is a map with values in an
$r$-dimensional real linear space $X$, given by $ f := \sum_{k=1}^{r}
f_k b_k $, where $f_k: \mZ^n \to \mR$ are averageable functions and
$b_1, \ldots, b_r$ is a basis in $X$, then the corresponding
average value is defined to be $ \bA(f) := \sum_{k=1}^{r} \bA(f_k) b_k
$. This definition does not depend on a particular choice of the
basis in $X$.

\subsection{Frequency measurable sets}
\label{FMS}

For any set $A \subset \mZ^n$, let
\begin{align}
\label{lowerfreq}
    \bF_*(A) & :=
    \bA_*(\cI_A)  =
    \liminf_{P \nearrow \infty}
    \frac{\#\left(A \bigcap P\right)}{\# P},\\
\label{upperfreq}
    \bF^*(A) & :=  \bA^*(\cI_A)  =
    \limsup_{P \nearrow \infty}
    \frac{\#\left(A \bigcap P\right)}{\# P}
\end{align}
denote the lower and upper average values of the indicator function
$
    \cI_A(x) :=
    \left\{
        \begin{matrix}
        1 & {\rm for} & x \in A \\ 0 & {\rm for} &
        x \not\in A
        \end{matrix}
    \right.
$ of the set.
In (\ref{lowerfreq}) and (\ref{upperfreq}), use is made of the equality
\begin{equation}
\label{ratio}
    W
    (
        \cI_A, P
    ) =
    \frac{\# (A \bigcap P)}{\# P}
\end{equation}
for the relative fraction of points of the set $A$ in (a discrete parallelepiped) $P$,
which follows from (\ref{WfP}). Note that
$$
    0 \< \bF_*(A) =
    1- \bF^* (\mZ^n \setminus A)
    \< \bF^*(A) \< 1.
$$

\begin{definition}\label{averset}
For a given set $A \subset \mZ^n$, the quantities $\bF_*(A)$ and $\bF^*(A)$ in (\ref{lowerfreq}) and (\ref{upperfreq}) are called the lower and upper frequencies of the set, respectively. If they coincide, then their common value is denoted by $\bF(A)$ and is called the frequency of $A$, and the set is called frequency measurable.
\end{definition}
Therefore, the frequency measurability of a set $A$ is equivalent to the averagability of its indicator function $\cI_A$. The frequency $\bF(A)$ of a frequency measurable set $A$ is the limit of the ratio (\ref{ratio}) in the sense of (\ref{Pfilter})--(\ref{upperlim}):
$$
    \bF(A)
    :=
    \lim_{P \nearrow \infty}
    \frac
    {\#(A \bigcap P)}{\# P}.
$$
That is, $\bF(A)$ is the limit fraction of points of a given set $A$ in unboundedly increasing rectangular fragments of the lattice.
A straightforward example of a frequency measurable set is $\mZ^n$, with $\bF(\mZ^n) = 1$. Also note that any finite subset of $\mZ^n$ has zero frequency. A class of nontrivial frequency measurable sets, relevant to the context of spatial discretizations of  dynamical systems, will be considered in Section~\ref{QOIL}.

\begin{lemma}\label{aversetprop}
\begin{itemize}
\item[{}]
\item[{\bf (a)}]
For any disjoint sets $A, B \subset \mZ^n$, their lower and upper frequencies satisfy the inequalities
$$
    \bF_*(A \bigcup B)
    \> \bF_*(A)+\bF_*(B),
    \qquad
    \bF^*(A \bigcup B)
    \< \bF^*(A)+\bF^*(B).
$$  Furthermore, $
    \bF_*(A) \< \bF_*(B)$ and
$
    \bF^*(A) \< \bF^*(B)
$ hold for any sets $A \subset B \subset \mZ^n$;
\item[{\bf (b)}]
the class $\cS$ of frequency measurable subsets of $\mZ^n$
contains the lattice and is closed under the union of disjoint sets
and  complement of a set to the lattice;
\item[{\bf (c)}]
the frequency functional  $\bF: \cS \to [0,1]$ is a finitely additive probability
measure,  which satisfies $\bF(\mZ^n) = 1$ and
$
    \bF(A \bigcup B) =
    \bF(A)  + \bF(B)
$ for any disjoint sets $A, B \in \cS$;
\item[{\bf (d)}]
if the symmetric difference $A \triangle B$ of sets  $A, B \subset
\mZ^n$ is frequency measurable with zero frequency $\bF(A \triangle B) = 0$, then
$A$ is frequency measurable if and only if so is $B$, with $\bF(A) = \bF(B)$ in the
case of frequency measurability.
\end{itemize}
\end{lemma}
\begin{proof}  The assertion (a) of the lemma follows from the concavity, convexity, monotonicity and positive homogeneity of the functionals $\bA_*$ and $\bA^*$ (see the assertion (a) of Lemma~\ref{averfunprop}) and from the equality \begin{equation}\label{sumofind} \cI_{A \bigcup  B} = \cI_A + \cI_B \end{equation} for any disjoint sets $A, B \subset \mZ^n$. The assertions (b) and (c) of the lemma follow from the linearity of the space $\cM$ and of the functional $\bA$ (see the assertions (b) and (c) of Lemma~\ref{averfunprop}), and from the equality (\ref{sumofind}) for disjoint sets. The assertion (d) of the lemma is established by using the inequalities
$$
    |\bF_*(A) -\bF_*(B)|, \ |\bF^*(A) -\bF^*(B)| \< \bF^*(A \triangle B)
$$ which follow from (\ref{Lipschitz}) and from the identity $\cI_{A \triangle B} = |\cI_A - \cI_B|$, thus completing the proof of the lemma.  \end{proof}

The assertions (b) and (c) of Lemma~\ref{aversetprop} show that the triple $(\mZ^n,\cS,\bF)$ can be regarded as a probability space, with the average value functional $\bA: \cM \to \mR$ playing the role of expectation.  In view of Lemma~\ref{aversetprop}(d), this probability space is complete. However, this space does not satisfy the axiomatics of A.N.Kolmogorov \cite{Shiryayev} since $\cS$ is not a $\sigma$-algebra and $\bF$ is not a countably additive measure. Despite this pathology, frequency measurable sets and the frequency functional on them possess geometric properties  which correspond to those of Lebesgue measurable sets and the Lebesgue measure in $\mR^n$.

 \begin{theorem}\label{aversetprop1}
For any set $A \in \cS $, any nonsingular matrix $F \in \mZ^{n \x n}$ and any vector $z \in \mZ^n$, the set $FA + z$ is also frequency measurable, with
 \begin{equation}\label{FFAz}
    \bF(FA+z) =
    \frac{\bF(A)}{|\det F|}.
 \end{equation}
\end{theorem}
\begin{proof}
We will first verify the equality (\ref{FFAz}) for the identity matrix $F = I_n$ of order $n$. This is equivalent to that the translated set $A + z$ inherits frequency measurability  from $A \in \cS$, with
\begin{equation}\label{FAz}
    \bF(A+z) = \bF(A)
\end{equation}
for any $z \in \mZ^n$. Note that $ P_{a,\ell} + z = P_{a+z, \ell} $ for any $a \in \mZ^n$ and $\ell \in \mN^n$, where $P_{a,\ell}$ is the discrete parallelepiped defined by (\ref{Pal}). Hence, each of the classes $\cP_N$ in (\ref{calPN}) is invariant under translations of the constituent parallelepipeds:
$$
    \cP_N + z := \{ P + z:\ P \in \cP_N\} = \cP_N.
$$
In combination with the identities $ \#((A+z) \bigcap P) = \#(A \bigcap (P-z))$ and $\# (P-z) = \# P$, the translation invariance implies that
\begin{align*}
    \inf_{P \in \scriptsize{\cP}_N}
    \frac{\#((A+z) \bigcap P)}{\# P} & =
    \inf_{P \in \scriptsize{\cP}_N}
    \frac{\#(A \bigcap P)}{\# P},\\
    \sup_{P \in \scriptsize{\cP}_N}
    \frac{\#((A+z) \bigcap P)}{\# P} & =
    \sup_{P \in \scriptsize{\cP}_N}
    \frac{\#(A \bigcap P)}{\# P}.
\end{align*}
From the arbitrariness of $N \in \mN$ in the last two equalities and
from the  assumption that $A \in \cS$, it follows that $ \bF_*(A+z) =
\bF^*(A+z) = \bF(A) $, which implies (\ref{FAz}).
Therefore, it now remains to prove the property (\ref{FFAz}) for $z = 0$.
More precisely, we will show that the set $FA$ is frequency measurable and
\begin{equation}
    \label{FFA}
    \bF(FA)
    =
    \frac{\bF(A)}{|\det F|}
\end{equation}
for any $A \in \cS$ and nonsingular $F \in \mZ^{n \x n}$. To this end, for every $\nu \in \mN$ and $\alpha := (\alpha_k)_{1 \< k \< n} \in \mZ^n$, consider a discrete cube
\begin{equation}
    \label{Palfnu}
    P_{\alpha}^{(\nu)}
    :=
    \left\{
        x :=(x_k)_{1 \< k \< n}\in \mZ^n:\
        \nu \alpha_k \< x_k < \nu(\alpha_k +1)\
        {\rm for\ all}\
        1 \< k \< n
    \right\}
\end{equation}
which consists of $\nu^n$ points. For any given $\nu$, the cubes $P_{\alpha}^{(\nu)}$ belong to the class $\cP_{\nu}$ in (\ref{calPN}) and form a partition of $\mZ^n$. Therefore, the sets $ A_{\nu, \alpha} = A \bigcap P_{\alpha}^{(\nu)} $,\ $ \alpha \in \mZ^n $, form a partition of the set $A \subset \mZ^n$, and for any nonempty finite set $P \subset \mZ^n$,
\begin{equation}\label{inout}
    F \bigcup_{\alpha \in \Phi_{\nu}(P)} A_{\nu,\alpha}
    \subset
    (FA) \bigcap P
    \subset
    F \bigcup_{\alpha \in \Psi_{\nu}(P)} A_{\nu, \alpha},
\end{equation}
where
$$
    \Phi_{\nu}(P)
    :=
    \left\{
        \alpha \in \mZ^n:\
        F P_{\alpha}^{(\nu)} \subset P
    \right\}
    \subset
    \Psi_{\nu}(P)
    :=
    \left\{
        \alpha \in \mZ^n:\
        (F P_{\alpha}^{(\nu)})
        \bigcap P \ne \emptyset
    \right\}.
$$
In view of the assumption that $\det F \ne 0$, it follows from (\ref{inout}) that
\begin{equation}\label{auxineq}
    \nu^n
    \# \Phi_{\nu}(P) \inf_{Q \in \scriptsize{\cP}_{\nu}} \frac{\# (A
    \bigcap Q)}{\# Q}
    \<
    \# ((F A) \bigcap P)
    \<
    \nu^n
    \#
    \Psi_{\nu}(P) \sup_{Q \in \scriptsize{\cP}_\nu}
    \frac{\# (A \bigcap Q)}{\# Q}.
\end{equation}
We will now prove  that the limits
\begin{equation}\label{auxlimits}
    \lim_{P \nearrow \infty}
    \frac{\nu^n \# \Phi_{\nu}(P)}{\# P} =
    \lim_{P \nearrow \infty}
    \frac{\nu^n \# \Psi_{\nu}(P)}{\# P} =
    \frac{1}{|\det F|}
\end{equation}
in the sense of (\ref{Pfilter})--(\ref{upperlim}) hold for any $\nu \in \mN$. To this end, consider the Minkowski sums $ \wt{P}_{a,\ell} := P_{a,\ell} + V $ and $ \wt{P}_{\alpha}^{(\nu)} := P_{\alpha}^{(\nu)} + V $ of the discrete parallelepipeds $P_{a,\ell}, P_{\alpha}^{(\nu)} \subset \mZ^n$ with the cube $ V := [-1/2,1/2)^n \subset \mR^n$. By using the matrix norm $ \sn F \sn := \max_{1 \< j \< n} \sum_{k=1}^{n} |f_{jk}| $, which is  induced by the Chebyshev norm in $\mR^n$, with $f_{jk}$ denoting the entries of the matrix $F$, it follows that
\begin{equation}\label{inc1}
    \wt{P}_{a,\ell}\ \ {}^{\underline{*}}\ \ ((2\nu-1)
    \sn
    F \sn +1)  V )
    \subset
    \bigcup_{\alpha \in \Phi_{\nu}(P_{a,\ell})}
    ( F \wt{P}_{\alpha}^{(\nu)})
\end{equation}
and
\begin{equation}\label{inc2}
     \bigcup_{\alpha \in \Psi_{\nu}(P_{a,\ell})}
    ( F \wt{P}_{\alpha}^{(\nu)} ) \subset
    \wt{P}_{a,\ell} + ((2\nu -1) \sn F \sn - 1) V
\end{equation}
where
$
    G \ ^{\underline{*}} \  H =
    \{
        u \in \mR^n:\
        u + H \subset G
    \}
$ denotes  the Minkowski subtraction of sets $G, H \subset \mR^n$. The
inclusions (\ref{inc1}) and (\ref{inc2}) imply that
\begin{align*}
    \inf_{P \in \cP_N}
    \frac{\nu^n \# \Phi_{\nu}(P)}{\# P}
    & \>
    \frac{1}{|\det F|} \left(1 - \frac{(2\nu-1) \sn
    F \sn +1}{N}\right)^n,\\
    \sup_{P \in \cP_N}
    \frac{\nu^n \# \Psi_{\nu}(P)}{\# P} & \<  \frac{1}{|\det F|}
    \left(1 + \frac{(2\nu -1) \sn F \sn - 1}{N}\right)^n.
\end{align*}
These inequalities lead to the limits in (\ref{auxlimits}). Now, from (\ref{auxineq}) and (\ref{auxlimits}), it follows that for any $\nu \in \mN$,
$$
    \frac{1}{|\det F|} \inf_{Q \in \scriptsize{\cP}_\nu}
    \frac{\# (A \bigcap
    Q)}{\# Q}
     \< \bF_*(FA) \< \bF^*(FA) \<
    \frac{1}{|\det F|} \sup_{Q \in \scriptsize{\cP}_\nu}
    \frac{\# (A \bigcap
    Q)}{\# Q},
$$
and hence,
$$
    \frac{\bF_*(A)}{|\det F|}  \< \bF_*(FA) \<
\bF^*(FA) \< \frac{\bF^*(A)}{|\det F|}.
$$
The last inequalities and the assumption $A \in \cS$ imply (\ref{FFA}), thereby completing the proof of the theorem. \end{proof}

A geometric interpretation of (\ref{FFAz}) is that the frequency functional $\bF$ is translation invariant, and a linear transformation of a frequency measurable set $ A$ with an integer matrix $F$, satisfying $|\det F|>1$, leads to a ``sparser'' set $FA$.

\subsection{Distributed  maps}
\label{DM}

Suppose $X$  is a metric space equipped with the Borel $\sigma$-algebra
$\cB$, and let $D$ be a countably additive probability measure on
the measurable  space $(X, \cB)$. Recall that a map $g: X \to Y$
with values in another metric space $Y$ is said to be $D$-continuous
\cite{Billingsley} if the discontinuity set of $g$ has zero
$D$-measure. Accordingly, a set $B \in \cB$ is called $D$-continuous
if its boundary $\partial B$ (which is the set of discontinuity points of the indicator function $\cI_B$  of the set $B$) satisfies $D(\partial B) = 0$.

\begin{definition}
\label{distribution}
Suppose $g:\mZ^n \to X$ is a map with values in a metric space $X$, and  $D$ is a countably additive probability measure on $(X,\cB)$. The map $g$ is said to be $D$-distributed if the preimage
$$
    g^{-1}(B) :=
    \left\{
        x \in \mZ^n:\ g(x) \in B
    \right\}
$$
of any $D$-continuous set $B \in \cB$ under the map $g$ is frequency measurable and its frequency satisfies
\begin{equation}
\label{calFg-1B}
    \bF
    (
        g^{-1}(B)
    ) =
    D(B).
\end{equation}
In this case, the algebra
$$
     \cS_g
     :=
     \left\{
        g^{-1}(B):\
        B \in \cB,\
        D(\partial B) = 0
     \right\},
$$
which consists of frequency measurable sets, is referred to as the algebra generated  by the map $g$.
\end{definition}

\begin{lemma}
\label{distrcriteria}
Suppose $g: \mZ^n \to X$ is a map with values in a metric space $X$, and let $D$ be a countably additive probability measure on $(X,\cB)$. Then the following properties are equivalent:
\begin{itemize}
\item[{\bf (a)}]
    for any bounded continuous function $f: X \to \mR$, the composition
    $f \circ g: \mZ^n \to \mR$  is averageable, with
    \begin{equation}
    \label{AfgD}
        \bA(f\circ g)
        =
        \int_X
        f(x)
        D(\rd x);
    \end{equation}
\item[{\bf (b)}]
    for any bounded $D$-continuous function $f:  X \to \mR$, the
    composition $f \circ g$  is averageable  and (\ref{AfgD}) holds;
\item[{\bf (c)}]
    the map $g$ is $D$-distributed.
\end{itemize}
\end{lemma}
\begin{proof}  For any nonempty finite set $P \subset \mZ^n$, consider a
countably additive probability measure $U_P$ on $(\mZ^n, 2^{\mZ^n})$
defined by
$$
    U_P(A)
    :=
    W(\cI_A, P)
    =
    \frac
    {\# (A \bigcap P)}
    {\# P},
$$
where use is made of (\ref{WfP}).
Any such set $P$ and any map $g:  \mZ^n \to X$ generate a countably
additive probability measure  $U_P \circ g^{-1}$ on $(X, \cB)$. The
expectation of a Borel measurable function $f: X \to \mR$, interpreted
as a random variable on the probability space $(X,\cB,U_P \circ g^{-1})$, takes the form
\begin{equation}
\label{expUPinvg}
    \int_X
    f(x)
    (
        U_P \circ g^{-1}
    )(\rd x)
    =
    W(f \circ g, P).
\end{equation}
In particular, application of the equality (\ref{expUPinvg})  to the indicator function $f= \cI_B$ of a set $B \subset
X$ yields
\begin{equation}
\label{UPinvgB}
    (
        U_P \circ g^{-1}
    )(B) =
    \frac{
    \#
    (
        g^{-1}(B) \bigcap P
    )}{\# P}.
\end{equation}
Now, we assume the map $g$ to be fixed but otherwise arbitrary, and, similarly to (\ref{calPN}) and
(\ref{Pfilter}), define a  topological filter base
\begin{equation}\label{Dfilter}
    \bD
    :=
    \{
        \cD_N:\
        N \in \mN
    \},
    \qquad
    \cD_N := \{ U_P
\circ g^{-1}:\ P \in \cP_N\},
\end{equation}
on the class of countably additive probability measures on $(X, \cB)$. The existence of a limit $D$ in the sense of (\ref{Dfilter}) in the topology  of weak convergence \cite{Billingsley} of probability measures on $(X, \cB)$ means that
\begin{itemize}
    \item[{\bf (a')}]
    for any bounded continuous function $f: X \to \mR$,
    \begin{equation}\label{limexpUPinvg}
        \lim_{P \nearrow \infty}
        \int_{X}
        f(x)
        (
            U_P \circ g^{-1}
        )(\rd x) =
        \int_{X}
        f(x)
        D(\rd x).
    \end{equation}
\end{itemize}
By the well-known criteria for the weak convergence, the
last property is equivalent to each of the following ones:
\begin{itemize}
\item[{\bf (b')}]
the convergence (\ref{limexpUPinvg}) holds for any bounded
$D$-continuous function $f: X \to \mR$;
\item[{\bf (c')}]
for any $D$-continuous set $B \in \cB$,
\begin{equation}\label{limUPinvgB}
    \lim_{P \nearrow \infty}
    (
        U_P \circ g^{-1}
    )(B) =
    D(B).
\end{equation}
\end{itemize}
The properties (a')--(c') are equivalent to the corresponding
properties (a)--(c) stated in the lemma. Indeed, in view of
(\ref{expUPinvg}), the left-hand side of (\ref{limexpUPinvg}) is
$\bA(f \circ g)$, and by (\ref{UPinvgB}), the left-hand side of
(\ref{limUPinvgB}) is $\bF(g^{-1}(B))$. Hence, the equivalence of
the properties (a')--(c') implies that (a)--(c) are also equivalent to each other, thereby
completing the proof of the lemma. \end{proof}

From the proof of Lemma~\ref{distrcriteria}, it follows that for any
map $g:  \mZ^n \to X$ with values in a metric space $X$, there
exists at most one countably additive probability measure $D$ on
$(X,\cB)$ such that $g$ is $D$-distributed.

\begin{lemma}\label{mapcomposition}
Suppose $X_1$ and $X_2$ are two metric spaces with Borel
$\sigma$-algebras $\cB_1$ and $\cB_2$, respectively. Also, let $g_1: \mZ^n \to X_1$ and $g_2: X_1 \to X_2$ be two maps which satisfy the following
conditions:
\begin{itemize}
\item[{\bf (a) }]
$g_1$ is $D_1$-distributed, with $D_1$ a countably additive
probability measure on $(X_1,\cB_1)$;
\item[{\bf (b) }]
$g_2$ is $D_1$-continuous.
\end{itemize}
Then the composition of the maps $g_2 \circ g_1: \mZ^n \to X_2$ is
$D_2$-distributed, with the probability measure $D_2$ on $(X_2, \cB_2)$
given by
    \begin{equation}\label{D1invg2B}
        D_2(B) =
        D_1
        (
            g_2^{-1}(B)
        ),
        \qquad
        B \in \cB_2,
    \end{equation}
and the generated algebras satisfy the inclusion
    \begin{equation}\label{calSg2g1}
        \cS_{g_2 \circ g_1}
        \subset
        \cS_{g_1}.
    \end{equation}
\end{lemma}

\begin{proof}
Let $\phi: X_2 \to \mR$ be a bounded continuous function. Then, by the assumption (b) of the lemma, the function
    \begin{equation}\label{phicircg2}
        f
        :=
        \phi \circ g_2 : X_1 \to \mR
    \end{equation}
is bounded and $D_1$-continuous. Hence, by the assumption (a) of the lemma and by the criterion (b) of Lemma~\ref{distrcriteria}, the function $f \circ g_1:  \mZ^n \to \mR$ is averageable and
\begin{equation}\label{Afcircg1}
    \bA
    (
        f \circ g_1
    ) =
    \int_{X_1}
    f(x)
    D_1(\rd x) =
    \int_{X_2}
    \phi(y)
    D_2(\rd y),
\end{equation}
where $D_2$ is the probability measure given by  (\ref{D1invg2B}). Since the function $\phi$ is otherwise arbitrary,  then, in view of the criterion (a) of Lemma~\ref{distrcriteria},  it follows from (\ref{phicircg2}) and (\ref{Afcircg1}) that the map $g_2 \circ g_1$ is $D_2$-distributed. It now remains to prove the inclusion (\ref{calSg2g1}). From the assumption (b) of the lemma and from (\ref{D1invg2B}), it follows that $ D_2 ( \partial B ) = D_1(\partial g_2^{-1}(B)) $ for any $B \in \cB_2$, and hence, $g_2^{-1}(B)$ is a $D_1$-continuous Borel subset of $X_1$ for any $D_2$-continuous set $B \in \cB_2$. Therefore,
    $$
         \{
            g_1^{-1}(g_2^{-1}(B)):\
            B \in \cB_2,\
            D_2(\partial B) = 0
         \}
         \subset
         \{
            g_1^{-1}(B):\
            B \in \cB_1,\
            D_1(\partial B) = 0
         \}.
    $$
Since the left-hand side of this inclusion is the algebra $\cS _{g_2 \circ g_1}$, whilst
the right-hand side is $\cS _{g_1}$, the inclusion (\ref{calSg2g1}) follows. \end{proof}

\section{Quasiperiodic objects on the integer lattice
\label{QOIL}}
\setcounter{equation}{0}

\subsection
{Cells \label{Cells}}

In what follows, we will use an extension of the notion of a space-filling polytope \cite{Conway} given below.

\begin{definition}
\label{cell} A set $V \subset \mR^m$ is called a cell if its translations $V+z$, considered for all $z \in \mZ^m$, form a partition of $\mR^m$.
\end{definition}
An example of a cell in $\mR^m$ is provided by the cube $[0,1)^m$. Recall that any unimodular matrix (that is, a square integer matrix with determinant $\pm 1$)  describes a linear bijection of the integer lattice.

\begin{lemma}\label{cellprop1}
\begin{itemize}
\item[{}]

\item[{\bf (a)}]
For any  unimodular matrix  $F \in \mZ^{m\x m}$, any $u \in
\mR^m$ and any cell  $V \subset \mR^m$, the set $F V + u$ is a cell
in $\mR^m$;

\item[{\bf (b)}]
a set $V \subset \mR^m$ is a cell if and only if there exists a
partition
    \begin{equation}\label{partOmega}
         \{
            \Omega_z:\
            z \in \mZ^m
        \}
    \end{equation}
    of $[0,1)^m$ satisfying
    \begin{equation}\label{cellV}
        V =
        \bigcup_{z \in \mZ^m}
        (
            \Omega_z +z
        ).
    \end{equation}
Moreover, such a partition is unique;

\item[{\bf (c)}]
any Lebesgue measurable cell $V \subset \mR^m$ has a unit measure, $ \mes_m V =
1$;

\item[{\bf (d)}]
for any cells $V_1 \subset \mR^{m_1}$ and $V_2 \subset \mR^{m_2}$,
the set $V_1 \x V_2$ is a cell in $\mR^{m_1 + m_2}$.
\end{itemize}
\end{lemma}

\begin{proof}
To establish the assertion (a), note that any unimodular
matrix $F \in \mZ^{m\x m}$ determines a linear bijection of
$\mZ^m$. Hence, for any cell $V \subset \mR^m$, any $u \in \mR^m$
and any $z \in \mZ^m \setminus \{0\}$, the set $G:= FV+u$ satisfies
\begin{align*}
    G + \mZ^m
    & =
    u + F (V + F^{-1} \mZ^m)
    =
    \mR^m,\\
    G
    \bigcap
    (G +z )
    & =
    u +
    F
    \left(V \bigcap (V +
    F^{-1} z)\right) =
\emptyset
\end{align*}
 which means that $G$ is  also a cell. In order to prove the assertion (b), we associate with a given but otherwise arbitrary set $V \subset \mR^m$
the sets
\begin{equation}
    \label{concreteOmegaz}
        \Omega_z^0
        :=
        [0,1)^m \bigcap (V - z),
        \qquad
        z \in \mZ^m,
\end{equation}
in terms of which the set $V$ is representable by (\ref{cellV}), that is, $V = \bigcup_{z\in \mZ^m} (\Omega_z^0 + z)$.  Indeed, for any $z \in \mZ^m$,
\begin{equation}
\label{OmegazV}
    \Omega_z^0 + z =
        (
            [0,1)^m + z
        )
        \bigcap
        V,
\end{equation}
and hence,
    $$
        \bigcup_{z \in \mZ^m}
        (\Omega_z^0 + z) =
        \left(
            \bigcup_{z \in \mZ^n}
            (
                [0,1)^m + z
            )
        \right)
        \bigcap
        V = V.
    $$
On the other hand, the sets $\Omega_z^0$ in (\ref{concreteOmegaz}) partition
the cube $[0,1)^m$ if and only if $V$ is a cell in $\mR^m$. This proves the first part of the
assertion (b). We will now prove that the sets $\Omega_z^0$ in (\ref{concreteOmegaz}) provide a
unique partition of $[0,1)^m$ such that a given cell $V \subset \mR^m$
is represented by (\ref{cellV}). To this end, let (\ref{partOmega}) describe  an arbitrary  partition of the cube $[0,1)^m$
satisfying (\ref{cellV}). Then for any $z \in \mZ^m$,
    $$
        \Omega_z + z =
        (
            [0,1)^m + z
        )
        \bigcap V,
    $$
which, in view of (\ref{concreteOmegaz}), implies  that $\Omega_z = \Omega_z^0$, thus proving the uniqueness of the partition.   The assertion (c) follows from the relations
    $$
        \mes_m V =
        \sum_{z \in \mZ^m}
        \mes_m (\Omega_z^0 + z) =
        \sum_{z \in \mZ^m}
        \mes_m \Omega_z^0 =
        \mes_m [0,1)^m = 1
    $$
which hold for any Lebesgue measurable cell $V \subset \mR^m$ and are based on (\ref{concreteOmegaz}), (\ref{OmegazV}) and the translation invariance of the  Lebesgue measure. The assertion (d) of the lemma follows from the identities $\mZ^{m_1+m_2} = \mZ^{m_1} \x \mZ^{m_2}$ and $\mR^{m_1+m_2} = \mR^{m_1} \x \mR^{m_2}$. \end{proof}

Lemma~\ref{cellprop1} shows that there exist more complicated cells in $\mR^m$  than the cube $[0,1)^m$. This is illustrated by Fig.~\ref{fig1} which provides an example of such a cell in $\mR^2$.
 \begin{figure}[h]
\begin{center}
\unitlength=0.5mm
\begin{picture}(120.00,130.00)
\put(36,66){\rule{18mm}{6mm}}
\put(48,54){\rule{6mm}{18mm}}
\multiput(24,46)(48,0){2}{\rule{6mm}{2mm}}%
\multiput(24,94)(48,0){2}{\rule{6mm}{2mm}}%
\multiput(28,42)(48,0){2}{\rule{2mm}{6mm}}%
\multiput(28,90)(48,0){2}{\rule{2mm}{6mm}}%
\multiput(0,18)(8,0){2}{\rule{2mm}{2mm}}%
\multiput(0,26)(8,0){2}{\rule{2mm}{2mm}}%
\multiput(0,114)(8,0){2}{\rule{2mm}{2mm}}%
\multiput(0,122)(8,0){2}{\rule{2mm}{2mm}}%
\multiput(96,114)(8,0){2}{\rule{2mm}{2mm}}%
\multiput(96,122)(8,0){2}{\rule{2mm}{2mm}}%
\multiput(96,18)(8,0){2}{\rule{2mm}{2mm}}%
\multiput(96,26)(8,0){2}{\rule{2mm}{2mm}}%

\multiput(0,18)(36,0){4}{\line(0,1){108.00}}%
\multiput(0,18)(0,36){4}{\line(1,0){108.00}}%
\put(-4,54){\vector(1,0){120.00}}
\put(36,14.5){\vector(0,1){120.00}}

\put(0,12){\makebox(0,0)[cc]{$-1$}}
\put(36,12){\makebox(0,0)[cc]{$0$}}
\put(72,12){\makebox(0,0)[cc]{$1$}}
\put(108,12){\makebox(0,0)[cc]{$2$}}
\put(-6,18){\makebox(0,0)[cc]{$-1$}}
\put(-6,54){\makebox(0,0)[cc]{$0$}}
\put(-6,90){\makebox(0,0)[cc]{$1$}}
\put(-6,126){\makebox(0,0)[cc]{$2$}}
\end{picture}
\end{center}\vskip-8mm
\caption{An example of a cell in $\mR^2$ which is obtained by
cutting the square $[0,1)^2$ into pieces and translating them by two-dimensional
integer vectors.} \label{fig1}
\end{figure}
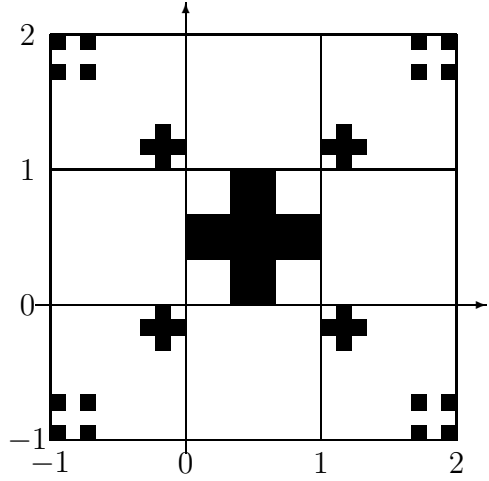

\begin{lemma}\label{cellprop2}
\begin{itemize}
\item[{}]
\item[{\bf (a)}]
Let $f: \mR^m \to \mR$ be a locally integrable function, unit
periodic with respect to its $m$ variables:
    \begin{equation}\label{fuz}
        f(u+z) = f(u)
        \quad
        {\rm for\ all}\
        u \in \mR^m,\ z \in \mZ^m.
    \end{equation}
Then $\int_{V} f(u) \rd u = \int_{[0,1]^m} f(u) \rd u$ for any Lebesgue
measurable cell $V \subset \mR^m$;

\item[{\bf (b)}]
any translation invariant set $G \subset \mR^m$ (under the group of translations of $\mZ^m$ in the sense that $G
+ \mZ^m = G$) is representable as $G = (V \bigcap G) + \mZ^m$ where
$V$ is an arbitrary cell in $\mR^m$;

\item[{\bf (c)}]
Let $G \subset \mR^m$ be a translation invariant and Lebesgue locally
measurable set. Then, for any Lebesgue measurable cell $V \subset \mR^m$,
$$
    \mes_m
    \left(V \bigcap G\right)
    =
    \mes_m
    \left(
        [0,1]^m \bigcap G
    \right).
$$
\end{itemize}
\end{lemma}
\begin{proof}
The assertion (a) of the lemma is proved by the following equalities for any unit periodic and locally integrable function $f: \mR^m \to \mR$:
    $$
        \int_{V} f(u) \rd u =
        \sum_{z \in \mZ^m}
        \int_{\Omega_z+z}
        f(u) \rd u =
        \sum_{z \in \mZ^m}
        \int_{\Omega_z} f(u) \rd u =
        \int_{[0,1]^m}
        f(u) \rd u,
    $$
which employ a partition (\ref{partOmega}) from Lemma~\ref{cellprop1}(b) for a Lebesgue measurable cell $V \subset \mR^m$ and
the translation invariance of the Lebesgue measure. The assertion (b) of the lemma is established by the relations
    $$
        \left(V \bigcap G\right) + \mZ^m =
        \bigcup_{z \in \mZ^m}
        \left((V+z) \bigcap (G+z) \right) =
        (V+\mZ^m) \bigcap G = G
    $$
which hold for any cell $V \subset \mR^m$ and any translation invariant set $G \subset \mR^m$. The assertion (c) of the lemma follows from the assertion (a) and from the property that the indicator function $\cI_G$ of any translation invariant set $G \subset \mR^m$ is unit periodic with respect to   its $m$ variables. \end{proof}

\subsection
{Quasiperiodic  sets}
\label{QS}

With any positive integer $m \in \mN$,  any set $G \subset \mR^m$ and any matrix $\Lambda \in \mR^{m
\x n}$, we associate a set
\begin{equation}\label{QmGL}
    Q_m(G,\Lambda)
    :=
    \left\{
        x \in \mZ^n:\
        \Lambda x
        \in
        G + \mZ^m
    \right\}
\end{equation}
The following lemma provides a useful formalism for set theoretic operations with such subsets of the lattice $\mZ^n$ which will  play an important role in the subsequent sections.

\begin{lemma}\label{formalizm}
The subsets of $\mZ^n$, defined by (\ref{QmGL}), possess the following properties:
\begin{itemize}
\item[{\bf (a) }]
for any $m \in \mN$, any $\Lambda \in \mR^{m \x n}$, any $G
\subset \mR^m$ and any cell $V \subset \mR^m$,
$$
    Q_m(G,\Lambda) =
    Q_m
    \left(
        V
        \bigcap
        (
            G+\mZ^m
        ),
        \Lambda
    \right);
$$

\item[{\bf (b) }]
for any $m \in \mN$, any $\Lambda \in \mR^{m \x n}$, any cell $V
\subset \mR^m$ and any $G, H \subset V$,
\begin{align*}
    Q_m(V, \Lambda)    & =     \mZ^n,\\
    \mZ^n \setminus Q_m(G,\Lambda)
    & =
    Q_m(V\setminus G,\Lambda),\\
    Q_m(G,\Lambda) \bigcap Q_m(H,\Lambda)
    & = Q_m\left(G \bigcap H,\Lambda\right),\\
    Q_m(G,\Lambda) \bigcup Q_m(H,\Lambda)
    & =
    Q_m\left(G \bigcup H, \Lambda\right);
\end{align*}

\item[{\bf (c) }]
for any $m \in \mN$, $G \subset \mR^{n+m}$ and $\Lambda \in \mR^{m
\x n}$, \begin{equation}\label{QnmGIL}
    Q_{n+m}
    \left(
        G,
        \begin{bmatrix}
        I_n \\
        \Lambda
        \end{bmatrix}
    \right) =
    Q_m(H, \Lambda),
\end{equation}
where
\begin{equation}\label{QnmGILH}
    H :=
    \left\{
        v \in \mR^m:\
        {\rm there\ exists}\
        u \in \mZ^n\
        {\rm such\ that}\
        \begin{bmatrix}
        u \\
        v
        \end{bmatrix}\in
        G
    \right\};
\end{equation}
\item[{\bf (d) }]
for any $m \in \mN$, any $G \subset \mR^m$, any $\Lambda \in \mR^{m
\x n}$ and any unimodular $F \in \mZ^{m \x m}$,
$$
    Q_m(FG, F\Lambda) = Q_m(G,\Lambda);
$$
\item[{\bf (e) }]
for any $m_1, m_2 \in \mN$, $G_1 \subset \mR^{m_1}$,\, $G_2 \subset
\mR^{m_2}$, $\Lambda_1 \in \mR^{m_1 \x n}$,\, $\Lambda_2 \in
\mR^{m_2 \x n}$,
$$
    Q_{m_1+m_2}
    \left(
        G_1 \x G_2,
        \begin{bmatrix}
        \Lambda_1 \\
        \Lambda_2
        \end{bmatrix}
    \right) =
    Q_{m_1}
    (
        G_1,\Lambda_1
    )
    \bigcap
    Q_{m_2}
    (
        G_2,\Lambda_2
    );
$$
\item[{\bf (f) }]
for any $m \in \mN$, $\Lambda \in \mR^{m \x n}$, $G \subset
\mR^m$ and $z \in \mZ^n$,
$$
    Q_m(G,\Lambda) + z =
    Q_m
    (
        G+\Lambda z, \Lambda
    ).
$$
\end{itemize}
\end{lemma}
\begin{proof}  Since the set $G +\mZ^m$ is translation invariant, then, in view of
Lemma~\ref{cellprop2}(b),  it can be represented as
$$
    G + \mZ^m =
    \left(
        V
        \bigcap
        (
            G + \mZ^m
        )
    \right) + \mZ^m
$$
for any cell $V \subset \mR^m$. In view of (\ref{QmGL}), this implies the
assertion (a) of Lemma~\ref{formalizm}. The assertion (b) follows from the
property immediately below. For any given cell $V \subset \mR^m$,
the map $K:  2^V \to 2^{\mR^m}$, defined by $ K(G) = G + \mZ^m $,
satisfies $K(V) = \mR^m$ and preserves the set theoretic
operations in the sense that
\begin{align*}
    \mR^m \setminus K(G) & =
    K(V \setminus G),\\
    K(G) \bigcap K(H) & =
    K\left(G \bigcap H\right),\\
    K(G) \bigcup K(H) & =
    K\left(G \bigcup H\right)
\end{align*}
for any $G, H \subset V$.
In order to prove the assertion (c), we note that a point $x \in \mZ^n$ belongs
to the set on the left-hand side  of (\ref{QnmGIL}) if and only if there exist
$ u \in \mR^n$,\ $v \in \mR^m$,\ $y \in \mZ^n$,\ $z \in \mZ^m$
satisfying $x = u + y$,\ $ \Lambda x = v + z$ and
$\begin{bmatrix}u \\ v\end{bmatrix} \in G$. The
last three relations hold if and only if $x$ belongs to the set on the
right-hand side  of (\ref{QnmGIL}), with the set $H$ given by
(\ref{QnmGILH}). The assertion (d) follows from the property that\
$(F G) + \mZ^m = F (G + \mZ^m)$ for any unimodular matrix $F \in \mZ^{m
\x m}$ and any set $G \subset \mR^m$. The assertion (e) of the lemma
follows from the equality\ $ (G_1 \x G_2) + \mZ^{m_1 + m_2} =
(G_1 + \mZ^{m_1}) \x (G_2 + \mZ^{m_2}) $ for any sets $G_1
\subset \mR^{m_1}$ and $G_2 \subset \mR^{m_2}$. Finally, the
assertion (f) is proved by noting that for any $z \in \mZ^n$, a point $x \in
\mZ^n$ satisfies $\Lambda x \in G + \mZ^m$ if and only if  $\Lambda
(x + z) \in G + \Lambda z + \mZ^m$. \end{proof}

By Lemma~\ref{formalizm}~(a), we can  restrict ourselves, without
loss of generality, to considering the sets $Q_m(G,\Lambda)$ in (\ref{QmGL}) only for $G
\subset V$, where $V$ is a fixed but otherwise arbitrary cell in
$\mR^m$. For any matrix $\Lambda \in \mR^{m \x n}$, we define  the
class of sets
\begin{equation}\label{calQmL}
    \cQ _m(\Lambda)
    :=
    \left\{
        Q_m(G,\Lambda):\
        G \subset [0,1)^m\
        {\rm is\ Jordan\ measurable}
    \right\}.
\end{equation}
From Lemma~\ref{formalizm}(b) and the property that Jordan
measurable subsets of $[0,1)^m$ form an algebra, it follows that $\cQ _m(\Lambda)$  in (\ref{calQmL}) is
an algebra of subsets of $\mZ^n$.

\begin{definition}\label{Lquasiper}
The algebra $\cQ_m(\Lambda)$, defined by (\ref{calQmL}) for given $m \in \mN$ and $\Lambda \in \mR^{m \x n}$, is  called the algebra of $\Lambda$-quasiperiodic subsets of $\mZ^n$.
\end{definition}

 \begin{lemma}\label{Lquasiperprop}
 The algebra (\ref{calQmL}) possesses the following
properties:

\begin{itemize}
 \item[{\bf (a)}]
 for  any unimodular matrix $F \in \mZ^{m \x m}$,\ \, $
\cQ _{m}(F \Lambda) = \cQ _{m}(\Lambda)$;
 \item[{\bf (b)}]
if  $\Lambda_1 \in \mR^{m_1 \x n}$ is a submatrix of $\Lambda_2 \in \mR^{m_2 \x n}$, then $ \cQ _{m_1}(\Lambda_1) \subset \cQ _{m_2}(\Lambda_2)$;
 \item[{\bf (c)}]
 for any $m \in \mN$ and any $\Lambda \in \mR^{m \x n}$,
\begin{equation}\label{QnmIL}
    \cQ _{n+m}
    \left(
        \begin{bmatrix}
            I_n \\
            \Lambda
        \end{bmatrix}
    \right)
    =
    \cQ _m(\Lambda) ;
\end{equation}
\item[{\bf (d)}]
the algebra $\cQ _m(\Lambda)$ is translation invariant in the sense that  $A \in \cQ _m(\Lambda)$ implies  $A + z \in \cQ _m(\Lambda)$ for any $z \in \mZ^n$.
\end{itemize}
 \end{lemma}

\begin{proof}
In order to prove the assertion (a), we note that, in view of the assertions (a) and (d) of Lemma~\ref{formalizm},
$$
    Q_m
    (
        G, F\Lambda
    )
    =
    Q_m(K(G), \Lambda)
$$
holds for any unimodular matrix $F \in \mZ^{m \x m}$ and any set $G \subset [0,1)^m$. Here, the set $K(G)\subset [0,1)^m$ is given by
$$
    K(G)
    =
    [0,1)^m
    \bigcap
    (F^{-1} G + \mZ^m).
$$
The map $K$ is a bijection on the algebra of Jordan measurable subsets of $[0,1)^m$, whence the assertion (a) of the lemma follows.
In order to establish the assertion (b), it suffices to consider the case where $$
    \Lambda_2
    =
    \begin{bmatrix}
        \Lambda_1 \\
        \Lambda_3
    \end{bmatrix},
$$
with $\Lambda_3 \in \mR^{(m_2-m_1) \x n}$. Indeed, this structure can always be achieved by permuting the rows of the matrices $\Lambda_1$ and $\Lambda_2$, which, in view of the assertion (a) of the lemma, does not affect the corresponding algebras  of quasiperiodic sets $\cQ_{m_1}(\Lambda_1)$ and $\cQ_{m_2}(\Lambda_2)$. By using the assertions (b) and (e) of Lemma~\ref{formalizm}, it follows that any $\Lambda_1$-quasiperiodic set $Q_{m_1}(G,\Lambda_1)$ (with a Jordan measurable set $G \subset [0,1)^{m_1}$) is representable as
$$
    Q_{m_1}(G,\Lambda_1)
    =
    Q_{m_2}(G \x [0,1)^{m_2-m_1},\Lambda_2).
$$
The right-hand side of this  equality is a $\Lambda_2$-quasiperiodic subset of $\mZ^n$, thus establishing the assertion (b) of the lemma. In order to prove the assertion (c), we note that, since $\Lambda \in \mR^{m \x n}$ is  a submatrix of $\begin{bmatrix} I_n \\ \Lambda\end{bmatrix}$,  then, in view of the above assertion (b), the algebra $\cQ _m(\Lambda)$ is a subalgebra of the algebra on the left-hand side of (\ref{QnmIL}). On the other hand, by applying Lemma~\ref{formalizm}(c), it follows that  any $\begin{bmatrix} I_n \\ \Lambda\end{bmatrix}$-quasiperiodic set is representable as a $\Lambda$-quasiperiodic set. Hence, these two algebras also satisfy the opposite inclusion and the equality (\ref{QnmIL}). The assertion (d) of the lemma follows from Lemma~\ref{formalizm}(f). \end{proof}

For what follows,  we need an extension of the algebra of quasiperiodic sets (\ref{calQmL}) to infinite-dimensional matrices $\Lambda$. More precisely, for a given matrix
\begin{equation}\label{calL}
    \cL
    :=
    (
        \lambda_{jk}
    )_{j \in \mZ,\ 1 \< k \< n}
    \in
    \mR^{\infty \x n},
\end{equation}
we define
\begin{align}
    \nonumber
    \cQ _{\infty}(\cL)
    :=&
    \left\{
        Q_m(G,\Lambda):\
        G \in [0,1)^m\
        {\rm is\ Jordan\ measurable},\right.\\
        \nonumber
        &
        \left. \quad \Lambda \in \mR^{m\x n}\
        {\rm is\ a\ submatrix\ of}\ \cL,\ m \in \mN
    \right\} \\
    \label{calQcalL}
    =&
    \bigcup_{m \> 1}\ \
    \bigcup_{\Lambda\ {\rm  is\ an }\
    (m\x n)-{\rm submatrix\ of}\ \cL}\
    \cQ _m(\Lambda)
\end{align}
which is also an algebra of subsets of the lattice $\mZ^n$.

\begin{definition}
\label{calLquasiper}
The algebra $\cQ_{\infty}(\cL)$, defined by (\ref{calQcalL}) for a given matrix $\cL \in \mR^{\infty \x n}$, is called an algebra of $\cL$-quasiperiodic sets.
\end{definition}

\subsection
{Frequency measurability of quasiperiodic sets}
\label{FMQS}

With any $\omega \in \mR^n$, we associate an elementary trigonometric
polynomial $T_{\omega}: \mR^n \to \mC$ defined by
\begin{equation}\label{Tlambda}
    T_{\omega}(x) :=
    \re^{
        2\pi i \omega^{\rT} x
    },
\end{equation}
where $i:= \sqrt{-1}$ is the imaginary unit.

\begin{lemma}\label{tripol}
For any $\omega \in \mR^n$, the function $T_{\omega}$ in (\ref{Tlambda}) is averageable and its average value is  $ \bA ( T_{\omega} ) = \cI_{\mZ^n}(\omega) $.
\end{lemma}

\begin{proof}  If $\omega \in \mZ^n$, then $T_{\omega}(x) = 1$ for all $x
\in \mZ^n$, and hence, $\bA(T_{\omega}) = 1= \cI_{\mZ^n}(\omega)$ holds for such vectors $\omega$. Now, suppose $\omega :=
(\omega_k)_{1\< k \< n} \not\in \mZ^n$, that is, the indices of noninteger entries of $\omega$ form a nonempty set
\begin{equation}
\label{cK}
    \cK
    :=
    \left\{
        1 \< k \< n:\ \omega_k \not\in \mZ
    \right\}.
\end{equation}
The function $T_{\omega}$ in (\ref{Tlambda}) has the following average value (\ref{WfP}) over the discrete parallelepiped $P_{a,\ell}$ in (\ref{Pal}):
\begin{equation}\label{WTPal}
    W
    (
        T_{\omega}, P_{a,l}
    )
    =
    \exp
    \left(
        2 \pi i
        \sum_{k=1}^{n} \omega_k
        \left(
            a_k + \frac{\ell_k-1}{2}
        \right)
    \right)
    \prod_{k=1}^{n}
    \psi(\ell_k, \omega_k),
\end{equation}
where the function $\psi:  \mN \x \mR \to [-1,1]$ is defined by
\begin{equation}\label{psiuv}
    \psi(u,v) =
    \left\{
        \begin{matrix}
        (-1)^{(u-1)v} & {\rm for} & v \in \mZ \\ \\
        \frac{\sin(\pi u v)}{u \sin(\pi v)} & {\rm for} & v
        \not\in \mZ
        \end{matrix}
    \right. .
\end{equation}
From (\ref{WTPal}) and (\ref{psiuv}), it follows that for any $N \in \mN$,
\begin{equation}\label{supPN}
    \sup_{P \in \cP_N}
    |
        W
        (
            T_{\omega}, P
        )
    |
    \<
    \frac{1}{N^{\# \cK }}
    \prod_{k \in \cK }
    \frac{1}{
    |
        \sin(\pi \omega_k)
    |},
\end{equation}
where $\cP_N$ is the class of sufficiently large parallelepipeds given by (\ref{calPN}).  Since the set $\cK$ in (\ref{cK}) is not empty (and hence, $\# \cK >0$), then the right-hand side of (\ref{supPN}) converges to zero as $N \to +\infty$. This convergence implies that, for any vector $\omega \in \mR^n\setminus \mZ^n$, the function $T_{\omega}$ in (\ref{Tlambda}) is averageable, with $\bA(T_{\omega}) = 0 = \cI_{\mZ^n}(\omega)$, thus  completing the proof of the lemma. \end{proof}

%
%

\begin{definition}
\label{nonres}
A matrix $\Lambda \in \mR^{m \x n}$ is said to be nonresonant if the
rows of the matrix $\begin{bmatrix} I_n \\ \Lambda
\end{bmatrix}\in \mR^{(n+m) \x n}$ are linearly
independent over the field of rationals.
\end{definition}
Note that nonresonant matrices  do exist. Moreover, \textit{resonant} matrices $\Lambda \in \mR^{m\x n}$ (which are not nonresonant) form a set of $mn$-dimensional Lebesgue measure zero. Indeed, since the nonresonance property is equivalent to
 \begin{equation}
 \label{LTu}
    \Lambda^{\rT} u \not\in \mZ^n
    \quad
    {\rm for\ all}\
    u \in
    \mZ^m \setminus \{0\},
 \end{equation}
then the set of all resonant matrices $\Lambda \in \mR^{m \x n}$ can be represented as a countable union
$$
    \bigcup_{u \in \mZ^m \setminus \{0\},\ v \in \mZ^n}
    \Gamma_{u,v}
$$
of $(m-1)n$-dimensional affine subspaces $ \Gamma_{u,v} := \{ \Lambda \in \mR^{m \x n}:\ \Lambda^{\rT} u = v \} $, each of which has zero $mn$-dimensional Lebesgue measure.

\begin{theorem}
\label{periodicfun}
Suppose $\Lambda \in \mR^{m \x n}$ is  a nonresonant matrix, and  a function $f: \mR^m \to \mR$ satisfies the conditions
\begin{itemize}
\item[{\bf (a)}]
$f$ is unit periodic with respect to its $m$ variables;
\item[{\bf (b)}]
$f$ is $\mes_m$-continuous;
\item[{\bf (c)}]
$f$ is bounded.
\end{itemize}
Then the function $f \circ \Lambda: \mZ^n \to \mR$ is averageable, and its average value is computed as
\begin{equation}\label{AfL}
    \bA(f \circ \Lambda) =
    \int_{[0,1]^m}
    f(u)
    \rd u.
\end{equation}
\end{theorem}
\begin{proof}
We will follow a standard scheme which is known as the method of trigonometric sums in number theory \cite{Vinogradov} and as Weyl's equidistribution  criterion in the theory of weak convergence of measures \cite{Billingsley}.  More precisely, the proof will be carried out in three steps: we will establish (\ref{AfL}) for trigonometric polynomials $f$, then for continuous functions $f$ and finally, for  arbitrary functions $f$ satisfying the assumptions (a)--(c) of the theorem. Throughout the proof, $\Lambda \in \mR^{m \x n}$ is a nonresonant matrix in the sense of (\ref{LTu}).

{\bf Step 1.} Suppose $f: \mR^m \to \mC$ is a trigonometric polynomial,
that is, a linear combination of functions $T_{\omega}$ from (\ref{Tlambda}):
\begin{equation}
\label{ftripol}
    f
    :=
    \sum_{\omega \in \Omega}
    c_{\omega}T_{\omega}.
\end{equation}
Here, $c_{\omega}$ are complex coefficients, and $\Omega$ is a finite subset of $\mZ^n$,  which, without  loss of generality, is assumed to contain the zero vector. Then
\begin{equation}\label{c0}
    c_0 =
    \int_{[0,1]^m}
    f(u)
    \rd u.
\end{equation}
The composition $f \circ \Lambda$ of the function $f$ with the linear  map specified  by the matrix $\Lambda$  is also a trigonometric polynomial:
$$
    f \circ \Lambda =
    \sum_{\omega \in \Omega}
    c_{\omega} T_{\Lambda^{\rT}\omega}.
$$
By using Lemmas \ref{averfunprop} and \ref{tripol}, it follows that the function $f\circ \Lambda$ is averageable with the average value
\begin{equation}
\label{AfLambda}
    \bA
    (
        f \circ \Lambda
    ) =
    \sum_{\omega \in \Omega}
    c_{\omega}
    \bA
    (
        T_{\Lambda^{\rT}\omega}
    ) =
    \sum_{\omega \in \Omega}
    c_{\omega}
    \cI_{\mZ^n}
    (
        \Lambda^{\rT} \omega
    ).
\end{equation}
Under the nonresonance condition (\ref{LTu}) for $\Lambda$, the relation (\ref{AfLambda}) reduces to
$\bA(f\circ \Lambda) = c_0$, which  is equivalent to
(\ref{AfL}) in view of (\ref{c0}).

{\bf Step 2.} Suppose $f: \mR^m \to \mR$ is a unit periodic continuous
function. By the Weierstrass approximation theorem, for any $\eps >
0$, there exists a trigonometric polynomial $f_{\eps}: \mR^m \to
\mC$, defined by (\ref{ftripol}), such that
$$
    \max_{u \in [0,1]^m}
    |f(u)-f_{\eps}(u)| \< \eps.
$$
Therefore, by using the Lipschitz continuity of the upper average value functional (see Lemma~\ref{averfunprop}(a)) and Step~1, \begin{equation}\label{limsup}
    \bA^*(f \circ \Lambda)
    \<
    \bA
    (
        f_{\eps} \circ \Lambda
    ) + \eps =
    \int_{[0, 1]^m}
    f_{\eps}(u) \rd u + \eps
    \<
    \int_{[0, 1]^m}
    f(u) \rd u + 2 \eps.
\end{equation}
In view of the arbitrariness of  $\eps > 0$, it follows from (\ref{limsup}) that
\begin{equation}\label{upperest}
    \bA^*(f \circ \Lambda)
    \<
    \int_{[0, 1]^m}
    f(u) \rd u.
\end{equation}
A similar reasoning leads to
\begin{equation}\label{lowerest}
    \bA_*(f \circ \Lambda)
    \>
    \int_{[0, 1]^m}
    f(u) \rd u.
\end{equation}
The inequalities (\ref{upperest}) and (\ref{lowerest}) imply
the averagability of the function $f \circ \Lambda$, with the average value given by (\ref{AfL}).

{\bf Step 3.}
Now, let $f: \mR^m \to \mR$ be an arbitrary  function satisfying the conditions (a)--(c) of the theorem. Consider the functions $f^-, f^+ : \mR^m \to \mR$ defined by
$$
    f^-(u) =
    \min
    \left(
        f(u),
        \liminf_{v \to u}
        f(v)
    \right),
    \quad
    f^+(u) =
    \max
    \left(
        f(u),
        \limsup_{v \to u}
        f(v)
    \right).
$$
Both functions $f^-$ and $f^+$ are unit periodic in their $m$ variables, and are lower and upper semicontinuous, respectively. Also, they satisfy the inequalities $$ f^-(u) \< f(u) \< f^+(u) $$ which turn into equalities for $\mes_m$-almost all $u \in [0,1)^m$ in view of the $\mes_m$-continuity of the function $f$. Hence,  there exists a decreasing sequence of unit periodic continuous functions $f_k^+: \mR^m \to \mR$ which converge to the function $f^+$ point-wise in the cube $[0,1)^m$ and hence, to the function $f$ almost everywhere in $[0,1)^m$ as $k \to +\infty$. Therefore, by using Step 2, it follows that
\begin{equation}\label{buster1}
    \bA^*(f \circ \Lambda) \<
    \int_{[0, 1]^m}
    f_k^+(u)
    \rd u.
\end{equation}
Application of the Lebesgue dominated convergence theorem to the right-hand side of (\ref{buster1}) leads to the upper bound (\ref{upperest}).  The lower bound (\ref{lowerest}) is verified similarly. As before, (\ref{upperest}) and (\ref{lowerest}) imply the avaragability of the function $f \circ \Lambda$ together with (\ref{AfL}), thus completing the proof of the theorem. \end{proof}

\bigskip
 \begin{theorem}
 \label{periodicset}
For any nonresonant matrix $\Lambda \in \mR^{m \x n}$ and any Jordan measurable set $G \subset \mR^m$, the
$\Lambda$-quasiperiodic set $Q_m(G,\Lambda)$ in (\ref{QmGL}) is frequency measurable, and its frequency is computed as
\begin{equation}\label{FAG}
    \bF
    (
        Q_m(G,\Lambda)
    ) =
    \mes_m
    \left(
        V
        \bigcap
        (
            G+\mZ^m
        )
    \right),
\end{equation}
where $V \subset \mR^m$ is an arbitrary Lebesgue measurable cell.
\end{theorem}
\begin{proof}  The Jordan measurability of the set $G \subset \mR^m$ ensures that the indicator function $f := \cI_{G + \mZ^m} $ satisfies the conditions (a)--(c) of Theorem~\ref{periodicfun}. Hence, if the matrix $\Lambda \in \mR^{m \x n}$  is nonresonant, then Theorem~\ref{periodicfun} implies that the function
\begin{equation}\label{fLGZ}
    f \circ \Lambda  =
    \cI_{G + \mZ^m} \circ \Lambda
\end{equation}
is averageable, and
\begin{equation}\label{final}
    \bA
    (
        f \circ \Lambda
    ) =
    \int_{[0,1]^m}
    f(u) \rd u =
    \mes_m
    \left(
        [0,1)^m
        \bigcap
        (
            G+\mZ^m
        )
    \right).
\end{equation}
Therefore, since the right-hand side of (\ref{fLGZ}) is the indicator function of the set
$Q_m(G,\Lambda)$ in (\ref{QmGL}),  then this set is indeed frequency measurable,  and, in view of
(\ref{final}), its frequency is given by (\ref{FAG}) for any
Lebesgue measurable cell $V \subset \mR^m$, with the last property
following from Lemma~\ref{cellprop2}(c). \end{proof}

According to Theorem~\ref{periodicset}, if $\Lambda \in \mR^{m\x n}$ is a nonresonant matrix, then all the $\Lambda$-quasiperiodic sets, which form the algebra $\cQ_m(\Lambda)$ in (\ref{calQmL}), are frequency measurable subsets of the lattice $\mZ^n$. This result can be generalised to infinite-dimensional matrices  $\cL$ in (\ref{calL}) by appropriately extending Definition~\ref{nonres}.

\begin{definition}
\label{infnonres}
A matrix $\cL \in \mR^{\infty \x n}$ is said to be nonresonant if all its submatrices $\Lambda \in \mR^{m \x n}$ are nonresonant in the sense of Definition~\ref{nonres} for any $m\in \mN$.
\end{definition}
The existence of nonresonant matrices $\cL \in \mR^{\infty \x n}$ is established similarly to that in the case of finite-dimensional matrices.

%
%

\begin{theorem}
\label{calLquasiperprop} Suppose a matrix $\cL \in \mR^{\infty \x n}$ is nonresonant. Then the algebra $\cQ_{\infty}(\cL)$ of $\cL$-quasiperiodic sets in (\ref{calQcalL}) has the following properties
\begin{itemize}
\item[{\bf (a)}]
$\cQ _{\infty}(\cL)$ consists of frequency measurable subsets of
$\mZ^n$;
\item[{\bf (b)}]
for any $N \in \mN$ and any collection of $N$ pairwise nonoverlapping
submatrices $\cL_k \in \mR^{\infty \x n}$ of the matrix $\cL$,  the corresponding algebras
$\cQ_{\infty}(\cL_k)$ are mutually independent in the
sense that
$$ \bF\left(\bigcap_{k=1}^{N} A_k\right) =
\prod_{k=1}^{N} \bF(A_k)
$$ for any $A_k \in \cQ_{\infty}(\cL_k)$.
\end{itemize}
\end{theorem}

\begin{proof}  The assertion (a) follows from Theorem~\ref{periodicset} and
from the property that for any $A \in \cQ_{\infty}(\cL)$, there exists
a finite-dimensional  submatrix $\Lambda \in \mR^{m \x n}$ of the matrix
$\cL$ such that $A \in \cQ_m(\Lambda)$.
In order to prove the assertion (b), we note that for any sets $A_k \in \cQ_{\infty}(\cL_k)$ described
in the theorem, there exist finite-dimensional submatrices
$\Lambda_k \in \mR^{m_k \x n}$ of the corresponding matrices $\cL_k$ (and hence, $\Lambda_1, \ldots, \Lambda_N$ are also non-overlapping submatrices of the matrix $\cL$),  and Jordan measurable sets $G_k \subset
[0,1)^{m_k}$  such that $A_k = Q_{m_k}(G_k, \Lambda_k)$.
Therefore, repeated application of Lemma~\ref{formalizm}(e) yields
\begin{equation}
\label{AAA}
    \bigcap_{k=1}^{N}
    A_k =
    Q_{m_1+\ldots +m_N}
    \left(
        G_1
        \x
        \ldots
        \x
        G_N,
            \begin{bmatrix}
            \Lambda_1 \\
            \vdots \\
            \Lambda_N
            \end{bmatrix}
    \right).
\end{equation}
Since the matrix $\cL$ is nonresonant, then the right-hand side of (\ref{AAA}) is a quasiperiodic set (associated  with a nonresonant matrix) whose frequency measurability  is guaranteed by Theorem~\ref{periodicset}, with
$$
    \bF
    \left(
    \bigcap_{k=1}^{N}
    A_k
    \right)
    =
    \prod_{k=1}^{N}
    \mes_{m_k}G_k
    =
    \prod_{k=1}^{N}
    \bF ( A_k ),
$$
thus completing the proof of Theorem~\ref{calLquasiperprop}. \end{proof}

\subsection{Distributed quasiperiodic maps}
\label{DQM}

\begin{theorem}\label{periodicmap}
Suppose a matrix $\Lambda \in \mR^{m \x n}$ is nonresonant in the sense of Definition~\ref{nonres}. Also, let $g: \mR^m \to X$ be a map with values in a metric space $X$ such that
\begin{itemize}
\item[{\bf (a) }]
$g$ is unit periodic with respect to  its $m$ variables;
\item[{\bf (b) }]
$g$ is $\mes_m$-continuous.
\end{itemize}
Then the composition $g \circ \Lambda: \mZ^n \to X$ is $D$-distributed, with $D$ a probability measure on $(X,\cB)$ given by
\begin{equation}\label{mesD}
    D(B)
    :=
    \mes_m
    \left(
        V \bigcap g^{-1}(B)
    \right),
\end{equation}
where $V \subset \mR^m$ is an arbitrary Lebesgue measurable
cell. Furthermore, the algebra $\cS_{g \circ \Lambda}$, generated by the map $g\circ \Lambda$, consists of $\Lambda$-quasiperiodic sets:
\begin{equation}\label{calSgL}
    \cS_{g \circ \Lambda}
    \subset
    \cQ_m(\Lambda).
\end{equation}
\end{theorem}
\begin{proof}
For what follows, let $ \mod_m: \mR^m \to [0,1)^m $ denote the map which sends a vector $u := (u_k)_{1 \< k \< m}$ to the vector
$$
    \mod_m(u) :=
    (\lfp u_k \rfp )_{1 \< k \< m},
$$
where $\lfp \cdot \rfp$ denotes the fractional part of a number.  The map $\mod_m$ is unit periodic with respect to its $m$ variables  and identically maps the cube $[0,1)^m$ onto itself. We will now prove that if $\Lambda \in \mR^{m \x n}$ is a nonresonant matrix, then the composition $\mod_m \circ \Lambda: \mZ^m \to [0,1)^m$ is $\mes_m$-distributed. To this end, note that, for any bounded continuous function  $\phi: [0,1)^m \to \mR$, the function $ f := \phi \circ \mod_m $ satisfies the conditions (a)--(c) of Theorem~\ref{periodicfun} and hence, the function $f \circ \Lambda$ is averageable, with
\begin{equation}\label{AphimodL}
    \bA
    (
        \phi \circ \mod_m \circ \Lambda
    ) =
    \int_{[0,1]^m}
    f(u) \rd u =
    \int_{[0,1]^m}
    \phi(u)
    \rd u .
\end{equation}
Since the function $\phi$ is otherwise arbitrary, then, in view of the criterion (a) of Lemma~\ref{distrcriteria}, it follows from (\ref{AphimodL}) that the map $\mod_m \circ \Lambda: \mZ^n \to [0,1)^m$ is $\mes_m$-distributed. Furthermore, application of Definition~\ref{distribution} yields
\begin{equation}\label{calSmodL}
    \cS_{\mod_m \circ \Lambda} =
    \cQ _m(\Lambda).
\end{equation}
Now, suppose a map $g: \mR^m \to X$ with values in a metric space $X$ satisfies the assumptions (a) and (b) of the theorem. In particular, the unit periodicity of $g$ implies that
\begin{equation}
\label{gmodL}
    g \circ \Lambda =
    g \circ \mod_m \circ \Lambda.
\end{equation}
Since  the map $g$ is $\mes_m$-continuous,  and $\mod_m \circ \Lambda$ has been proved above to be $\mes_m$-distributed, then, in view of Lemma~\ref{mapcomposition},  the map $g \circ \Lambda$ is $D$-distributed, where $D$ is a probability measure on $(X, \cB)$ given by
\begin{equation}\label{mesD1}
    D(B) =
    \mes_m
    \left(
        [0,1)^m
        \bigcap
        g^{-1}(B)
    \right).
\end{equation}
Due to the unit periodicity of $g$, the preimage $g^{-1}(B)$ of any set $B \subset X$ is a translation invariant subset of $\mR^m$. Hence, in view of the assertion (c) of Lemma~\ref{cellprop2}, the right-hand side of (\ref{mesD1}) coincides with that of (\ref{mesD}) for any Lebesgue measurable cell $V \subset \mR^m$. Moreover, by Lemma~\ref{mapcomposition},  it follows from (\ref{gmodL})  that $\cS_{g \circ \Lambda} \subset \cS _{\mod_m \circ\Lambda}$. The latter inclusion, combined with (\ref{calSmodL}), implies (\ref{calSgL}), thus completing the proof of the theorem. \end{proof}

\section{Quantized linear systems: general case}
\label{QLSGC}
\setcounter{equation}{0}

\subsection
{Definition of a quantized linear system}
\label{DQLS}

Suppose $R: \mR^n \to \mZ^n$ is a map which commutes with the additive group of translations of the lattice $\mZ^n$:
    \begin{equation}\label{commutation}
        R(u+z) = R(u) + z
        \quad
        {\rm for\ all}\
        u \in \mR^n,\ z \in \mZ^n.
    \end{equation}
Such a map $R$ is completely specified by the set $R^{-1}(0)$ which is a cell in $\mR^n$ in the sense of Definition~\ref{cell}. Indeed, in view of (\ref{commutation}), the preimages
$$
    R^{-1}(z) := \{ u \in \mR^n:\ R(u) = z
\} = R^{-1}(0)
+ z,
$$
considered for all $z\in \mZ^n$,  form a partition of $\mR^n$. Moreover, for a given set $V \subset \mR^n$, there exists a unique map $R: \mR^n \to \mZ^n$, satisfying (\ref{commutation}) with $R^{-1}(0) = V$, if and only if $V$ is a cell in $\mR^n$.

\begin{definition}\label{quantizer}
A map $R: \mR^n \to \mZ^n$ is called a quantizer if it satisfies
(\ref{commutation}) and the set $R^{-1}(0)$ is Jordan measurable.
\end{definition}

There is a one-to-one correspondence between quantizers and Jordan measurable cells in $\mR^n$.  An example of a quantizer is as follows.

\begin{definition}\label{rquantizer}
The quantizer $R_*: \mR^n \to \mZ^n$, with $R_*^{-1}(0) = [-
1/2, 1/2)^n$, is called the roundoff quantizer.
\end{definition}

The roundoff quantizer $R_*$ maps a vector $u := (u_k)_{1 \< k \< n} \in \mR^n$ to a nearest node of  the lattice $\mZ^n$ given by
$$
    R_*(u)
    :=
    \left( \left\lfloor u_k + \frac{1}{2} \right\rfloor \right)_{1 \< k \< n},
$$
with $\lfloor \cdot\rfloor$ the floor function. This is an idealised   model of discretization in fixed-point arithmetic (with no overflow taken into account). In the framework of this model, there are more complicated quantizers $R$ whose cells $R^{-1}(0)$ are different from the cube $[-1/2, 1/2)^n$,  as discussed in Section~\ref{Cells}.

\begin{definition}\label{system}
Suppose $R:\mR^n \to \mZ^n$ is a quantizer and $L \in \mR^{n \x n}$ is a nonsingular matrix. The dynamical system  in the state space $\mZ^n$, with the transition operator $T: \mZ^n\to \mZ^n$ given by
    \begin{equation}\label{trans}
        T := R \circ L,
    \end{equation}
is called a quantized linear $(R,L)$-system.
\end{definition}
The map $T$,  defined by (\ref{trans}), provides a model for the spatially discretized dynamical system which arises in simulating a linear system in $\mR^n$ (specified by the matrix $L$) on a computer with fixed-point arithmetic.

\subsection{Associated, backward and forward algebras}
\label{ABFA}

The following lemma shows that the class of quasiperiodic subsets of the lattice, introduced in Section~\ref{QS}, is closed under the dynamics of the quantized linear system  in the sense that quasiperiodic sets are transformed to quasiperiodic sets.

\begin{lemma}\label{formalism2}
For any $m \in \mN$, any set $G \subset \mR^m$ and any matrix
$\Lambda \in \mR^{m\x n}$, the quasiperiodic set $Q_m(G,\Lambda)$ in (\ref{QmGL}) is transformed
by the transition operator $T$ in (\ref{trans}) and its set-valued inverse $T^{-1}$
as
\begin{align}
    \label{T-1QmGL}
        T^{-1}(Q_m(G,\Lambda))
        &=
        Q_{n+m}
        \left(
            \{0\}\x G +
                \begin{bmatrix}
                    I_n \\ \Lambda
                \end{bmatrix}
            R^{-1}(0),
                \begin{bmatrix}
                    I_n \\ \Lambda
                \end{bmatrix}
                L
        \right),\\
    \label{TQmGL}
        T(Q_m(G,\Lambda))
        &=
        Q_{m+n}
        \left(
            G \x \{0\}
            -
                \begin{bmatrix}
                    \Lambda
                    \\
                    I_n
                \end{bmatrix} L^{-1} R^{-1}(0),
                \begin{bmatrix}
                    \Lambda \\ I_n
                \end{bmatrix}
                L^{-1}
        \right).
    \end{align}
\end{lemma}
\begin{proof}
The definition (\ref{QmGL}) implies that a point $x \in \mZ^n$ belongs to the set on the right-hand side of (\ref{T-1QmGL}) if and only if there exist
    \begin{equation}\label{uvyz1}
        u \in G,
        \quad
        v \in R^{-1}(0),
        \quad
        y \in \mZ^n,
        \quad
        z \in \mZ^m
    \end{equation}
satisfying
    $$
            \begin{bmatrix}
                I_n \\ \Lambda
            \end{bmatrix}
            L x
            =
            \begin{bmatrix}
                0 \\ u
            \end{bmatrix}
            +
            \begin{bmatrix}
                I_n \\ \Lambda
            \end{bmatrix}v
            +
            \begin{bmatrix}
                y \\ z
            \end{bmatrix}.
    $$
In terms of the corresponding subvectors, the last equality is equivalent to
    \begin{equation}\label{uvyz2}
        L x = v + y,
        \qquad
        \Lambda(Lx-v)= \Lambda y = u + z.
    \end{equation}
On the other hand, in view of (\ref{uvyz1}), the leftmost equality in (\ref{uvyz2}) is equivalent to $T(x)=y$, whereas the rightmost equality in (\ref{uvyz2}) is equivalent to $y \in Q_m(G,\Lambda)$. Therefore, the fulfillment of (\ref{uvyz1}) and (\ref{uvyz2}) is equivalent to $T(x) \in Q_m(G,\Lambda)$, that is, $x \in T^{-1}(Q_m(G,\Lambda))$.
Since the point $x$ in the above considerations was arbitrary,
the set on the right-hand side of (\ref{T-1QmGL}) indeed coincides with $T^{-1}(Q_m(G,\Lambda))$.
Although the proof of  the representation (\ref{TQmGL}) is similar, we will provide it for completeness of exposition.  A point $x \in \mZ^n$
belongs to the set on the right-hand side of (\ref{TQmGL}) if and only if there
exist
    \begin{equation}\label{uvyz3}
        u \in G,
        \quad
        v \in R^{-1}(0),
        \quad
        y \in \mZ^n,
        \quad
        z \in \mZ^m
    \end{equation}
satisfying
    $$
            \begin{bmatrix}
                \Lambda \\ I_n
            \end{bmatrix}
            L^{-1} x =
            \begin{bmatrix}
                u \\ 0
            \end{bmatrix}
            -
            \begin{bmatrix}
                \Lambda \\ I_n
            \end{bmatrix}
            L^{-1} v +
            \begin{bmatrix}
                z \\ y
            \end{bmatrix}.
    $$
The last relation is equivalent to
    \begin{equation}\label{uvyz4}
        L y = v + x,
        \qquad
        \Lambda (L^{-1}x+L^{-1}v)=\Lambda y = u + z.
    \end{equation}
On the other hand, in view of (\ref{uvyz3}), the leftmost equality in (\ref{uvyz4}) is equivalent to
$ T(y) = x$, whilst the rightmost equality in (\ref{uvyz4}) is equivalent to $y \in Q_m(G, \Lambda) $. Therefore, the fulfillment of (\ref{uvyz3}) (\ref{uvyz4}) is equivalent to  $x \in T(Q_m(G,L))$. Hence, by the arbitrariness of the point $x$,  it follows that the set on the right-hand side of
(\ref{TQmGL}) is equal to $T(Q_m(G,\Lambda))$, thus completing the proof of the lemma. \end{proof}

Lemma~\ref{formalism2} shows that the action of the transition operator $T$ of the quantized linear system or its inverse  $T^{-1}$ on a quasiperiodic set $Q_m(G,\Lambda)$ modifies the matrix $\Lambda$, thus leading to a set with a qualitatively different quasiperiodicity pattern.  However, an algebraic closedness can be achieved here by restricting $\Lambda \in \mR^{m\x n}$ to submatrices of an infinite matrix $\cL \in \mR^{\infty \x n}$ which satisfies the property
\begin{equation}
\label{closure}
    \Lambda
    \lll \cL
    \qquad \Longrightarrow\qquad
                \begin{bmatrix}
                    I_n \\ \Lambda
                \end{bmatrix}
                L \lll \cL
                \ {\rm and}\
                \begin{bmatrix}
                    \Lambda \\ I_n
                \end{bmatrix}
                L^{-1}
                \lll \cL,
\end{equation}
where $A\lll B$ signifies ``$A$ is a submatrix of $B$'' for matrices with the same number of columns $n$. In particular, (\ref{closure}) implies that such a matrix $\cL$ must contain $L$ and $L^{-1}$ as submatrices. Hence, by induction, the minimal matrix $\cL$, which satisfies (\ref{closure}), is formed from integer powers of $L$. Moreover, in view of Lemma~\ref{Lquasiperprop}(c),  the zeroth power $L^0=I_n$ is redundant and can be discarded without affecting the algebras of quasiperiodic sets. Thus, for what follows, we define the matrix
    \begin{equation}\label{degL}
        \cL
        :=
        \begin{bmatrix}
        \vdots \\
        L^{-2} \\
        L^{-1} \\
        L \\
        L^2 \\
        \vdots
        \end{bmatrix}
        =
        \begin{bmatrix}
        \cL ^- \\
        \cL ^+
        \end{bmatrix} \in \mR^{\infty \x n},
    \end{equation}
which is obtained by ``stacking" nonzero integer powers of $L$ one underneath the other and is partitioned into the submatrices
    \begin{equation}\label{deg-+L}
        \cL ^-
        :=
        \begin{bmatrix}
        \vdots \\
        L^{-2} \\
        L^{-1} \\
        \end{bmatrix},
        \qquad
        \cL ^+
        :=
        \begin{bmatrix}
        L \\
        L^2 \\
        \vdots
        \end{bmatrix}.
    \end{equation}
Also, we denote by
    \begin{equation}\label{assalgebra}
        \cQ
        :=
        \cQ _{\infty}(\cL ),
    \end{equation}
    \begin{equation}\label{backforalgebras}
        \cQ ^- :=
        \cQ _{\infty}(\cL ^-),
        \qquad
        \cQ ^+ :=
        \cQ _{\infty}(\cL ^+)
    \end{equation}
the algebras of quasiperiodic subsets of the lattice $\mZ^n$ generated by the
matrices (\ref{degL}) and (\ref{deg-+L}), so that $\cQ ^-$ and
$\cQ ^+$ are subalgebras of $\cQ $.

\begin{definition}\label{}
For a given quantized linear $(R,L)$-system, the algebra $\cQ$ of $\cL $-quasipe\-riodic sets in
(\ref{assalgebra}),
generated by the matrix $\cL$ in (\ref{degL}), is called an
associated algebra.   The algebras $\cQ^-$ and $\cQ^+$ of
$\cL ^-$- and $\cL ^+$-quasiperiodic sets in (\ref{backforalgebras}), generated by the
matrices $\cL^-$ and $\cL^+$ in (\ref{deg-+L}), are called backward and forward
algebras, respectively.
\end{definition}
In order to clarify the structure of the associated, backward and forward algebras $\cQ$, $\cQ^-$ and $\cQ^+$, we define, for any $a, b \in
\mN$, the matrix
\begin{equation}\label{calLab}
    \cL _{a,b}
    :=
        \begin{bmatrix}
        L^{-a}\\
        \vdots \\
        L^{-1} \\
        L \\
        \vdots \\
        L^b
        \end{bmatrix}
        =
        \begin{bmatrix}
        \cL _a^- \\
        \cL _b^+
        \end{bmatrix}
        \in \mR^{(a+b)n \x n}
\end{equation}
which is partitioned into submatrices $\cL _a^-$ and $\cL _b^+$ given by
\begin{equation}\label{calLN}
    \cL _N^-
    :=
        \begin{bmatrix}
        L^{-N} \\
        \vdots \\
        L^{-1}
        \end{bmatrix},
    \qquad
    \cL _N^+
    :=
        \begin{bmatrix}
        L \\
        \vdots \\
        L^N
        \end{bmatrix}
\end{equation}
for any $N \in \mN$. We will now consider the corresponding algebras of quasiperiodic subsets of the lattice $\mZ^n$:
\begin{equation}\label{calQab}
    \cQ _{a,b} := \cQ _{(a+b)n} (\cL _{a,b}),
\end{equation}
\begin{equation}\label{calQN}
    \cQ _N^- := \cQ _{Nn}(\cL _N^-),
    \qquad
    \cQ _N^+ := \cQ _{Nn}(\cL _N^+).
\end{equation}
These algebras are monotonically increasing with respect to the corresponding subscripts in the sense that
$$
    \cQ_{a,b} \subset \cQ_{a+1,b},
    \quad
    \cQ_{a,b} \subset \cQ_{a,b+1},
    \quad
    \cQ_{a,b} \subset \cQ_{a+1,b+1},
    \quad
    \cQ_N^- \subset \cQ_{N+1}^-,
    \quad
    \cQ_N^+ \subset \cQ_{N+1}^+
$$
for all $a,b, N \in \mN$. Furthermore, the algebras $\cQ$, $\cQ^-$ and $\cQ^+$ in (\ref{assalgebra})
and (\ref{backforalgebras}) are representable in terms of (\ref{calQab}) and (\ref{calQN}) as
$$
    \cQ = \bigcup_{a,b \> 1} \cQ _{a,b},
    \qquad
    \cQ ^- = \bigcup_{N \> 1} \cQ _N^-,
    \qquad
    \cQ ^+ = \bigcup_{N \> 1} \cQ _N^+.
$$

\begin{lemma}\label{algebraprop2}
The algebras $\cQ_{a,b}$, $\cQ_N^-$ and $\cQ_N^+$ in (\ref{calQab}) and (\ref{calQN}) are
transformed by the transition operator $T$ in (\ref{trans}) and its
set-valued inverse $T^{-1}$ as follows:
\begin{itemize}
\item[{\bf (a)}]
for any $a,b \in \mN$,
\begin{equation}\label{invTcalQab}
    T^{-1}
    (
        \cQ _{a,b}
    )
    \subset
    \cQ _{a-1,b+1},
    \qquad
    T
    (
        \cQ _{a,b}
    )
    \subset
    \cQ _{a+1,b-1};
\end{equation}
\item[{\bf (b)}]
for any $N \in \mN$,
\begin{equation}\label{invTcalQN}
    T^{-1}
    (
        \cQ _N^+
    )
    \subset
    \cQ _{N+1}^+,
    \qquad
    T
    (
        \cQ _N^-
    )
    \subset
    \cQ _{N+1}^- .
\end{equation}
\end{itemize}
\end{lemma}

\begin{proof}
We will first prove the assertion (a). Suppose $a, b \in \mN$, and let $A$ be a quasiperidoc set $A := Q_{(a+b)n}(G,
\cL _{a,b})\in
\cQ _{a,b}$, where $G \subset
\mR^{(a+b)n}$ is a Jordan measurable set. Then application of Lemma~\ref{formalism2} yields
\begin{equation}\label{bred1}
    T^{-1}(A) =
    Q_{(a+b+1)n}
    \left(
        G_+,\,
            \begin{bmatrix}
            I_n \\
            \cL _{a,b}
            \end{bmatrix}
        L
    \right),
    \qquad
    T(A) =
    Q_{(a+b+1)n}
    \left(
        G_-,\,
            \begin{bmatrix}
            \cL _{a,b} \\
            I_n
            \end{bmatrix}
        L^{-1}
    \right),
\end{equation}
where the sets $G_-, G_+ \subset \mR^{(a+b+1)n}$ are given by
$$
    G_+
    :=
    \{0\} \x G +
        \begin{bmatrix}
        I_n \\
        \cL _{a,b}
        \end{bmatrix}
    R^{-1}(0),
    \qquad
    G_-
    :=
    G \x \{ 0\} -
        \begin{bmatrix}
        \cL _{a,b} \\
        I_n
        \end{bmatrix}
    L^{-1} R^{-1}(0)
$$
and inherit Jordan measurability from $G$ and $R^{-1}(0)$. The definition of the matrix $\cL_{a,b}$ in  (\ref{calLab}) implies that
\begin{equation}\label{bred2}
        \begin{bmatrix}
        I_n \\
        \cL _{a,b}
        \end{bmatrix}
    L =
    F_+
        \begin{bmatrix}
        I_n \\
        \cL _{a-1,b+1}
        \end{bmatrix},
    \qquad
        \begin{bmatrix}
        \cL _{a,b} \\
        I_n
        \end{bmatrix}
        L^{-1} =
    F_-
        \begin{bmatrix}
        I_n \\
        \cL _{a+1,b-1}
        \end{bmatrix},
\end{equation}
where $F_+, F_- \in \{0,1\}^{(a+b+1)n \x (a+b+1)n}$ are permutation matrices
$$
    F_+ :=
        \begin{bmatrix}
        0 & 0 & I_n & 0 \\
        0 & I_{(a-1)n} & 0 & 0 \\
        I_n & 0 & 0 & 0 \\
        0 & 0 & 0 & I_{bn}
        \end{bmatrix},
    \qquad
    F_- :=
        \begin{bmatrix}
        I_{an} & 0 & 0 & 0 \\
        0 & 0 & 0 & I_n \\
        0 & 0 & I_{(b-1)n} & 0 \\
        0 & I_n & 0 & 0
        \end{bmatrix}.
$$
Hence, by using the assertions (c) and (d) of Lemma~\ref{formalizm}, it now follows from (\ref{bred1}) and (\ref{bred2}) that
\begin{equation}\label{invTA}
    T^{-1}(A) =
    Q_{(a+b)n}
    (
        H_+ , \cL _{a-1,b+1}
    ),
    \qquad
    T(A) =
    Q_{(a+b)n}
    (
        H_- , \cL _{a+1, b-1}
    ),
\end{equation}
where the sets
\begin{align*}
    H_+
    & :=
    \left\{
        v \in \mR^{(a+b)n}:\
            \begin{bmatrix}
            u \\
            v
            \end{bmatrix}
        \in
        F_+ G_+\
        {\rm for\ some}\         u \in \mZ^n
    \right\},\\
    H_-
    & :=
    \left\{
        u \in \mR^{(a+b)n}:\
            \begin{bmatrix}
            u \\
            v
            \end{bmatrix}
        \in
        F_- G_-\
        {\rm for\ some}\
        v \in \mZ^n
    \right\}
\end{align*}
are Jordan measurable. Therefore, the representations (\ref{invTA}) imply that $ T^{-1}(A) \in \cQ _{a-1, b+1}$ and $
T(A) \in \cQ _{a+1, b-1}$, thus establishing  the inclusions (\ref{invTcalQab}) in view of the arbitrariness of the set $A
\in \cQ _{a, b}$. We will now prove the assertion (b) of the lemma. Suppose $N \in \mN$ and $A \in \cQ_N^+$, that is,
\begin{equation}\label{AQNn}
    A := Q_{Nn}(G,\cL _N^+)
 \end{equation}
for a Jordan measurable set $G \subset \mR^{Nn}$. Then application of Lemma~\ref{formalism2} leads to
\begin{equation}\label{invTA1}
    T^{-1}(A) =
    Q_{(N+1)n}
    \left(
        H,
            \begin{bmatrix}
            I_n \\
            \cL _N^+
            \end{bmatrix}
        L
    \right),
 \end{equation}
where the set $H \subset \mR^{(N+1)n}$ is given by
\begin{equation}\label{setH}
    H
    :=
    \{0\} \x G +
        \begin{bmatrix}
        I_n \\
        \cL _N^+
        \end{bmatrix}
    R^{-1}(0).
\end{equation}
Since $ \begin{bmatrix} I_n \\ \cL _N^+ \end{bmatrix}
L = \cL _{N+1}^+ $ in view of (\ref{calLN}), then it follows from (\ref{invTA1}) and the
Jordan measurability of the set $H$ in (\ref{setH}) that
\begin{equation}\label{star3}
    T^{-1}(A) =
    Q_{(N+1)n}
    (
        H, \cL _{N+1}^+
    )
    \in
    \cQ _{N+1}^+.
 \end{equation}
This  representation implies the first of the inclusions in
(\ref{invTcalQN}) due to the arbitrariness of $A \in \cQ _N^+$. The second of the inclusions (\ref{invTcalQN}) can be  established in a similar fashion
by using the relations
$$
    T
    (
        Q_{Nn}
        (
            G, \cL _N^-
        )
    ) =
    Q_{(N+1)n}
    (
        G \x \{ 0\} - \cL _{N+1}^-  R^{-1}(0), \,
        \cL _{N+1}^-
    )
    \in
    \cQ _{N+1}^-
 $$
 which hold for any $N \in \mN$ and any Jordan measurable set
 $G \subset \mR^{Nn}$, which completes the proof of the
 lemma. \end{proof}

\begin{theorem}
\label{algebraprop}
The associated, backward and forward algebras $\cQ$, $\cQ^-$ and $\cQ^+$ in  (\ref{assalgebra}) and (\ref{backforalgebras})
are transformed by the transition operator $T$ in (\ref{trans}) and
its set-valued inverse $T^{-1}$ as follows:
\begin{itemize}
\item[{\bf (a)}]
the associated algebra $\cQ $ is invariant under the maps $T$ and $T^{-1}$:
$$
    T^{-1}
    (
        \cQ
    )
    \subset \cQ,
    \qquad
    T
    (
        \cQ
    )
    \subset
    \cQ;
$$
\item[{\bf (b)}]
the forward algebra  $\cQ ^+$ is invariant under the map $T^{-1}$:
$$
    T^{-1}
    (
        \cQ ^+
    )
    \subset
    \cQ ^+;
$$
\item[{\bf (c)}]
the backward algebra $\cQ ^-$ is invariant under $T$:
$$
    T
   (
        \cQ ^-
    )
    \subset
    \cQ ^-;
$$
\item[{\bf (d)}]
for any $A \in \cQ $, there exists $N \in \mN$ such that $
T^{-k}(A) \in \cQ ^+$ and $ T^k (A) \in \cQ ^- $ for
all $ k \> N $.
\end{itemize}
\end{theorem}
\begin{proof}
The assertions (a), (b) and (c) of the theorem follow from Lemma~\ref{algebraprop2}. The assertion (d) can be proved by using the inclusions
$$
    T^{-(a+c)}
    (
        \cQ _{a,b}
    )
    \subset
    \cQ _{a+b+c}^+
    \subset
    \cQ ^+,
    \qquad
    T^{(b+c)}
    (
        \cQ _{a,b}
    )
    \subset
    \cQ _{a+b+c}^-
    \subset
    \cQ ^-
$$
(which hold for all $a, b, c \in \mN$ and follow from Lemma~\ref{algebraprop2})  and the property that for any $A \in \cQ $ there exist $a, b \in \mN$ satisfying $A \in \cQ_{a,b}$. \end{proof}

By Theorem~\ref{algebraprop}(b), the transition operator $T$ is measurable with respect to the forward algebra $\cQ ^+$. The assertion (d) of the theorem can be interpreted as an absorbing property  of $\cQ ^+$ with respect to the map $T^{-1}$ and the absorbing property of  the backward algebra $\cQ ^-$ with respect to the transition operator $T$.

\subsection{Frequency preservation on the forward algebra}\label{PFFA}

We will need the following enhancement of the nonresonance property for the matrix $L$ of the quantized linear $(R,L)$-system.

%
%
%
%

\begin{definition}
\label{iternonres}
A nonsingular matrix $L \in \mR^{n\x n}$ is said to be iteratively nonresonant if the corresponding matrix $\cL \in \mR^{\infty \x n}$, associated with $L$ by (\ref{degL}), is nonresonant in the sense of Definition~\ref{infnonres}.
\end{definition}
Since the matrix $L$ is nonsingular, then
%
the iterative nonresonance property is equivalent to the rational independence of  the rows of the matrix
$$
        \begin{bmatrix}
        I_n \\
        \cL _N^+
        \end{bmatrix}
    =
        \begin{bmatrix}
        I_n \\
        L \\
        \vdots \\
        L^N
        \end{bmatrix}
    \in
    \mR^{n(N+1)\x n}
$$
for every $N \in \mN$, where $\cL_N^+$ is the matrix given by (\ref{calLN}). Therefore,
$L$ is iteratively nonresonant if and only if so is $L^{-1}$.
Also note that  iteratively nonresonant matrices $L \in \mR^{n \x n}$ do exist. Moreover, \emph{iteratively resonant} matrices $L$ (which
are not iteratively nonresonant) form  a set of zero
$n^2$-dimensional Lebesgue measure. Indeed, the set of such matrices $L
\in \mR^{n\x n}$ can be
represented as a countable union
$$
    \bigcup_{N \> 1}\
    \bigcup_{u \in \mZ^{Nn}\setminus\{0\},\ v \in \mZ^n}
    \Upsilon_{N,u,v}
$$
of the following sets
$$
    \Upsilon_{N,u,v}
    :=
    \left\{
        L \in \mR^{n \x n}:\
        \det L \ne 0\
        {\rm and}\
        \sum_{k=1}^{N}
        (
            L^{\rT}
        )^k
        u_k = v
    \right\},
$$
where $u:= \small{\begin{bmatrix}u_1\\ \vdots\\ u_N\end{bmatrix}}$ is an $Nn$-dimensional vector partitioned into
$n$-dimensional subvectors $u_1, \ldots, u_N$. Each of the sets $\Upsilon_{N,u,v}$ has zero $n^2$-dimensional Lebesgue measure which can be verified  as follows. By assuming, without loss of generality, that   $|u_N|=1$, the Hilbert space $\mR^{n\x n}$ (endowed with the Frobenius inner product of matrices \cite{Horn}) can be split into the orthogonal sum $\mR^{n\x n} = \Span(Z) \op Z^{\bot}$ of the one-dimensional subspace spanned by a nonzero idempotent matrix $Z:= u_Nu_N^{\rT}$ and the corresponding orthogonal complement $Z^{\bot}:= \{Y\in \mR^{n\x n}:\ \Tr (Y^{\rT} Z) = u_N^{\rT} Y u_N = 0\}$ which is an $(n^2-1)$-dimensional hyperplane in $\mR^{n\x n}$. Now, by considering an $\mR^n$-valued polynomial $h(X):= \sum_{k=1}^N X^k u_k - v$  for $X:= Y + \lambda Z$, with $Y \in Z^{\bot}$, it follows that $h(Y + \lambda Z)$ is a polynomial of degree $N$ with respect to $\lambda \in \mR$, with the leading coefficient $Z^N u_N = u_Nu_N^{\rT}u_N = u_N\ne 0$ in view of the idempotence of $Z$. Therefore, the following integral with respect to the $(n^2-1)$-dimensional Lebesgue measure over the hyperplane $Z^{\bot}$ vanishes,
\begin{align*}
    \mes_{n^2} \Upsilon_{N,u,v} & \< 
    \mes_{n^2} 
    \big\{
        X\in \mR^{n\x n}:\, 
        h(X) = 0
    \big\}\\
    & =
    \int_{Z^{\bot}} 
    \mes_1 
    \big\{
        \lambda\in \mR:\, 
        h(Y + \lambda Z) = 0
    \big\}\rd Y 
    = 
    0
\end{align*}
because the set, involved  in the integrand, is finite (consisting of at most $N$ values of $\lambda$) and, thus, has zero one-dimensional Lebesgue measure for every matrix $Y \in Z^{\bot}$.

\begin{theorem}\label{algebraprop1}
Suppose the matrix $L \in \mR^{n\x n}$ of the quantized linear $(R,L)$-system  is iteratively nonresonant. Then:

\begin{itemize}
\item[{\bf (a)}]
the associated algebra $\cQ$ in (\ref{assalgebra}) consists of frequency measurable subsets of the lattice $\mZ^n$;
\item[{\bf (b)}]
the backward and forward algebras $\cQ^-$ and $\cQ^+$ in (\ref{backforalgebras}) are independent in the sense that
$$ \bF \left( A \bigcap B \right) = \bF(A) \bF(B)$$ for
all $ A \in \cQ ^-$,\ $B \in \cQ ^+ $;
\item[{\bf (c)}]
the transition operator $T$ in (\ref{trans})  preserves the frequency $\bF$ on the forward algebra $\cQ^+$:
\begin{equation}\label{FinvTA}
    \bF
    (
        T^{-1}(A)
    ) =
    \bF(A)
    \quad
    {\rm for\ all}\
    A \in \cQ ^+.
\end{equation}
\end{itemize}
\end{theorem}
\begin{proof}
The assertion (a) of the theorem follows from Theorem~\ref{calLquasiperprop}(a). The assertion (b) is a corollary from Theorem~\ref{calLquasiperprop}(b) since the matrices $\cL ^-$ and $\cL ^+$, defined by (\ref{deg-+L}), are nonoverlapping submatrices of the matrix $\cL $ given by (\ref{degL}). We will now prove the assertion (c). Suppose $A \in \cQ ^+$ is a fixed but otherwise arbitrary set from the forward algebra, and hence, $A$ is representable by (\ref{AQNn}) for $N \in \mN$ and a Jordan measurable set $G \subset [0,1)^{Nn}$. Then, as was obtained in the proof of Lemma~\ref{algebraprop2}(b), the set $T^{-1}(A)$ is given by (\ref{star3}), where $H \subset \mR^{(N+1)n}$ is a Jordan measurable set defined by (\ref{setH}).  Therefore, application of Theorem~\ref{periodicset} (under the assumption that $L$ is iteratively nonresonant) yields
\begin{align}
\label{FAmesG}
    \bF(A)
    & =
    \mes_{Nn} G,\\
\label{FinvTAmesVH}
    \bF
    (
        T^{-1}(A)
    )
    &=
    \mes_{(N+1)n}
    \left(
        V \bigcap
        (
            H + \mZ^{(N+1)n}
        )
    \right),
\end{align}
where $V$ is an arbitrary Lebesgue measurable cell in $\mR^{(N+1)n}$. It will be convenient to use the set
\begin{equation}\label{specialcell}
    V
    :=
    R^{-1}(0) \x [0,1)^{Nn}
\end{equation}
which is a cell in $\mR^{(N+1)n}$ in view of
Lemma~\ref{cellprop1}(d). From (\ref{setH}) and
(\ref{specialcell}), it follows that  the set on the right-hand side
of (\ref{FinvTAmesVH}) is representable as
$$
    K
    :=
    V \bigcap
    (
        H + \mZ^{(N+1)n}
    )
    =
    \left\{
            \begin{bmatrix}
            u \\
            v
            \end{bmatrix}:\
        u \in R^{-1}(0),\
        v \in W_u
    \right\},
$$
where
\begin{equation}
\label{Wu}
    W_u
    :=
    [0,1)^{Nn} \bigcap
    (
        \cL_N^+ u + G + \mZ^{Nn}
    ).
\end{equation}
Hence,
\begin{equation}\label{repeatedint}
    \mes_{(N+1)n} K =
    \int_{R^{-1}(0)}
    \mes_{Nn} W_u
    \rd u.
\end{equation}
Here, the integrand is identically constant since (\ref{Wu}) and the above assumption that $G \subset [0,1)^{Nn}$ imply that $ \mes_{Nn} W_u = \mes_{Nn} G $ for all $u \in \mR^n$. In combination with  the equality $\mes_n R^{-1}(0) = 1$, this reduces the integral in (\ref{repeatedint}) to $\mes_{(N+1)n} K = \mes_{Nn} G$, and hence, (\ref{FinvTAmesVH}) takes the form
\begin{equation}
\label{star8}
    \bF
    (
        T^{-1}(A)
    ) =
    \mes_{Nn} G .
\end{equation}
In view of arbitrariness of the set $A \in \cQ^+$,   comparison of (\ref{FAmesG}) with (\ref{star8}) establishes (\ref{FinvTA}), thus completing the proof of the theorem. \end{proof}

Theorems~\ref{algebraprop} and~\ref{algebraprop1} are illustrated by
Fig.~\ref{fig2}.
 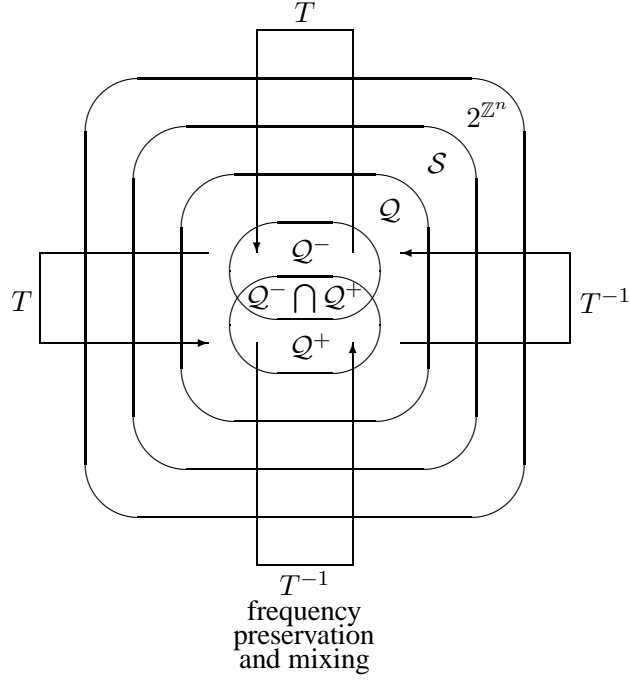
\begin{figure}[htbp]
\begin{center}
\unitlength=0.8mm
\linethickness{0.4pt}
\begin{picture}(120.00,110.00)
\put(60,63){\oval(25,16)} \put(60,72){\oval(25,16)}
\put(68,75){\line(0,1){37.00}} \put(52,112){\vector(0,-1){37.00}}
\put(68,112){\line(-1,0){16.00}}
\put(60,115){\makebox(0,0)[cc]{\small$T$}}
\put(68,23){\vector(0,1){37.00}} \put(52,60){\line(0,-1){37.00}}
\put(52,23){\line(1,0){16.00}}
\put(60,20){\makebox(0,0)[cc]{\small$T^{-1}$}}
\put(60,13){\makebox(0,4)[cc]{\small frequency}}
\put(60,9){\makebox(0,4)[cc]{\small preservation}}
\put(60,5){\makebox(0,4)[cc]{\small and mixing}}
\put(104,75){\vector(-1,0){28.00}} \put(76,60){\line(1,0){28.00}}
\put(104,60){\line(0,1){15.00}}
\put(110,67){\makebox(0,0)[cc]{\small$T^{-1}$}}
\put(16,60){\vector(1,0){28.00}} \put(44,75){\line(-1,0){28.00}}
\put(16,75){\line(0,-1){15.00}} \put(13,67){\makebox(0,0)[cc]{\small$T$}}
\put(60,67.5){\makebox(0,0)[cc]{\small$\cQ ^-\bigcap \cQ ^+$}}
\put(61,75.5){\makebox(0,0)[cc]{\small$\cQ ^-$}}
\put(61,59.5){\makebox(0,0)[cc]{\small$\cQ ^+$}}
\put(60,67.5){\oval(41,41)} \put(60,67.5){\oval(57,57)}
\put(60,67.5){\oval(73,73)} \put(74,82){\makebox(0,0)[cc]{\small$\cQ$}} \put(82,90){\makebox(0,0)[cc]{\small$\cS $}}
\put(90.5,98){\makebox(0,0)[cc]{\small$2^{\mZ^n}$}}
\end{picture}
\end{center}\vskip-1cm
\caption{ The  associated algebra $\cQ $ is invariant under the transition operator $T$ and its set-valued inverse $T^{-1}$. The backward algebra $\cQ ^-$  and the forward algebra $\cQ^+$ are invariant under $T$ and $T^{-1}$, respectively. Under the assumption that the matrix $L$ is iteratively nonresonant, the algebra  $\cQ $ is contained by the class $\cS$ of frequency measurable subsets of $\mZ^n$. The map $T$ is frequency preserving and mixing on $\cQ^+$. The algebra $\cQ ^- \bigcap \cQ ^+$ consists of trivial sets whose frequencies are either zero or one. } \label{fig2}
\end{figure}
In particular, in view of Theorem~\ref{algebraprop1}(b) (under the assumption that the matrix $L$ is iteratively nonresonant),  $ \bF(A) \in \{ 0, 1\} $ for all $ A \in \cQ ^- \bigcap \cQ ^+ $. Indeed, any set $A \in \cQ ^- \bigcap \cQ^+$ is self-independent in the sense that $\bF(A) = \bF\left(A \bigcap A\right) = (\bF(A))^2$ and hence, $\bF(A)$ equals either zero or one. This is  a version of Kolmogorov's zero-one law \cite{Shiryayev}.

Theorem~\ref{algebraprop}(b) and Theorem~\ref{algebraprop1}(c) show that, under the iterative nonresonance assumption on $L$, the quadruple $(\mZ^n, \cQ ^+, \bF, T)$ can be regarded as a  dynamical system with an invariant finitely additive probability measure $\bF$, which can be studied from the viewpoint of ergodic theory. In particular, Theorem~\ref{mixing} in the next section establishes a mixing property for this quadruple.

An ergodic theoretic result is provided by Theorem~\ref{prolongation} below. Note that the transition operator $T$ preserves the frequency $\bF$  on the forward algebra $\cQ ^+$, and this property does not necessarily hold for the associated algebra $\cQ $.  It turns out that  the restriction $\bF|_{\cQ^+}$ can be extended to a finitely additive probability measure on the whole associated algebra $\cQ$ in such a way that the extended measure is preserved under the transition operator.

\begin{theorem}\label{prolongation}
Suppose the matrix $L \in \mR^{n\x n}$ of the quantized linear $(R,L)$-system is iteratively nonresonant. Then:
\begin{itemize}
 \item[{\bf (a)}]
for any element $A \in \cQ $ of the associated algebra $\cQ$ in
(\ref{assalgebra}), there exists a limit
\begin{equation}\label{hatF}
    \wh{\bF}(A) =
    \lim_{k \to +\infty}
    \bF
    (
        T^{-k}(A)
    );
\end{equation}
\item[{\bf (b)}]
the functional $\wh{\bF}: \cQ  \to [0,1]$ is a finitely
additive probability measure on $(\mZ^n, \cQ )$ and
satisfies \begin{equation}\label{hatFprop1}
    \wh{\bF}
    (
        T^{-1}(A)
    ) =
    \wh{\bF}(A) \<
    \wh{\bF}
    (
        T(A)
    )
    \qquad
    {\rm for\ all}\
    A \in \cQ;
\end{equation}
 \item[{\bf (c)}]
 the measures $\wh{\bF}$ and $\bF$
 coincide on the forward algebra
 (\ref{backforalgebras}),
\begin{equation}\label{hatFprop2} \wh{\bF}(A) = \bF(A) \quad {\rm for\ all}\ A \in
\cQ ^+. \end{equation}
\end{itemize}
\end{theorem}

\begin{proof}  In order to prove the assertion (a) of the theorem, we  fix an arbitrary $A \in \cQ $. The absorbing property of the forward algebra $\cQ ^+$ with respect to $T^{-1}$ from Theorem~\ref{algebraprop}(d) implies that there exists $N \in \mN$ such that $T^{-k}(A) \in \cQ ^+$ for all $k \> N$. Hence, from the frequency preservation property of the transition operator $T$ on $\cQ^+$, established in Theorem~\ref{algebraprop1}(c), it follows that  $\bF(T^{-k}(A)) = \bF(T^{-N}(A))$ for all $k \> N$, which implies the convergence (\ref{hatF}). In order to prove the assertion (b), we note that the property that $\wh{\bF}$ is a finitely additive probability measure on $(\mZ^n, \cQ )$ follows from that of the functional $\bF \circ T^{-k}: \cQ  \to [0,1]$ for any $k\in \mN$.  The equality in (\ref{hatFprop1}) follows from the definition of $\wh{\bF}$. The inequality in (\ref{hatFprop1}) can be obtained by using the inclusion $A \subset T^{-1}(T(A))$ for any $A \subset \mZ^n$. Indeed, by combining this inclusion with the properties of $\wh{\bF}$ established above, it follows that
$$\wh{\bF}(A) \< \wh{\bF}(T^{-1}(T(A))) = \wh{\bF}(T(A))$$ for any $A \in \cQ $.  Finally, the assertion (c) is proved by noting that (\ref{hatFprop2}) follows from the measurability and frequency preservation properties
of the transition operator $T$ with respect to the forward algebra
$\cQ ^+$ (see Theorem~\ref{algebraprop}(b) and Theorem~\ref{algebraprop1}(c)), which completes the proof of the theorem.  \end{proof}

\subsection
{Independence and uniform distribution of quantization
errors}
\label{IUDQE}

We will now apply the frequency-based analysis on quasiperiodic subsets of the lattice $\mZ^n$  to the deviation of trajectories of the quantized linear $(R,L)$-system from those of the original linear system with a nonsingular matrix $L \in \mR^{n\x n}$. In one step of the system dynamics, such deviation is described by a map $E:\mZ^n
\to R^{-1}(0)$ defined by
 \begin{equation}
 \label{onesteperror}
    E(x) := L x - T(x)
\end{equation}
More generally, for any $k \in \mN$, the deviation of trajectories of the systems in $k$ steps of their evolution can be expressed as
\begin{equation}
\label{devN}
    L^N x - T^N(x)
    = \sum_{k=1}^{N}L^{N-k}E_k(x)
\end{equation}
in terms of maps $E_k:\mZ^n \to R^{-1}(0)$ defined by
\begin{equation}
    \label{kthsteperror}
    E_k := E \circ T^{k-1}.
\end{equation}

\begin{definition}\label{error}
The map $E_k:\mZ^n \to R^{-1}(0)$, defined by (\ref{onesteperror}) and (\ref{kthsteperror}), is called the $k$th quantization error.
\end{definition}

Also, for any $N \in \mN$, we define a map $\cE_N: \mZ^n \to ( R^{-1}(0) )^N$ which is formed from the first $N$ quantization errors as
\begin{equation}\label{calEN}
    \cE_N
    :=
        \begin{bmatrix}
        E_1 \\
        \vdots \\
        E_N
        \end{bmatrix}.
\end{equation}

\begin{lemma}\label{reperrors}
For any $N \in \mN$, there exists a  map  $g_N: \mR^{Nn} \to
(R^{-1}(0))^N$ such that
\begin{itemize}
\item[{\bf (a)}]
the map $\cE_N$ in (\ref{calEN}) is representable  as
\begin{equation}\label{calENgNLN}
    \cE _N =
    g_N \circ \cL _N^+,
\end{equation}
where the matrix $\cL _N^+ \in \mR^{Nn \x n}$ is given by
(\ref{calLN});
\item[{\bf (b)}]
the map $g_N$ is unit periodic with respect to its $Nn$ variables;
\item[{\bf (c)}]
$g_N$ is $\mes_{Nn}$-continuous;
\item[{\bf (d)}]
$g_N$ is a $\mes_{Nn}$-preserving bijection of the set
$(R^{-1}(0))^N$ onto itself.
\end{itemize}
\end{lemma}

\begin{proof}
We will carry out the proof by induction on $N$. By recalling (\ref{trans}), it follows  that the first quantization error $E_1 = E$ in
(\ref{onesteperror}) is representable as
\begin{equation}\label{EgL}
    \cE _1 = E = g \circ L,
 \end{equation}
where the map $g: \mR^n \to R^{-1}(0)$ is given by
\begin{equation}\label{auxg}
    g(u) = u - R(u).
\end{equation}
In view of the commutation property (\ref{commutation})  and Jordan
measurability of the set $R^{-1}(0)$, the map $g$ is unit periodic
with respect to its $n$ variables and is $\mes_n$-continuous. Furthermore, $g$ maps
the set $R^{-1}(0)$ identically  onto itself and hence, preserves
the $n$-dimensional Lebesgue measure on this set.  Therefore, the map $g$ in
(\ref{auxg}) indeed satisfies the conditions (a)--(d) of the lemma
for $N = 1$. Now, assume that the assertion of the lemma holds for some $N \in \mN$.
Consider the next map $\cE_{N+1}:\mZ^n \to (R^{-1}(0))^{N+1}$ in (\ref{calEN}):
\begin{equation}\label{nextcalE}
    \cE _{N+1}
    =
        \begin{bmatrix}
        \cE _N \\
        E_{N+1}
        \end{bmatrix}.
\end{equation}
From (\ref{devN}), it follows that
$$
    T^N(x) = L^N x - \sum_{k=1}^{N}
L^{N-k}E_k(x) ,
$$
which, in combination with (\ref{EgL}) and (\ref{auxg}), leads to
\begin{equation}\label{nexterror}
    E_{N+1}(x)
    :=
    E
    (
        T^N(x)
    ) =
    g
    (
        L^{N+1} x -
        F_N
        \cE _N(x)
    ),
\end{equation}
where the matrix $F_N\in \mR^{n \x Nn}$   is given by
$$
    F_N
    :=
    \begin{bmatrix}
        L^N & \ldots & L
        \end{bmatrix}.
$$
In view of
(\ref{calENgNLN}), the representation (\ref{nexterror}) implies that
\begin{equation}\label{nextE}
    E_{N+1} =
    h_N \circ \cL _{N+1}^+,
\end{equation}
where the map $h_N: \mR^{(N+1)n} \to R^{-1}(0)$ is defined by
\begin{equation}\label{auxhN}
    h_N(u)
    :=
    g
    \left(
        u_{N+1} -
        F_N
        g_N
        \left(
                \begin{bmatrix}
                u_1 \\
                \vdots \\
                u_N
                \end{bmatrix}
        \right)
    \right),
    \qquad
    u :=
        \begin{bmatrix}
        u_1 \\
        \vdots \\
        u_{N+1}
        \end{bmatrix},
\end{equation}
where the vector $u \in \mR^{(N+1)n}$ is partitioned into subvectors $u_1, \ldots, u_{N+1} \in \mR^n$. By substituting (\ref{nextE}) into (\ref{nextcalE}) and recalling (\ref{calENgNLN}), it follows that $ \cE _{N+1} = g_{N+1} \circ \cL _{N+1}^+ $, where the map $g_{N+1}: \mR^{(N+1)n} \to (R^{-1}(0))^{N+1}$ is given by
\begin{equation}\label{nextgN}
    g_{N+1}(u)
    :=
        \begin{bmatrix}
        g_N
        \left(
                \begin{bmatrix}
                u_1 \\
                \vdots \\
                u_N
                \end{bmatrix}
        \right) \\ \\
        h_N(u)
        \end{bmatrix}.
\end{equation}
Note that the map $g_{N+1}$  is the skew
product~\cite{Martin,Sinai} of the maps $g_N$ and $h_N$. From the properties of the map $g$ in (\ref{auxg}), it follows that the map $h_N$ in (\ref{auxhN}) is
unit periodic with respect to its $(N+1)n$ variables (which are the entries of the vectors $u_1, \ldots, u_{N+1}$) and is
$\mes_{(N+1)n}$-continuous. Moreover, for any fixed subvectors $u_1, \ldots, u_N
\in \mR^n$, the map
$$
    h_N
    \left(
            \begin{bmatrix}
            u_1 \\
            \vdots \\
            u_N \\
            \bullet
            \end{bmatrix}
    \right):
    \mR^n \to R^{-1}(0),
$$
which is obtained in (\ref{auxhN}) from $g$ by translating the argument,  is unit periodic in  its $n$ variables, $\mes_n$-continuous and bijectively maps the set $R^{-1}(0)$ onto itself, preserving the $n$-dimensional Lebesgue measure on this set. Therefore, the map $g_{N+1}$ in (\ref{nextgN})  is unit periodic in its $(N+1)n$ variables, $\mes_{(N+1)n}$-continuous and bijectively maps the set $(R^{-1}(0))^{N+1}$ onto itself. The property that $g_{N+1}$ preserves the $(N+1)n$-dimensional Lebesgue measure is established by a similar reasoning as for the skew products of measure preserving automorphisms. This completes the induction step and the proof of the lemma. \end{proof}

\begin{theorem}\label{distrerrors}
Suppose the matrix $L$ is iteratively nonresonant. Then the
quantization errors $E_k$ in (\ref{kthsteperror}) are mutually
independent, uniformly distributed on the set $R^{-1}(0)$ and
are measurable with respect to the forward algebra $\cQ ^+$. More precisely,
for any $N \in \mN$, the map $\cE_N$ in (\ref{calEN}) is
$\mes_{Nn}$-distributed over the set $(R^{-1}(0))^N$, and the algebra $\cS_{\cE_N}$, generated  by this map, coincides with the algebra $\cQ_N^+$ given by (\ref{calQN}).
\end{theorem}

\begin{proof}
For a fixed but otherwise arbitrary $N \in \mN$, let $g_N: \mR^{Nn} \to (R^{-1}(0))^N$ be the map satisfying the conditions (a)--(d) of Lemma~\ref{reperrors} and constructed in the proof of the lemma. Then (\ref{calENgNLN}) implies that
\begin{equation}\label{invcalENB}
    \cE _N^{-1}(B) =
    \left\{
        x \in \mZ^n:\
        \cL_N^+
        x \in g_N^{-1}(B)
    \right\}
\end{equation}
for any set $B \subset (R^{-1}(0))^N$. For what follows, $B$ is assumed to be Jordan measurable.   Now, since $R^{-1}(0)$ is a Jordan measurable cell in $\mR^n$, then, in view of Lemma~\ref{cellprop1}(d), the set  $(R^{-1}(0))^N$ is a Jordan measurable cell in $\mR^{Nn}$. In combination with the condition (b) of Lemma~\ref{reperrors}, this implies that $g_N^{-1}(B)$  is a translation invariant subset of $\mR^{Nn}$ which, in view of Lemma~\ref{cellprop2}(b), is representable as
\begin{equation}
\label{invgNB}
    g_N^{-1}(B) = G + \mZ^{Nn},
\end{equation}
with
\begin{equation}\label{auxsetG}
    G :=
    (
        R^{-1}(0)
    )^N
    \bigcap
    g_N^{-1}(B).
\end{equation}
Due to the condition (c) of Lemma~\ref{reperrors}, the set $G$
is Jordan measurable. Therefore, from (\ref{invcalENB}) and
(\ref{invgNB}), it follows that
\begin{equation}\label{invcalENBfinal}
    \cE _N^{-1}(B) =
    Q_{Nn}
    (
        G, \cL_N^+
    ),
\end{equation}
and hence, this set is $\cL_N^+$-quasiperiodic. By
Theorem~\ref{periodicset} under the iterative nonresonance assumption on $L$, the set $\cE _N^{-1}(B)$
is frequency measurable, with
\begin{equation}\label{FinvcalENB}
    \bF
    (
        \cE _N^{-1}(B)
    ) =
    \mes_{Nn} G.
\end{equation}
By recalling (\ref{auxsetG}) and the $\mes_{Nn}$-preserving
property of the map $g_N$  (see the condition (d) of
Lemma~\ref{reperrors}), it follows that $\mes_{Nn} G = \mes_{Nn} B$ and hence, (\ref{FinvcalENB}) takes the form
$$
    \bF
    (
        \cE _N^{-1}(B)
    ) =
    \mes_{Nn} B.
$$
Since the last equality holds for an arbitrary Jordan measurable set $B \subset (R^{-1}(0))^N$, then the map $\cE _N$ is indeed $\mes_{Nn}$-distributed (that is, uniformly distributed) over the set $(R^{-1}(0))^N$). Finally, by using the property that $g_N$ bijectively maps the set $(R^{-1}(0))^N$ onto itself (see the condition (d) of Lemma~\ref{reperrors}), it follows that for any Jordan measurable set $G \subset (R^{-1}(0))^N$, there exists a Jordan measurable set $B \subset (R^{-1}(0))^N$ which represents $G$ by (\ref{auxsetG}). In combination with (\ref{invcalENBfinal}), this  implies that the inclusion $\cS_{\cE_N} \subset \cQ_N^+$ is, in fact, an equality $\cS _{\cE _N} = \cQ _N^+$, which completes the proof of the theorem. \end{proof}

\begin{theorem}\label{distrerrors1}
Suppose the matrix $L\in \mR^{n\x n}$ of the quantized linear $(R,L)$-system is iteratively nonresonant. Then the algebras $\cS_{E_k}$, generated by the quantization errors $\cE_k$ in (\ref{kthsteperror}), are representable as
\begin{equation}\label{calSEk}
    \cS _{E_k} =
    T^{-(k-1)}
    (
        \cQ _1^+
    )
    \subset
    \cQ _k^{+}
    \quad
    {\rm for\ all}\
    k \in \mN,
\end{equation}
and are mutually independent in the sense that for
any $N \in \mN$ and any sets $A_k \in \cS_{E_k}$,\ $1 \< k \< N$,
\begin{equation}\label{FAkFak}
    \bF
    \left(
        \bigcap_{k=1}^{N} A_k
    \right) =
    \prod_{k=1}^{N}
    \bF
    (
        A_k
    ).
\end{equation}
\end{theorem}

\begin{proof}
The equality in (\ref{calSEk}) is obtained by using the definition (\ref{kthsteperror}) and the relation $\cS_{E} = \cQ_1^+$ which was established in the proof of Theorem~\ref{distrerrors}. The inclusion in (\ref{calSEk}) follows from Lemma~\ref{algebraprop2}(b). We will now prove the mutual independence of the algebras generated by the quantization errors. To this end, suppose $N\in \mN$, and let  $A_k \in \cS_{E_k}$ be arbitrary elements of the corresponding algebras, representable  as
\begin{equation}
\label{Ak}
    A_k =
    E_k^{-1}
    (
        G_k
    ) =
    \cE _k^{-1}
    (
        (
            R^{-1}(0)
        )^{k-1}
        \x
        G_k
    ),
    \qquad
    k = 1, \ldots, N,
\end{equation}
where $G_k$ are Jordan measurable subsets of $R^{-1}(0)$, and use is made of the maps $\cE_k$ defined by (\ref{calEN}). Then application of Theorem~\ref{distrerrors} leads to
\begin{equation}\label{FAk}
    \bF
    (
        A_k
    ) =
    \mes_{kn}
    (
        (
            R^{-1}(0)
        )^{k-1}
        \x
        G_k
    ) =
    \mes_n G_k,
\end{equation}
where the rightmost equality also follows from the frequency preservation property of the transition operator $T$ on the forward algebra $\cQ ^+$. On the other hand, the intersection of the sets $A_k$ in (\ref{Ak}) is representable as
$$
    \bigcap_{k=1}^N
    A_k =
    \cE _N^{-1}
    (
        G_1 \x \ldots \x G_N
    )
$$
and hence,   by Theorem~\ref{distrerrors}, its frequency takes the form
\begin{equation}\label{FintersectAk}
    \bF
    \left(
        \bigcap_{k=1}^N A_k
    \right) =
    \mes_{Nn}
    (
        G_1 \x \ldots \x G_N
    ) =
    \prod_{k=1}^{N}
    \mes_n G_k.
\end{equation}
A comparison of (\ref{FintersectAk}) with (\ref{FAk}) leads to
(\ref{FAkFak}), which completes the proof of the theorem. \end{proof}

In addition to extending the results of \cite{V_1967} on the asymptotic distribution of roundoff errors, Theorems~\ref{distrerrors} and~\ref{distrerrors1} clarify the role of the forward algebra $\cQ ^+$.   More precisely, under the assumption that $L$ is iteratively nonresonant, $\cQ^+$ is the minimal algebra containing all the mutually independent algebras $\cS_{E_k}$ of frequency measurable quasiperiodic subsets of the lattice, generated by the quantization errors for $k\in \mN$. Equivalently, the forward algebra describes events which pertain to  the deviation of positive semitrajectories of the quantized linear system from those of the original system over finite time intervals (which will be considered in Section~\ref{DDPSQLSS}).

The following theorem is a corollary from Theorem~\ref{distrerrors}
and develops an ergodic theoretic point of view for the quantized
linear $(R,L)$-system, for which, as mentioned before, the restriction of the frequency
$\bF$ to the forward algebra $\cQ ^+$ is an invariant finitely
additive probability measure.

\begin{theorem}\label{mixing}
Suppose the matrix $L\in \mR^{n\x n}$ is iteratively nonresonant. Then the quadruple
$(\mZ^n, \cQ ^+, \bF, T)$ satisfies the mixing
property
\begin{equation}\label{limFinvTkAB}
    \lim_{k \to +\infty}
    \bF
    \left(
        T^{-k}(A) \bigcap B
    \right) =
    \bF(A) \bF(B),
    \qquad
    A, B \in \cQ ^+.
\end{equation}
\end{theorem}

\begin{proof}
Let $A, B \in \cQ ^+$ be fixed but otherwise arbitrary elements of the forward algebra. Then there exists $N \in \mN$ such that $A, B \in \cQ_N^+$, where the algebra $\cQ _N^+$
is defined by (\ref{calQN}). Hence, in view of Theorem~\ref{distrerrors}, the sets $A$ and $B$ are representable as
\begin{equation}
\label{AGBH}
    A = \cE _N^{-1}(G),
    \qquad
    B = \cE _N^{-1}(H)
\end{equation}
for some Jordan
measurable sets $G, H \subset (R^{-1}(0))^N$. From the definitions (\ref{kthsteperror}) and (\ref{calEN}), it follows that
$$
    \cE _N \circ T^k =
        \begin{bmatrix}
        E_{k+1} \\
        \vdots \\
        E_{k+N}
        \end{bmatrix}
$$
for any $k\> 0$. In combination with the first of the equalities in (\ref{AGBH}), this implies that the set
$$
    T^{-k}(A)
    =
    (\cE_N \circ T^k)^{-1}(G)
    =
    \left\{
        x \in \mZ^n:\
            \begin{bmatrix}
            E_{k+1}(x) \\
            \vdots \\
            E_{k+N}(x)
            \end{bmatrix}
        \in G
    \right\}
$$
belongs to the algebra $\bigvee_{i=1}^N \cS _{E_{k+i}}$ generated by the
quantization errors $E_{k+1}, \ldots, E_{k+N}$. Since for
every $k \> N$,  the maps $\cE _N$ and\
$
        \begin{bmatrix}
        E_{k+1} \\
        \vdots \\
        E_{k+N}
        \end{bmatrix}
$
are formed from nonoverlapping sets of quantization
errors, they induce mutually independent algebras. Hence, by recalling the second equality from (\ref{AGBH}), it follows that
$$
    \bF
    \left(
        T^{-k}(A) \bigcap B
    \right) =
    \bF
    (
        T^{-k}(A)
    )
    \bF
    (
        B
    ) =
    \bF(A) \bF(B).
$$
Since the last two equalities hold for any $k \> N$, they  imply the convergence
(\ref{limFinvTkAB}), thus completing the proof. \end{proof}

Note that the proof of Theorem~\ref{mixing} establishes a stronger result: under the assumption that $L$ is iteratively nonresonant, the algebras $\cQ _N^+$ and $T^{-k}(\cQ_N^+)$ are mutually independent for all $N \in \mN$ and $k \> N$,  and hence, by the version of Kolmogorov's zero-one law of Section~\ref{PFFA}, the algebra $\cQ _N^+ \bigcap T^{-k}(\cQ _N^+)$  consists of trivial subsets of the lattice whose frequencies are either zero or one.

\subsection{Distribution of the deviation of positive semitrajectories of
the quantized linear and supporting systems}
\label{DDPSQLSS}

For what follows, we define an affine map $T_*: \mR^n \to \mR^n$ by
\begin{equation}\label{Tstar}
    T_*(x)
    :=
    L x - \mu,
 \end{equation}
where
 \begin{equation}\label{vectormu}
    \mu
    :=
    \int_{R^{-1}(0)}
    u \rd u
\end{equation}
is the mean vector of the uniform probability distribution over the cell $R^{-1}(0)$.

\begin{definition}\label{suplinsys}
For the quantized linear $(R,L)$-system, the dynamical system in $\mR^n$ with the transition operator $T_*$ given by (\ref{Tstar}) is referred to as the supporting system.
\end{definition}

For every $N \in \mN$, we introduce a \emph{deviation map} ${\delta}_N: \mZ^n \to \mR^n$, which describes the deviation between the $N$th iterates of the quantized and supporting systems as
\begin{equation}\label{deviation1}
    \delta_N(x)
    :=
    T_*^N(x) - T^N(x)
    =
    \sum_{k=1}^{N} L^{N-k}
    (
        E_k(x) - \mu
    ).
\end{equation}
Also, we define another map $\xi_N: \mZ^n \to \mR^n$ by
\begin{equation}
\label{deviation}
    \xi_N(x)
    :=
    x -
    T_*^{-N}
    (
        T^N(x)
    ) =
    \sum_{k=1}^{N} L^{-k}
    (
        E_k(x) - \mu
    ).
\end{equation}
The maps $\delta_N$ and $\xi_N$ are related to each other by a
linear transformation
\begin{equation}\label{linbij}
    \delta_N
    =
    L^N \xi_N
\end{equation}
and satisfy recurrence equations
\begin{equation}
\label{recursion}
    \delta_N = L \delta_{N-1} + E_N - \mu,
    \qquad
    \xi_N = \xi_{N-1} + L^{-N}(E_N - \mu)
\end{equation}
for all $N \in \mN$, with initial conditions $\delta_0 = \xi_0 = 0$.

For every $N \in \mN$, the map $\xi_N$ in (\ref{deviation}) is an affine function of the first $N$ quantization errors. Hence, in view of Lemma~\ref{mapcomposition} and Theorem~\ref{distrerrors} for any iteratively  nonresonant matrix $L$, a map $\Xi_N: \mZ^n \to \mR^{Nn}$ defined by
$$
    \Xi_N
    :=
        \begin{bmatrix}
        \xi_1 \\
        \vdots \\
        \xi_N
        \end{bmatrix}
    =
        \begin{bmatrix}
        L^{-1} & & 0 \\
        \vdots & \ddots & \\
         L^{-1} & \cdots & L^{-N}
         \end{bmatrix}
        \begin{bmatrix}
        E_1 -\mu \\
        \vdots \\
        E_N - \mu
        \end{bmatrix},
$$
is uniformly distributed over the set
 $$
        \begin{bmatrix}
        L^{-1} & & 0 \\
        \vdots & \ddots & \\
         L^{-1} & \cdots & L^{-N}
        \end{bmatrix}
    (
        R^{-1}(0) - \mu
    )^N
 $$
with constant probability density function (PDF) $ | \det L |^{\frac{N(N+1)}{2}} $ with respect to $\mes_{Nn}$. Moreover, the corresponding algebra $\cS_{\Xi_N}$ coincides with the algebra $\cQ _N^+$ defined by (\ref{calQN}).

Therefore,  the maps $\delta_N$ and $\xi_N$ can be regarded as $\mR^n$-valued   random vectors on the probability space $(\mZ^n, \cQ ^+, \bF)$. From (\ref{recursion}) and the mutual independence and uniform distribution of the quantization errors, it follows that the sequence $(\delta_N)_{N \> 1}$  has the structure of a homogeneous Markov chain,  whilst $(\xi_N)_{N \> 1}$  is a non-homogeneous Markov chain with independent increments (whose average values are zero since $\bA(E_k) = \mu$ in view of (\ref{vectormu})). We will use these properties in Section~\ref{FCLTDPSQLSS} in order to establish a functional central limit theorem for the sequence $(\xi_N)_{N \> 1}$ for a class of quantized linear systems.

\subsection{Distribution of the deviation of negative semitrajectories of
the quantized linear and supporting systems}
\label{DDNSQLSS}

For any $k \in \mN$, the preimage
$$
    T^{-k}(x)
    :=
    \left\{
        y \in \mZ^n:\
        T^k(y) = x
    \right\}
$$
of a point $x \in \mZ^n$
is a finite (possibly empty) subset of the lattice $\mZ^n$. The set-valued sequence $(T^{-k}(x))_{k \> 1}$ describes a \emph{negative semitrajectory} of the quantized linear $(R,L)$-system, and its fragment
\begin{equation}\label{Nthbasin}
    T^{-1}(x),
    \ldots, T^{-N}(x)
\end{equation}
is a \emph{basin of attraction} of depth $N$ for the point $x$. The aim of this subsection is to study the deviation of negative semitrajectories of the quantized linear $(R,L)$-system from those of the supporting system introduced in the previous section. To this end, we define  yet another  dynamical system on $\mZ^n$ with the transition operator
$\wt{T}:  \mZ^n \to \mZ^n$ given by
\begin{equation}
\label{TT}
    \wt{T}(x)
    :=
    R
    (
        \mu + T_*^{-1}(x)
    ),
\end{equation}
where $\mu$ is the mean vector from (\ref{vectormu}), and $T_*$ is the transition operator of the supporting system in (\ref{Tstar}) with the inverse
$$
    T_*^{-1}(x) = L^{-1}(x+\mu).
$$
Therefore, the map $\wt{T}$ in (\ref{TT}) is representable as
\begin{equation}
\label{comtrans}
    \wt{T} = \wt{R}  \circ L^{-1},
\end{equation}
where $\wt{R}: \mR^{n} \to \mZ^n$ is a quantizer given
by \begin{equation}
\label{comquant}
    \wt{R}(u) :=
    R
    (
        u +
        (
            I_n + L^{-1}
        )
        \mu
    ).
\end{equation}

\begin{definition}\label{comsys}
For the quantized linear $(R,L)$-system, the quantized linear $(\wt{R}, L^{-1})$-system with the quantizer  $\wt{R}$ in (\ref{comquant}), is called the  compensating system.
\end{definition}

Note that the map $(R,L)\mapsto (\wt{R},L^{-1})$ is an involution in the sense that the compensating system for the quantized linear $(\wt{R},L^{-1})$-system  coincides with  the quantized linear $(R,L)$-system.

We will now define quantization errors $\wt{E}_k : \mZ^n \to \wt{R}^{-1}(0)$ for the compensating system in a similar fashion to the quantization errors $E_k$ for the quantized linear $(R,L)$-system in Section~\ref{IUDQE}:
\begin{equation}\label{comerrors}
    \wt{E}_k
    :=
    \wt{E}
    \circ
    \wt{T}^{k-1},
    \qquad
    k = 1, 2, 3, \ldots,
\end{equation}
where
\begin{equation}\label{firstcomerror}
    \wt{E}(x)
    :=
    L^{-1} x - \wt{T}(x).
\end{equation}
The quantization errors $\wt{E}_k$ of the compensating system take values in the cell $\wt{R}^{-1}(0)$ which, in view of (\ref{comquant}), is given by \begin{equation}\label{cominvR}
    \wt{R}^{-1}(0) =
    R^{-1}(0) -
    (
        I_n + L^{-1}
    )
    \mu.
\end{equation}

\begin{theorem}\label{distrcomerrors}
Suppose the matrix $L\in \mR^{n\x n}$ is iteratively nonresonant. Then:
\begin{itemize}
\item[{\bf (a)}]
the quantization errors $\wt{E}_k$ of the
compensating system in (\ref{comerrors}) are mutually independent, uniformly
distributed on the cell $\wt{R}^{-1}(0)$ in (\ref{cominvR}) and are measurable with
respect to the backward algebra $\cQ ^-$ given by
(\ref{backforalgebras}). More precisely, for any $N \in \mN$,
the map $\wt{\cE }_N: \mZ^n \to     (
        \wt{R}^{-1}(0)
    )^N
$, defined by
\begin{equation}\label{calcomE}
    \wt{\cE }_N
    :=
        \begin{bmatrix}
        \wt{E}_1 \\
        \vdots \\
        \wt{E}_N
        \end{bmatrix},
\end{equation}
is  $\mes_{Nn}$-distributed, and the algebra $\cS_{\wt{\cE }_N}$, generated by this map, coincides
with the algebra  $\cQ _N^-$ defined in (\ref{calQN});
\item[{\bf (b)}]
the algebras, generated by the quantization errors of the
compensating system, are representable as
$
    \cS _{\wt{E}_k} =
    \wt{T}^{-(k-1)}
    (
        \cQ _1^-
    )
    \subset
    \cQ _k^{-}
$
for all
$
    k \in \mN
$,
and are mutually independent.
\end{itemize}
\end{theorem}

\begin{proof}  It was remarked in Section~\ref{PFFA} that if the matrix $L$ is iteratively nonresonant then so is its inverse $L^{-1}$. Therefore, the assertions of the theorem can be established by applying Theorems~\ref{distrerrors} and~\ref{distrerrors1} to the compensating system. \end{proof}

Now, for any $N \in \mN$, we define a set-valued map  $\Sigma_N : \mZ^n \to
2^{\mZ^n}$ by using the transition operator $\wt{T}$ of the compensating system as
\begin{equation}\label{setSigmaN}
    \Sigma_N(x)
    :=
    T^{-N}(x) -
    \wt{T}^{N}(x).
\end{equation}
This allows  the preimage $T^{-N}(x)$ of a point $x\in \mZ^n$ to be represented as the translation of the set $\Sigma_N(x) \subset \mZ^n$ by the vector $\wt{T}^N(x) \in \mZ^n$:
\begin{equation}\label{invTNx}
    T^{-N}(x) =
    \wt{T}^N(x) + \Sigma_N(x).
\end{equation}
In particular, for $N=0$, the definition (\ref{setSigmaN}) implies that the set $\Sigma_0(x)$ is a singleton consisting of the origin:
\begin{equation}\label{setSigma0}
    \Sigma_0(x) = \{0\}.
\end{equation}
In view of this initial condition, it follows by induction that the sets $\Sigma_N(x)$ in (\ref{setSigmaN}) satisfy the recurrence equation
\begin{equation}\label{recSigma}
    \Sigma_N(x) =
    \mZ^n
    \bigcap
    (
        \wt{E}_N(x) + L^{-1}
        (
            \Sigma_{N-1}(x) + R^{-1}(0)
        )
    ),
    \qquad
    N \in \mN.
\end{equation}
Indeed, by using (\ref{setSigmaN}), (\ref{invTNx}) and recalling the first quantization error $\wt{E}$ from (\ref{firstcomerror}), it follows that
\begin{align*}
    \Sigma_N(x) & =
    \left(
        \mZ^n
        \bigcap
        (
            L^{-1}
            (
                T^{-(N-1)}(x) + R^{-1}(0)
            )
        )
    \right) -
    \wt{T}^N(x) \\
    & =
    \left(
        \mZ^n
        \bigcap
        \left(
            L^{-1}
            (
                \wt{T}^{N-1}(x) +
                \Sigma_{N-1}(x) + R^{-1}(0)
            )
        \right)
    \right) -
    \wt{T}^N(x) \\
    & =
    \left(
        \mZ^n
        \bigcap
        \left(
            L^{-1}
            \wt{T}^{N-1}(x)
            -
            \wt{T}^N(x)
            +
            L^{-1}
            (
                \Sigma_{N-1}(x) + R^{-1}(0)
            )
        \right)
    \right) \\
    & =
    \left(
        \mZ^n
        \bigcap
        \left(
            \wt{E}
            (
                \wt{T}^{N-1}(x)
            ) +
            L^{-1}
            (
                \Sigma_{N-1}(x) + R^{-1}(0)
            )
        \right)
    \right) \\
    & =
    \mZ^n
    \bigcap
    \left(
        \wt{E}_N(x) + L^{-1}
        (
            \Sigma_{N-1}(x) + R^{-1}(0)
        )
    \right)
\end{align*}
for any $N \in \mN$.
For what follows, we add the empty subset of the lattice to the class $\Pi$ in (\ref{totPi}) and denote by
\begin{equation}\label{spaceXi}
    \wt{\Pi}
    :=
    \left\{
        A \subset \mZ^n:\
        \# A < +\infty
    \right\} =
    \Pi
    \bigcup
    \{\emptyset\}
\end{equation}
the class of finite subsets of  $\mZ^n$. Also, we define a function $\Delta: \wt{\Pi}^2 \to [0,1]$ which maps a pair of such sets $A$ and $B$ to \begin{equation}\label{DelBA}
    \Delta(B \mid A) =
    \mes_n
    \left\{
        u \in \wt{R}^{-1}(0):\
        \mZ^n
        \bigcap
        \left(
            u  + L^{-1}
            (
                A + R^{-1}(0)
            )
        \right) = B
    \right\}.
\end{equation}
Note that
$ \sum_{B \in \wt{\Pi}} \Delta(B \mid A) = 1 $ for any $A \in \wt{\Pi}$, and hence, the function $\Delta(\cdot \mid A): \wt{\Pi} \to [0,1]$  describes a conditional probability mass function on the class $\wt{\Pi}$ (which is a denumerable set).

\begin{theorem}\label{markovbasin}
Suppose the matrix $L$ is iteratively nonresonant. Then the sequence of
the set-valued maps $\Sigma_N : \mZ^n \to \wt{\Pi}$ defined
by (\ref{setSigmaN}) is a homogeneous Markov chain with respect to
the filtration $\{\cQ _N^-\}_{N \> 1}$ formed from the algebras in
(\ref{calQN}) with the transition kernel $\Delta(\cdot \mid \cdot)$ in (\ref{DelBA}).
More precisely, for any $N \in \mN$ and any finite subsets $B_1, \ldots, B_N$ of the lattice $\mZ^n$ (that is, elements of the class $\wt{\Pi}$ in (\ref{spaceXi})),
$$
    \bigcap_{k=1}^{N}
    \Sigma_k^{-1}
    (
        B_k
    ) =
    \big\{
        x \in \mZ^n:\
        \Sigma_k(x) = B_k
        \ \
        {\rm for\ all}\
        k = 1, \ldots, N
    \big\} \in
    \cQ _N^-
$$
and
$$
    \bF
    \left(
        \bigcap_{k=1}^{N}
        \Sigma_k^{-1}
        (
            B_k
        )
    \right) =
    \prod_{k=1}^{N}
    \Delta
    (
        B_k \mid  B_{k-1}
    )
$$
with $B_0 := \{0\}$.
\end{theorem}

\begin{proof}
From (\ref{setSigma0}) and (\ref{recSigma}), it follows by induction that for any $N \in \mN$, the map $\Sigma_N$ in (\ref{setSigmaN}) is representable as
\begin{equation}\label{rhoNEN}
    \Sigma_N
    =
    \rho_N \circ \wt{\cE }_N.
\end{equation}
Here, $\wt{\cE }_N$ is the map defined by (\ref{calcomE}) and $\rho_N: (\wt{R}^{-1}(0))^N \to \wt{\Pi}$ is a $\mes_{Nn}$-continuous map satisfying the recurrence equation
\begin{equation}\label{maprhoN}
    \rho_N(u) =
        \mZ^n
        \bigcap
        \left(
            u_N + L^{-1}
            \left(
                \rho_{N-1}
                \left(
                        \begin{bmatrix}
                        u_1 \\
                        \vdots \\
                        u_{N-1}
                        \end{bmatrix}
                \right) +
                R^{-1}(0)
            \right)
        \right),
\end{equation}
where the vector
$
    u :=
        {\small\begin{bmatrix}
        u_1 \\
        \vdots \\
        u_N
        \end{bmatrix}}
    \in
    (
        \wt{R}^{-1}(0)
    )^N
$ is partitioned into $n$-dimensional subvectors $u_1, \ldots, u_N \in R^{-1}(0)$, and the initial condition $\rho_0 = \{0\}$ is used. The assertions of the theorem can now be obtained from the representation (\ref{rhoNEN}) and the recurrence relation (\ref{maprhoN}) by using Lemma~\ref{mapcomposition} and Theorem~\ref{distrcomerrors}. \end{proof}

The relation (\ref{invTNx}) implies that
\begin{equation}\label{invdev}
    T^{-N}(x) - T_*^{-N}(x) =
    \Sigma_N(x)-\wt{\delta}_N(x),
\end{equation}
where the maps $\wt{\delta}_N:
    \mZ^n \to \mR^n$ are defined by
$$
    \wt{\delta}_N
    :=
    T_*^{-N} - \wt{T}^N =
    \sum_{k=1}^{N}
    L^{k-N}
    (
        \wt{E}_k + L^{-1} \mu
    )
 $$
 and
satisfy the recurrence equation
$$
    \wt{\delta}_N =
    L^{-1}
    \wt{\delta}_{N-1} +
    \wt{E}_N +
    L^{-1} \mu,
    \qquad
    N = 1, 2, 3, \ldots,
 $$
with the initial condition $\wt{\delta}_0 = 0$. This equation is similar to the first of the recurrence relations
(\ref{recursion}).

Therefore, the importance of Theorems~\ref{distrcomerrors}
and~\ref{markovbasin} consists in representing the deviation
(\ref{invdev}) as the Minkowski sum  of two sequences driven
by mutually independent and uniformly distributed
quantization errors of the compensating system: the
$n$-dimensional sequence $\wt{\delta}_N(x)$ and the
set-valued sequence $\Sigma_N(x)$ which  are homogeneous
Markov chains with respect to the filtration $\{\cQ_N^-\}_{N \> 1}$ formed from the algebras (\ref{calQN}).

This  clarifies the role of the backward algebra $\cQ^-$ defined by (\ref{backforalgebras}): while the forward
algebra $\cQ  ^+$ is formed from events which pertain to the
deviation of positive semitrajectories of the quantized
linear $(R,L)$-system and the supporting system (see
Section~\ref{DDPSQLSS}), the backward algebra $\cQ ^-$
describes events related to the deviation of negative
semitrajectories of these systems.

Now, for any $k \in \mN$, we define a function $\nu_k : \mZ^n \to \mZ_+$ by
\begin{equation}\label{cardk}
    \nu_k(x)
    :=
    \# T^{-k}(x) =
    \# \Sigma_k(x)
\end{equation}
The corresponding map
$
    \cN_N :=
        \begin{bmatrix}
        \nu_1 \\
        \vdots \\
        \nu_N
        \end{bmatrix}:
    \mZ^n \to \mZ_+^N
$
describes the \emph{cardinality structure} of the basin
of attraction (\ref{Nthbasin}). From
Theorem~\ref{markovbasin}  (under the assumption that the matrix $L$ is iteratively nonresonant), it
follows that for any $N \in \mN$, the map $\cN_N$ is
distributed and $\cS_{\cN_N} \subset \cQ _N^-$.

\begin{theorem}\label{martingale}
Suppose the matrix $L$ is iteratively nonresonant. Then the
following functions, obtained by rescaling   (\ref{cardk}) as
\begin{equation}
\label{nut}
    \wt{\nu}_N :=
    |\det L|^N \nu_N:
    \mZ^n \to \mR_+,
\end{equation}
 form a martingale \cite{Shiryayev} with respect to the filtration
$\{\cQ _N^-\}_{N \> 1}$ of the  algebras
(\ref{calQN})  in the sense that
\begin{equation}
\label{mart}
    \bA
    (
        \wt{\nu}_{N+1} \cI_{A}
    ) =
    \bA
    (
        \wt{\nu}_N \cI_{A}
    )
    \quad
    {\rm for\ all}\ N \in \mN,\ A \in \cQ _N^- .
\end{equation}
\end{theorem}

\begin{proof}
Since, by Theorem~\ref{markovbasin}, the sequence $(\Sigma_N)_{N \> 1}$ of
the set-valued maps (\ref{setSigmaN}) is a homogeneous
Markov chain with respect to the filtration $\{\cQ_N^-\}_{N \> 1}$ with the transition probabilities
(\ref{DelBA}), the statement of Theorem~\ref{martingale}
will be proved  if we show that
\begin{equation}\label{DelBA1}
    \sum_{B \in \wt{\Pi}}
    \Delta(B \mid A) \# B
    =
    \frac{\# A}{|\det L|}
    \quad
    {\rm for\ all}\
    A \in \wt{\Pi}.
\end{equation}
To this end, we will first prove that if $V, G \subset \mR^n$ are Lebesgue measurable
and $V$ is a cell, then
\begin{equation}\label{DelBA2}
    \int_{V}
    \#
    \Big(
        \mZ^n
        \bigcap
        (
            u + G
        )
    \Big)
    \rd u =
    \mes_n G.
\end{equation}
Indeed,
\begin{align*}
    \int_V
    \#
    \Big(
        \mZ^n
        \bigcap
        (
            u + G
        )
    \Big)
    \rd u
    & =
    \int_V
    \sum_{z \in \mZ^n}
    \cI_{u+G}(z)
    du \\
    & =
    \sum_{z \in \mZ^n}
    \int_V
    \cI_{G}(z-u)
    du =
    \sum_{z \in \mZ^n}
    \mes_n
    (
        G \bigcap (z-V)
    )\\
     & =
    \mes_n
    \Big(G \bigcap \big(\mZ^n - V\big)\Big)
    =
    \mes_n G.
\end{align*}
A combination of (\ref{DelBA}) and (\ref{DelBA2}) allows the sum on
the left-hand side of (\ref{DelBA1}) to be represented as
\begin{align}
    \nonumber
    \sum_{B \in \wt{\Pi}}
    \Delta(B \mid A) \# B
    & =
    \int_{\wt{R}^{-1}(0)}
    \#
    \Big(
        \mZ^n
        \bigcap
        \big(
            u + L^{-1}
            (
                A + R^{-1}(0)
            )
        \big)
    \Big)
    \rd u \\
    \label{DelBA3}
    & =
    \mes_n
    (
        L^{-1}
        (
            A + R^{-1}(0)
        )
    ) =
    \frac
    {\mes_n
    (
        A + R^{-1}(0)
    )}
    {|\det L|}.
\end{align}
Since $R^{-1}(0)$ is a cell in $\mR^n$ and $A$ is a finite subset of
$\mZ^n$, then\ $ \mes_n ( A + R^{-1}(0) ) = \# A $.
Hence, the right-hand side of (\ref{DelBA3}) leads to that of (\ref{DelBA1}),
which completes the proof of the theorem. \end{proof}

The martingale property (\ref{mart}) of the rescaled functions (\ref{nut})   established  by Theorem~\ref{martingale} implies that (under the assumption that $L$ is iteratively nonresonant)
the average values  of the functions (\ref{cardk}) can be computed as $
\bA ( \nu_k ) = |\det L|^{-k}$ for all $ k \in \mN $.

In order to provide another corollary from Theorem~\ref{markovbasin}, we will now consider
the set $T^N(\mZ^n)$ of those points of $\mZ^n$ which are
\emph{reachable} in $N$ iterates of  the transition operator $T$. These reachability sets satisfy
\begin{equation}\label{Nreachable}
    T^{N+1}
    (
        \mZ^n
    )
    \subset
    T^N
    (
        \mZ^n
    ) =
    \{
        x \in \mZ^n:\
        \nu_N(x) > 0
    \}
    \quad
    {\rm for\  all}\
    N \in \mN.
\end{equation}

\begin{theorem}\label{holes}
Suppose the matrix $L$ is iteratively nonresonant. Then for any
$N \in \mN$, the frequency of the $N$th reachability set in (\ref{Nreachable}) can be computed as
$$
    \bF
    (
        T^N
        (
            \mZ^n
        )
    ) =
    H_N(\{0\})
$$
in terms of
functions $H_N: \Pi \to
[0,1]$ which are governed by the following linear recurrence equation
$$
    H_N(A)
    =
    \sum_{B \in \Pi}
    \Delta(B \mid A)  H_{N-1}(B)
$$
for
all $ A \in \Pi $, with
the initial condition $H_0 = 1$, where $\Delta(\cdot \mid \cdot)$ is the Markov transition kernel from (\ref{DelBA}).
\end{theorem}

\begin{proof}  The assertion of the theorem follows from
Theorem~\ref{markovbasin} and from the property that $\emptyset$ is
an absorbing state of the set-valued Markov chain $(\Sigma_N)_{N
\> 1}$ in its state space $\wt{\Pi}$ defined by (\ref{spaceXi}). \end{proof}

From Theorem~\ref{holes}, it follows that, under the assumption that the matrix $L$ is iteratively nonresonant,
\begin{equation}\label{FTZ}
    \bF
    (
        T
        (
            \mZ^n
        )
    ) =
    H_1(\{0\}) =
    1 -
    \Delta
    (
        \emptyset \mid \{0\}
    ),
\end{equation}
where
\begin{equation}\label{Del00}
    \Delta
    (
        \emptyset \mid \{0\}
    ) =
    \mes_n
    \Big\{
        u \in \wt{R}^{-1}(0):\
        \mZ^n \bigcap
        \big(
            u + L^{-1} R^{-1}(0)
        \big) =
        \emptyset
    \Big\}
\end{equation}
is the frequency of the set $\mZ^n \setminus T(\mZ^n)$ of ``holes'' in $\mZ^n$ (not reachable for $T$) which quantifies nonsurjectivity of the transition operator $T$. Moreover, since the relation (\ref{FTZ}) involves only the first iterate of $T$, it remains valid if the nonsingular matrix $L$ is resonant (that is, not necessarily \emph{iteratively} nonresonant), with $T(\mZ^n)$ being a frequency measurable $L^{-1}$-quasiperiodic subset of $\mZ^n$.

\section
{Quantized linear systems: neutral case}
\label{QLSNC}
\setcounter{equation}{0}

\subsection
{The class of quantized linear systems being considered}
\label{CQLSBC}

We will now consider a particular class of quantized linear $(R,L)$-systems on the integer lattice $\mZ^n$ of even dimension $n := 2r$, with the matrix $L \in \mR^{n \x n}$ being similar to an orthogonal matrix:
\begin{equation}\label{LUJU}
    L
    :=
    \sum_{k=1}^{r}
    U_k J_k V_k.
\end{equation}
Here, $U_k \in \mR^{n \x  2}$ and $V_k \in \mR^{2\x n}$
are the blocks of a nonsingular matrix $U \in \mR^{n\x n}$ and its
inverse,
\begin{equation}\label{UkVk}
    U  :=
        \begin{bmatrix}
         U_1 & \ldots & U_r
         \end{bmatrix},
        \qquad
        U^{-1}
         :=
            \begin{bmatrix}
            V_1 \\
            \vdots \\
            V_r
            \end{bmatrix},
 \end{equation}
and
\begin{equation}\label{Jk}
    J_k :=
        \begin{bmatrix}
        \cos \theta_k & - \sin \theta_k \\
        \sin \theta_k &   \cos \theta_k
        \end{bmatrix}
\end{equation}
are the matrices of planar rotation by angles $\theta_k \in (0,2\pi)$.
The matrix $L$, given by (\ref{LUJU})--(\ref{Jk}), has $r$
two-dimensional invariant subspaces $U_k \mR^2$. For what follows, we assume that the spectrum of $L$ is nondegenerate, that is,
its eigenvalues $\re^{\pm i\theta_k}$,
$1 \< k \< r$, are all pairwise different.

Note that, for any given nonsingular matrix  $U \in \mR^{n \x n}$, those vectors  $\theta := (\theta_k)_{1 \< k \< r} \in (0,2\pi)^r$,  for which the corresponding matrix  $L$ in (\ref{LUJU})--(\ref{Jk}) is iteratively resonant, form a set of $r$-dimensional Lebesgue measure zero. Indeed, the set of such vectors $\theta$ is representable as a countable union
$$
    \bigcup_{N \> 1}\
    \bigcup_{u \in \mZ^{Nn}\setminus \{0\},
    v \in \mZ^n}
    \Theta_{N,u,v}
$$
of the sets
$$
    \Theta_{N,u,v}
    :=
    \Theta_{N,u,v}^{(1)}
    \x
    \ldots
    \x
    \Theta_{N,u,v}^{(r)},
$$
where
$$
    \Theta_{N,u,v}^{(k)}
    :=
    \left\{
        \theta_k \in (0,2\pi):\
        \sum_{\ell=1}^{N}
        J_k^{-\ell} U_k^{\rT} u_{\ell} = U_k v
    \right\}.
$$
Here, the vector $u$ is partitioned into $n$-dimensional subvectors $u_1, \ldots, u_N$. For any $N \in \mN$,\ $u \in \mZ^{Nn}\setminus \{0\}$ and $v \in \mZ^n$, there exists at least one value of the index $k=1, \ldots, r$ such that the set $\Theta_{N,u,v}^{(k)}$ consists of real roots of a nonconstant trigonometric polynomial and hence, is finite. Therefore, each of the sets $\Theta_{N,u,v}$ has zero $r$-dimensional Lebesgue measure and so does their countable union.

Hence, for any given nonsingular matrix $U \in \mR^{n \x n}$, the matrix $L$,  described by (\ref{LUJU})--(\ref{Jk}), is iteratively nonresonant and has a nondegenerate spectrum for $\mes_r$-almost all vectors $\theta \in (0,2\pi)^r$ of rotation angles.

Since the eigenvalues of the matrix $L$ have unit modulus (and hence, the linear dynamical system $x \mapsto Lx$ is marginally stable),  the corresponding quantized linear $(R,L)$-system will be called \emph{neutral}.

\subsection{A functional central limit theorem for the deviation of positive
semitrajectories of  the quantized linear and supporting
systems}\label{FCLTDPSQLSS}

In what follows, we will need a technical lemma which is given below for completeness of exposition.

\begin{lemma}
\label{LPsiL} Suppose the matrix $L \in \mR^{n \x n}$, given by
(\ref{LUJU})--(\ref{Jk}), has a nondegenerate spectrum. Then, for any real
positive definite  symmetric matrix $\Psi$ of order $n$,  the following convergence holds:
\begin{equation}\label{matrixPhi}
    \Phi
    :=
    \lim_{N \to +\infty}
    \left(
        \frac{1}{N}
        \sum_{\ell=1}^{N}
         L^{-\ell} \Psi
         (
            L^{-\ell}
         )^{\rT}
    \right)
        =
        \frac{1}{2}
        \sum_{k=1}^{r}
        \sigma_k^2 U_k U_k^{\rT},
 \end{equation}
 where
\begin{equation}\label{sigmak}
     \sigma_k
     :=
     \sqrt
     {\Tr
     (
        V_k \Psi V_k^{\rT}
     )}.
\end{equation}
\end{lemma}
\begin{proof}
For any $N \in \mN$, let
\begin{equation}\label{matrixPhiN}
    \Phi_N
    :=
    \frac{1}{N}
    \sum_{\ell=1}^{N}
    L^{-\ell} \Psi \big(L^{-\ell}\big)^{\rT}
 \end{equation}
denote the matrix whose convergence is considered in (\ref{matrixPhi}). By using
(\ref{LUJU})--(\ref{UkVk}), this  matrix  can be represented as
\begin{equation}\label{matrixPhiN1}
    \Phi_N
    =
    \sum_{j,k=1}^{r}
    U_j \Phi_{j,k}^{(N)} U_k^{\rT},
\end{equation}
where
\begin{equation}\label{blockPhijk}
    \Phi_{j,k}^{(N)}
    :=
    \frac{1}{N}
    \sum_{\ell=1}^{N}
    J_j^{-\ell}
    \Psi_{j,k}
    J_k^{\ell},
\end{equation}
and
 \begin{equation}\label{Psijk}
    \Psi_{j,k}
    :=
    V_j  \Psi  V_k^{\rT}.
\end{equation}
Each of the planar rotation matrices in (\ref{Jk}) is representable as
\begin{equation}\label{JkW}
    J_k =
    W
        \begin{bmatrix}
        \re^{i\theta_k} & 0 \\
        0 &
    \re^{-i\theta_k}
        \end{bmatrix}
    W^*,
\end{equation}
where $(\cdot)^*$ denotes the complex  conjugate transpose, and
\begin{equation}\label{matrixW}
     W
     :=
     \frac{1}{\sqrt{2}}
        \begin{bmatrix}
        1 & 1 \\
        -i & i
        \end{bmatrix}.
 \end{equation}
Substitution of (\ref{JkW}) into (\ref{blockPhijk}) leads to
 \begin{equation}\label{PhijkN}
    \Phi_{j,k}^{(N)} =
    W
    (
        F_{j,k}^{(N)}
        \odot
        (
            W^* \Psi_{j,k} W
        )
    )
    W^*,
 \end{equation}
 where $\odot$ denotes the Hadamard product  of matrices \cite{Horn},  and \begin{equation}\label{matrixFjkN}
    F_{j,k}^{(N)} =
    \frac{1}{N}
    \sum_{\ell=1}^{N}
        \begin{bmatrix}
        \re^{-i\ell\theta_j} \\
        \re^{i\ell\theta_j}
        \end{bmatrix}
        \begin{bmatrix}
        \re^{-i\ell\theta_k} \\
        \re^{i\ell\theta_k}
        \end{bmatrix}^*.
 \end{equation}
Under the assumption that the spectrum of the matrix $L$ is nondegenerate, the following convergence holds:
 \begin{equation}\label{limFjkN}
    \lim_{N \to +\infty}
    F_{j,k}^{(N)}
    =
    \left\{
        \begin{matrix}
        I_2 & {\rm for} & j = k \\
        0 & {\rm for} & j \ne k
        \end{matrix}
    \right..
 \end{equation}
A straightforward verification leads to the property of the matrix $W$ in (\ref{matrixW}) that
\begin{equation}
\label{Hadiden}
    W
    \left(
        I_2 \odot
        (
            W^* F W
        )
    \right) W^*
    =
    \frac{\Tr F}{2}
    I_2
 \end{equation}
for any real symmetric $(2\x 2)$-matrix $F$. Now, from  (\ref{matrixPhiN1}) and (\ref{PhijkN})--(\ref{Hadiden}), it follows that
\begin{equation}
\label{xxxx}
    \lim_{N \to +\infty}
    \Phi_N =
    \frac{1}{2}
    \sum_{k=1}^{r}
    U_k U_k^{\rT}
    \Tr \Psi_{k,k}.
\end{equation}
In view of (\ref{sigmak}), (\ref{matrixPhiN}) and (\ref{Psijk}), the convergence (\ref{xxxx}) is equivalent to (\ref{matrixPhi}), which completes the proof of the lemma. \end{proof}

Note that the matrix $\Phi$, given by (\ref{matrixPhi}) and (\ref{sigmak}), is symmetric, positive definite and satisfies $ L \Phi L^{\rT} = \Phi $. The latter property implies the orthogonality of the matrix $\Phi^{-1/2} L \sqrt{\Phi}$, where $\sqrt{\Phi} \in \mR^{n\x n}$ is a (not necessarily symmetric) matrix square root of $\Phi$ (satisfying $\sqrt{\Phi}(\sqrt{\Phi})^{\rT} = \Phi$) with the inverse $\Phi^{-1/2} := (\sqrt{\Phi})^{-1}$. It will be convenient to use the following  matrix square root of $\Phi$:
\begin{equation}\label{Phiroot}
    \sqrt{\Phi}
    :=
    \frac{1}{\sqrt{2}}
        \begin{bmatrix}
        \sigma_1 U_1 & \ldots & \sigma_r U_r
        \end{bmatrix}
\end{equation}
which satisfies
\begin{equation}\label{PhiLPhi}
    \Phi^{-1/2}  L \sqrt{\Phi}
    =
        \begin{bmatrix}
        J_1 &  & 0 \\
        & \ddots &  \\
        0 &  & J_r
        \end{bmatrix}= J.
 \end{equation}

Now, let the matrix $\Psi \in  \mR^{n \x n}$ in Lemma~\ref{LPsiL} be the covariance matrix of  the uniform distribution over the set $R^{-1}(0)$:
\begin{equation}
\label{matrixPsi}
    \Psi
    :=
    \int_{R^{-1}(0)}
    (u-\mu)(u-\mu)^{\rT}\rd u,
\end{equation}
where $\mu$ is the mean vector from (\ref{vectormu}). The matrix $\Psi$ is symmetric and positive definite. Its nonsingularity follows from the property that $R^{-1}(0)$ is a cell in $\mR^n$ and hence, can not be contained by an $(n-1)$-dimensional hyperplane.

By using the matrix $\Phi \in \mR^{n\x n}$ expressed through (\ref{matrixPhi}) in terms of the matrix $\Psi$ in (\ref{matrixPsi}), and recalling (\ref{deviation}), we will now define, for arbitrary  $N \in \mN$ and $t \in [0,1]$, a map $W_{N,t}:  \mZ^n \to \mR^n$ by
\begin{equation}\label{WNt}
    W_{N,t}
    :=
    \frac{1}{\sqrt{N}}
    \Phi^{-1/2}
    \left(
        \xi_{\lfloor Nt \rfloor} + \lfp Nt \rfp
        \big(
            \xi_{\lceil Nt \rceil} -
            \xi_{\lfloor Nt \rfloor}
        \big)
    \right).
\end{equation}
Here, $\lceil  \cdot \rceil$  denotes  the ceiling function, and $t$ plays the role of a continuous  time variable. For any fixed $N \in \mN$ and $x \in \mZ^n$, the map $W_{N,\bullet}(x): [0,1] \to \mR^n$ is a continuous piecewise linear function which satisfies
$$
    W_{N,\frac{k}{N}}(x)
    =
        \frac{1}{\sqrt{N}}
    \Phi^{-1/2}
    \xi_k(x),
    \qquad
    k = 0, \ldots, N.
$$
Hence, the map $W_{N,t}$ in (\ref{WNt}) generates a map $\wt{W}_N: \mZ^n \to C$ which maps a point $x \in \mZ^n$ to the function $W_{N,\bullet}(x)$ belonging to the complete separable metric space
 \begin{equation}\label{spaceC0}
    C
    :=
    C_0
    (
        [0,1], \mR^n
    )
 \end{equation}
of continuous functions $f: [0,1] \to \mR^n$ satisfying $f(0) = 0$. The space $C$ is endowed in a standard fashion with the uniform norm $\| f \| := \max_{t \in [0,1]} |f(t)|$ and the corresponding Borel $\sigma$-algebra $\cB_C$.

Under the assumptions that the matrix $L$ is iteratively nonresonant and has a nondegenerate  spectrum, the family $\{W_{N,t}:\ t \in [0,1]\}$  of the maps (\ref{WNt}) can be regarded, for any given $N\in \mN$,  as an $\mR^n$-valued random process on the probability space $(\mZ^n, \cQ ^+, \bF)$. This follows from the property (which can be obtained by using Theorem~\ref{distrerrors}) that for any $p \in \mN$ and any reals $0 < t_1 < \ldots < t_p \< 1$, the map
\begin{equation}\label{colWNt}
        \mho_{N,t_1, \ldots, t_p}
        :=
        \begin{bmatrix}
        W_{N,t_1} \\
        \vdots \\
         W_{N,t_p}
         \end{bmatrix}:\,
    \mZ^n \to \mR^{pn}
\end{equation} is distributed with  a   countably additive probability measure
$D_{N, t_1, \ldots, t_p}$ on $(\mR^{pn}, \cB^{pn})$ (with $\cB^{pn}$ the $\sigma$-algebra of
Borel subsets of $\mR^{pn}$),  and the algebra induced by this map is
contained by the algebra $\cQ _N^+$ given by (\ref{calQN}).  The
probability measures  $D_{N,  t_1, \ldots, t_p}$ satisfy the
compatibility conditions of A.N.Kolmogorov~\cite{Shiryayev} and are the
finite-dimensional distributions of the process (\ref{WNt}).

On the other hand, under the above mentioned assumptions on matrix $L$, for any $N \in \mN$,  the map $\wt{W}_N:  \mZ^n \to C$,  specified by by (\ref{WNt}),  is distributed with a countably additive probability measure $\wt{D}_N$ on  $(C, \cB_C)$, and the algebra induced by this map is again contained by $\cQ _N^+$. Therefore, the map $\wt{W}_N$ can be regarded as a random element on the probability space $(\mZ^n, \cQ ^+,\bF)$ with values in the metric space $C$. The relationship between the probability measures $\wt{D}_N$ and $D_{N, t_1, \ldots, t_p}$ is described in terms of the cylinder sets as
$$
    D_{N,  t_1, \ldots, t_p}(B) =
    \wt{D}_N
    \left\{
        f \in C:
            \begin{bmatrix}
            f(t_1) \\
            \vdots \\
            f(t_p)
            \end{bmatrix}
        \in
        B
    \right\}
$$
for all $ B \in \cB^{pn} $.

We will now establish a functional central limit theorem for the map $\wt{W}_N$, that is, the weak convergence of the probability measure $\wt{D}_N$ as $N \to +\infty$, to the countably additive probability measure $\cW$ on $(C, \cB_C)$ corresponding to the $n$-dimensional standard Wiener process. This will be carried out using the standard scheme~\cite{Billingsley}:  first, the weak convergence of the finite-dimensional distributions $D_{N,t_1,\ldots,t_p}$ of the random process (\ref{WNt}) to those of the standard Wiener process will be proved; then it will be shown that the sequence of probability measures $(\wt{D}_N)_{N \> 1}$ satisfies the Prokhorov criterion of relative compactness in the topology of weak convergence of probability measures on $(C, \cB_C)$.

\begin{lemma}\label{flt}
Suppose the matrix $L$ in (\ref{LUJU})--(\ref{Jk}) is iteratively nonresonant and has a nondegenerate
spectrum. Then the finite-dimensional
distributions  of the  random process (\ref{WNt}) weakly
converge to those of the  $n$-dimensional standard Wiener
process as  $N \to +\infty$. More precisely, for any $p \in
\mN$,  any reals
 \begin{equation}\label{tk}
     t_0 := 0 < t_1 <\ldots < t_p \< 1
 \end{equation}
 and any Jordan measurable  set $B \subset \mR^{pn}$, the distribution of the map
        $\mho_{N, t_1, \ldots, t_p}$ in (\ref{colWNt}) satisfies
\begin{align}
    \nonumber
    \lim_{N \to +\infty}\!
    \bF\!
    \left\{
        x\in \mZ^n:\,
        \mho_{N, t_1, \ldots, t_p}(x) \in B
%
    \right\} \!
    =&
    (2\pi)^{-pn/2}
    \prod_{k=1}^{p}
    (
        t_k-t_{k-1}
    )^{-n/2} \\
    \label{limFANp}
    & \x\!
    \int_{B}
    \exp
    \left(
        -\frac{1}{2}
        \sum_{k=1}^{p}
        \frac{|u_k -u_{k-1}|^2}{t_k-t_{k-1}}
    \right)
    \rd u_1\!\x \ldots \x\! \rd u_p,
\end{align}
where $u_0 := 0$.
\end{lemma}

\begin{proof}
Let the numbers $p \in \mN$ and $t_1, \ldots, t_p$ in (\ref{tk}) be fixed but otherwise arbitrary. Then, for
any
$
    N
    \>
    \frac{1}{\min_{1 \< k \< p}(t_k -
    t_{k-1})}
$,
 the distribution $D_{N,  t_1, \ldots, t_p}$ of the map $\mho_{N, t_1, \ldots, t_p}$
in (\ref{colWNt})  is absolutely continuous with respect to
$\mes_{pn}$. Hence, for any such $N$, the following equality  holds for any Jordan measurable set
$B \subset \mR^{pn}$:
$$
    \bF
    \left\{
        x\in \mZ^n:\,
        \mho_{N, t_1, \ldots, t_p}(x)
        \in
        B
    \right\} =
    D_{N,  t_1, \ldots, t_p}(B).
$$
Therefore, the assertion of the lemma will be proved if we establish the weak convergence of the probability measure $D_{N,  t_1, \ldots, t_p}$ to the joint distribution of the values of the $n$-dimensional standard Wiener process at the moments of time (\ref{tk}). In view of the property that the sequence $(\xi_k)_{k\> 1}$ of the maps (\ref{deviation}) has independent increments, this task reduces to proving the weak convergence of the distribution $D_N$ of the map
\begin{equation}\label{WN1}
    W_{N,1} =
    \frac{1}{\sqrt{N}}\Phi^{-1/2}
    \xi_N
    =
    \frac{1}{\sqrt{N}}\Phi^{-1/2}     \sum_{k=1}^{N} L^{-k}
    (
        E_k - \mu
    ) :
    \mZ^n \to \mR^n
\end{equation}
to the $n$-dimensional Gaussian measure $D$ with zero
mean vector and the identity covariance matrix.
The probability  measure $D_N$ can be uniquely recovered from its
characteristic function $\phi_N: \mR^n \to \mC$
defined by \cite{Shiryayev}
\begin{equation}\label{phiN1}
    \phi_N(u) :=
    \int_{\mR^{pn}}
    \re^{i u^{\rT} v}
     D_N(\rd v).
 \end{equation}
From Theorem~\ref{distrerrors} and (\ref{WN1}), it follows that
\begin{equation}\label{phiN2}
    \phi_N(u) =
    \prod_{k=1}^{N}
    \phi
    \left(
        \frac{1}{\sqrt{N}}
        (
            \Phi^{-1/2} L^{-k}
        )^{\rT}
        u
    \right),
\end{equation}
where $\phi: \mR^n \to \mC$ is the characteristic function for
the uniform distribution over the set  $R^{-1}(0)-\mu$:
 \begin{equation}\label{phiv1}
    \phi(v)
    :=
    \int_{R^{-1}(0)}
    \re^{i v^{\rT} (w-\mu)}
     \rd w.
\end{equation}
Note that the function (\ref{phiv1}) is infinitely differentiable with respect to $v \in \mR^n$ and its asymptotic behaviour near the origin is
\begin{equation}\label{phiv2}
    \phi(v)
    =
    1 - \frac{1}{2} \| v \|_{\Psi}^2 + o(|v|^2)
    \quad
    {\rm as}\
    v \to 0,
\end{equation}
where $\|v\|_{\Psi} := \sqrt{v^{\rT} \Psi v}$ denotes the Euclidean
norm in $\mR^n$ associated with the covariance matrix $\Psi$ in (\ref{matrixPsi}). By using
(\ref{phiv2}), (\ref{phiN2}) and Lemma~\ref{LPsiL}, it follows that
\begin{align*}
    \lim_{N \to +\infty}
    {\rm ln}
    \phi_N(u)
    & =
    -\frac{1}{2}
     \lim_{N \to +\infty}
     \left(
        \frac{1}{N}
        \sum_{k=1}^{N}
        \left\|
            (
                \Phi^{-1/2} L^{-k}
            )^{\rT}
            u
        \right\|^2_{\Psi}
     \right) \\
    & =
    - \frac{1}{2}
    u^{\rT}
    \Phi^{-1/2}
    \lim_{N \to +\infty}
    \left(
        \frac{1}{N}
        \sum_{k = 1}^{N}
        L^{-k}
        \Psi
        (L^{-k})^{\rT}
    \right)
    (\Phi^{-1/2})^{\rT}
    u
     =
    - \frac{|u|^2}{2}
\end{align*}
holds for any $u \in \mR^n$,
 and hence,
$
     \lim_{N \to +\infty}
     \phi_N(u)  =
     \re^{-\frac{|u|^2}{2}}
$. The latter limit implies that $D_N$ converges weakly to the standard normal distribution $D$ in $\mR^n$
as $N \to +\infty$, and the lemma is proved. \end{proof}

In order to formulate the functional central limit theorem below, we will now define
the subalgebra
\begin{equation}\label{tildBC0}
    \wt{\cB}_C
    :=
    \left\{
        B \in \cB_C:\
        \cW
        (
            \partial B
        ) = 0\
        {\rm and}\
        \wt{D}_N
        (
            \partial B
        )=0\
        {\rm for\ all\ sufficiently\ large}\ N \in \mN \right\}
\end{equation}
of the algebra of $\cW$-continuous Borel subsets of the
metric space (\ref{spaceC0}) (as before, $\cW$ denotes  the
Wiener measure on $( C, \cB_C)$).

\begin{theorem}\label{mainflt}
Suppose the matrix $L$ in (\ref{LUJU})--(\ref{Jk}) is iteratively nonresonant and has a nondegenerate spectrum. Then,  for any $B \in
\wt{\cB}_C$,
\begin{equation}\label{mainflt1}
    \lim_{N \to +\infty}
    \bF
    \left\{
        x \in \mZ^n:\
        \wt{W}_N(x) \in B
    \right\} =
    \cW(B) .
\end{equation}
\end{theorem}
\begin{proof}
For any element $B\in \wt{\cB}_C$ of the algebra (\ref{tildBC0}) and for
all sufficiently large $N \in \mN$, the set
$
    \wt{W}_N^{-1}(B) =
    \left\{
        x \in \mZ^n:\
        \wt{W}_N(x) \in B
    \right\}
$
is frequency measurable, and its frequency is given by
$
    \bF
    (
        \wt{W}_N^{-1}(B)
    ) =
    \wt{D}_N(B)
$.
Since every $B \in  \wt{\cB}_C$ is a $\cW$-continuous Borel subset of $C$, the relation (\ref{mainflt1}) will follow from the weak convergence of the measure $\wt{D}_N$ to $\cW$ as $N \to +\infty$. Therefore, in view of Lemma~\ref{flt}, it remains to show that the sequence $(\wt{D}_N)_{N \> 1}$ satisfies the Prokhorov criterion of relative compactness in the topology of weak convergence of probability measures on $(C, \cB_C)$. The latter task  reduces to verifying the following condition (see, for example, \cite[Theorem 8.4]{Billingsley}):
 \begin{equation}\label{love00}
    \inf_{\lambda > 1}
    \left(
        \lambda^2
        \limsup_{N \to +\infty}
        \max_{0 \< j \< N}
        \bF
        \left\{
            x \in \mZ^n:\
            \max_{j \< k \< N}
            \left|
                W_{N, \frac{j}{N}}(x) -
                W_{N, \frac{k}{N}}(x)
            \right|
            \>
            \lambda
        \right\}
    \right) = 0 .
\end{equation}
To this end, let $N \in \mN$,\ $0 \< j \< N$ and $\lambda > 1$ be fixed but otherwise arbitrary. Since, under the assumption that $L$ is iteratively nonresonant, the sequence of the maps (\ref{deviation}) has mutually independent increments with zero average value, then a multidimensional version of \cite[inequality (10.7)]{Billingsley} leads to
\begin{align}
    \nonumber
    & \bF
    \left\{
        x \in \mZ^n:\
        \max_{j \< k \< N}
        \left|
            W_{N, \frac{j}{N}}(x) -
            W_{N, \frac{k}{N}}(x)
        \right| \> \lambda
    \right\} \\
    \nonumber
     & =
    \bF
    \left\{
        x \in \mZ^n:\
        \frac{1}{\sqrt{N}}
        \max_{j \< k \< N}
        \left|
            \Phi^{-1/2}
            \sum_{\ell = j+1}^{k} L^{-\ell}
            (
                E_{\ell}(x)-\mu
            )
        \right|
        \>
        \lambda
    \right\} \\
    \label{love01}
    & \<
    2
    \bF
    \left\{
        x \in \mZ^n:\
        \frac{1}{\sqrt{N}}
        \left|
            \Phi^{-1/2}
            \sum_{\ell = j+1}^{N} L^{-\ell}
            (
                E_{\ell}(x)-\mu
            )
        \right|
        \> \lambda - \sqrt{2}
    \right\} .
\end{align}
By applying the Chebyshev inequality \cite{Shiryayev} to the right-hand side of
(\ref{love01}), it follows that, for any $\lambda > \sqrt{2}$,
\begin{align}
    \nonumber
    & \bF
    \left\{
        x \in \mZ^n:\
        \frac{1}{\sqrt{N}}
        \left|
            \Phi^{-1/2}
            \sum_{\ell = j+1}^{N} L^{-\ell}
            (
                E_{\ell}(x)-\mu
            )
        \right|
        \>
        \lambda - \sqrt{2}
    \right\} \\
    \label{love02}
    & \qquad \qquad \<
    \frac{1}{N^2    (
        \lambda - \sqrt{2}
    )^4
}\,
    \bA
    \left(
        \left|
            \Phi^{-1/2}
            \sum_{\ell = j+1}^{N} L^{-\ell}
            (
                E_{\ell}-\mu
            )
        \right|^4
    \right)  .
 \end{align}
The fourth moment on the right-hand side of (\ref{love02}) can be estimated by using the mutual independence and uniform distribution of the quantization errors over  the set $R^{-1}(0)$. From the orthogonality of the matrix (\ref{PhiLPhi}) and from (\ref{matrixPsi}), it follows that
\begin{align}
\nonumber
    \bA
    \left(
        \left|
            \Phi^{-1/2}
            \sum_{\ell = j+1}^{N}
            L^{-\ell}
            (E_{\ell}-\mu )
        \right|^4
    \right)
    = &
    \bA
    \left(
        \left|
            \sum_{\ell = j+1}^{N}
            J^{-\ell}
            \Phi^{-1/2}
            (E_{\ell}-\mu)
        \right|^4
    \right) \\
\nonumber
     = &
    (N-j)
    \int_{R^{-1}(0)}
    \left|
        \Phi^{-1/2}
        (u - \mu)
    \right|^4 \rd u \\
\nonumber
    & +
    (N-j) (N-j-1)
    (
        \Tr
        \Upsilon
    )^2 \\
\nonumber
    & +
    2
    \sum_{j+1 \< \ell \ne m \< N}
    \Tr
    \left(
        \Upsilon
        J^{\ell -m}
        \Upsilon
        J^{m-\ell}
    \right) \\
 \label{love03}
    & \<
    N
    \int_{R^{-1}(0)}
    \left|
        \Phi^{-1/2}
        (u - \mu)
    \right|^4 \rd u + N^2 K,
 \end{align}
where the parameter $ K := (\Tr  \Upsilon)^2 + 2 \Tr (\Upsilon^2) $ is associated with an  auxiliary matrix $ \Upsilon := \Phi^{-1/2} \Psi (\Phi^{-1/2})^{\rT} $. The relations (\ref{love01})--(\ref{love03}) lead to the inequality
\begin{equation}
\label{last}
    \limsup_{N \to +\infty}
    \max_{0 \< j \< N}
    \bF
    \left\{
        x \in \mZ^n:\
        \max_{j \< k \<  N}
        \left|
            W_{N, \frac{j}{N}}(x) - W_{N, \frac{k}{N}}(x)
        \right|
        \>
        \lambda
    \right\}
    \<
    \frac{2 K}{(\lambda -\sqrt{2})^4}
\end{equation}
which holds for any $\lambda > \sqrt{2}$. Since the right-hand side of (\ref{last}) is $o(\lambda^{-2})$ as $\lambda \to +\infty$, this inequality establishes (\ref{love00}), which, in view of the above discussion, implies the assertion of the theorem.
\end{proof}

\begin{theorem}\label{clt}
Suppose the matrix $L$ in (\ref{LUJU})--(\ref{Jk}) is iteratively nonresonant and has a nondegenerate spectrum. Then the distribution  of the map
 \begin{equation}
    \label{deltanormed}
        \frac{1}{\sqrt{N}}
        \Phi^{-1/2} \delta_N : \mZ^n \to \mR^n
 \end{equation}
 (where use is made of (\ref{deviation1}), (\ref{matrixPhi}) and
 (\ref{matrixPsi}))  weakly converges to the $n$-dimensional
 Gaussian measure with the zero mean vector and identity covariance
 matrix as  $N \to +\infty$.  More precisely, for any Jordan measurable set
 $B \subset \mR^n$,
\begin{equation}
    \label{love}
    \lim_{N \to +\infty}
    \bF
    \left\{
        x \in \mZ^n:\
        \frac{1}{\sqrt{N}}
        \Phi^{-1/2}
        \delta_N(x)
        \in B
    \right\}
    =
    (2\pi)^{-n/2}
    \int_{B}
    \re^{- \frac{|u|^2}{2}}
    \rd u .
\end{equation}
\end{theorem}

\begin{proof}
The linear bijection (\ref{linbij}) between the maps (\ref{deviation1}) and (\ref{deviation}) allows the map (\ref{deltanormed}) to be represented as
\begin{equation}
\label{clt1}
        \frac{1}{\sqrt{N}}
        \Phi^{-1/2}
        \delta_N
        =
        \frac{1}{\sqrt{N}}
        \Phi^{-1/2}
        L^N \xi_N
        = J^N W_{N,1},
\end{equation}
where use is made of (\ref{PhiLPhi}) and (\ref{WN1}). Therefore, the map in (\ref{deltanormed}) is $D_N \circ J^{-N}$-distributed, where $D_N$ is the distribution of the map $W_{N,1}$. Hence, for any Jordan measurable $B \subset \mR^n$, the frequency of the set on the left-hand side of (\ref{love}) coincides with $D_N(J^{-N} B)$. Therefore, the assertion of the theorem can be proved by establishing the weak convergence of the probability measure $D_N \circ J^{-N}$ to the   $n$-dimensional Gaussian measure $D$ with the zero mean vector and  identity covariance matrix as $N \to +\infty$.  For this purpose, we will use the weak convergence of $D_N$ to $D$ as $N\to +\infty$ (which was established in the proof of Lemma~\ref{flt}), and the invariance of the probability measure $D$ with respect to orthogonal transformations (since standard normal PDFs are isotropic).
With any given bounded Lipschitz continuous function $f: \mR^n \to \mR$, we associate the functions
\begin{equation}
\label{fN}
    f_N := f \circ J^N,
    \qquad
    N \in \mN.
\end{equation}
Due to orthogonality of the matrix $J$ in (\ref{PhiLPhi}), the functions $f_N$ have the same Lipschitz continuity constant as the function $f$. Therefore, in view of the Arzela-Ascoli theorem, for any $\eps > 0$, the restrictions of these functions to a ball $ B_{\eps} := \{ u \in \mR^n:\ |u| \< \frac{1}{\eps}\}$ form a totally bounded set
\begin{equation}
\label{fNBeps}
  \left\{
    \left.
        f_N
    \right|_{B_{\eps}}:\
    N \in \mN
    \right\}
\end{equation}
in the Banach space $C(B_{\eps},\mR)$ of continuous functions on the ball. Hence, there exist $N_{\eps} \in \mN$, a surjective map\ $ \nu_{\eps}:  \mN \to \{1, 2, \ldots, N_{\eps}\} $ and a finite $\eps$-net
\begin{equation}
\label{epsnet}
    \{
        f_k^{(\eps)}:\
        1 \< k \< N_{\eps}
    \}
    \subset
    C(B_{\eps},\mR)
\end{equation}
for the set (\ref{fNBeps}) satisfying
\begin{equation}
\label{setfromnet}
    \sup_{N \> 1,\ u \in B_{\eps}}
    \big|
        f_N(u)-f_{\nu_{\eps}(N)}^{(\eps)}(u)
    \big|
    \<
    \eps .
\end{equation}
Now, recalling (\ref{clt1}) and  (\ref{fN}), consider the  following integrals
\begin{align}
\nonumber
    \bA\left(        f\left(\frac{1}{\sqrt{N}}
        \Phi^{-1/2}
        \delta_N\right)
\right)
& =
    \int_{\mR^n}
    f(u) (D_N \circ J^{-N})(\rd u)\\
\label{fDJN}
    & =
    \int_{B_{\eps}}
    f_N(u) D_N(\rd u)
    +
    \int_{\mR^n \setminus B_{\eps}}
    f_N(u) D_N(\rd u)  .
\end{align}
In view of the weak
convergence of $D_N$ to $D$, the rightmost integral in
(\ref{fDJN}) can be bounded asymptotically as
\begin{equation}
\label{lab3}
    \limsup_{N \to +\infty}
    \left|
        \int_{\mR^n \setminus B_{\eps}}
        f_N(u) D_N(\rd u)
    \right|
    \<
    \left(
        1 -
        D(B_{\eps})
    \right)
    \|f\|,
\end{equation}
where $\|f\|:=     \sup_{u \in \mR^n} |f(u)|$ is the uniform norm of $f$.
A combination of the same weak convergence with  (\ref{setfromnet}) leads to
\begin{equation}\label{lab1}
    \liminf_{N \to +\infty}
    \int_{B_{\eps}}
    f_N(u) D_N(\rd u)
    \>
    \min_{1 \< k \< N_{\eps}}
    \int_{B_{\eps}}
    f_k^{(\eps)}(u) D(\rd u) - \eps .
\end{equation}
On the other hand, from (\ref{setfromnet}) and the rotational invariance of the Gaussian measure $D$, it follows that
\begin{equation}\label{lab2}
    \min_{1 \< k \< N_{\eps}}
    \int_{B_{\eps}}
    f_k^{(\eps)}(u)
    D(\rd u)
    \>
    \min_{N \> 1}
    \int_{B_{\eps}}
    f_N(u) D(\rd u) - \eps
    =
    \int_{B_{\eps}}
    f(u) D(\rd u) - \eps .
\end{equation}
A combination of (\ref{lab1}) with (\ref{lab2}) yields  $ \liminf_{N \to +\infty} \int_{B_{\eps}} f_N(u) D_N(\rd u) \> \int_{B_{\eps}} f(u) D(\rd u) - 2 \eps $. Hence, in view of  (\ref{lab3}), it follows that
$$
    \liminf_{N \to +\infty}
    \int_{\mR^n}
    f_N(u) D_N(\rd u)
    \>
    \int_{\mR^n}
    f(u) D(\rd u) - 2
    \Big(
        \eps +
        \left(
            1 - D(B_{\eps})
        \right)
        \|f\|
    \Big).
$$
The arbitrariness of $\eps > 0$ in the latter inequality and the property $\lim_{\eps \to +0} D(B_{\eps}) = 1$ imply that
$$
    \liminf_{N \to +\infty} \int_{\mR^n} f_N(u) D_N(\rd u) \> \int_{\mR^n} f(u) D(\rd u).
$$
A similar reasoning leads to the inequality
$$
    \limsup_{N \to +\infty} \int_{\mR^n} f_N(u) D_N(\rd u) \< \int_{\mR^n} f(u) D(\rd u).
$$
A combination of the last two inequalities leads to the following convergence for the left-hand side of (\ref{fDJN}):
\begin{equation}\label{fff}
    \lim_{N \to +\infty}
    \int_{\mR^n}
    f(u) (D_N \circ J^{-N})(\rd u) = \int_{\mR^n} f(u) D(\rd u).
\end{equation}
Since $f: \mR^n \to \mR$ is an arbitrary bounded Lipschitz continuous function, then, in view of the well-known criterion~\cite{Shiryayev} of the weak convergence of probability measures, (\ref{fff}) implies that $D_N \circ J^{-N}$ weakly converges to $D$ as $N \to +\infty$, whence the assertion of the theorem follows.
\end{proof}

\begin{theorem}
\label{fltcor}
Suppose the matrix $L$ in (\ref{LUJU})--(\ref{Jk}) is iteratively nonresonant and has a nondegenerate spectrum. Then for any Jordan
measurable set $B \subset \mR_+^r$,
\begin{align}
    \nonumber
    \lim_{N \to +\infty}
    \bF&
    \left\{
        x \in \mZ^n:\
        \sqrt{\frac{2}{N}}
        \left(
            \frac
            {\max_{1 \< k \< N}
            |
                V_j \delta_k(x)
            |}
            {\sigma_j}
        \right)_{1\< j \< r}
        \in
        B
    \right\} \\
    \label{love000}
    & =
    \int_{B}
    \prod_{k=1}^{r}
    \tau
    (
        u_k
    )
    du_1
    \x
    \ldots
    \x
    du_r,
\end{align}
where use is made of (\ref{deviation1}), (\ref{matrixPhi}) and (\ref{matrixPsi}), and $\tau(\cdot)$ denotes the PDF for the largest Euclidean deviation of the two-dimensional standard Wiener process from the origin over the time interval  $[0,1]$.
\end{theorem}

\begin{proof}
By using (\ref{Phiroot}) and (\ref{PhiLPhi}), it follows that,
for any $1 \< j \< r$ and  $1 \< k \< N$,
\begin{equation}\label{love001}
    \sqrt{\frac{2}{N}}
    \sigma_j^{-1}
    V_j
    \delta_k =
    J_j^k
    \Pi_j
    W_{N, \frac{k}{N}},
\end{equation}
where
\begin{equation}\label{Pij}
    \Pi_j :=
        \begin{bmatrix}
        0_{2 \x 2(j-1)} & I_2 & 0_{2\x (n-2j)}
        \end{bmatrix}
    \in
    \mR^{2 \x n}
\end{equation}
with $0_{p \x q}$ denoting the zero $(p  \x q)$-matrix. Since each of the rotation matrices $J_j$ is orthogonal, (\ref{love001})  implies that
\begin{equation}\label{love002}
    \sqrt{\frac{2}{N}}
    \sigma_j^{-1}
    |
        V_j
        \delta_k
    | =
    \left|
        \Pi_j
        W_{N, \frac{k}{N}}
    \right| .
\end{equation}
Note that the vector $W_{N,t}(x)$ in (\ref{WNt}) depends on the auxiliary time variable $t \in [0,1]$ in a continuous   piece-wise linear fashion (it is a linear function of $t \in [(k-1)/N, k/N]$ for any $1 \< k \< N$). Therefore,  it follows  from (\ref{love002}) that
\begin{equation}
\label{love003}
    \sqrt{\frac{2}{N}}\,
    \frac
    {\max_{1 \< k \< N}
    |
        V_j
        \delta_k
    |}
    {\sigma_j} =
    \max_{t \in [0,1]}
    |
        \Pi_j
        W_{N,t}
    | .
\end{equation}
Now, consider the vector-valued map
\begin{equation}\label{alphaN}
    \alpha_N
    :=
    \sqrt{\frac{2}{N}}
    \left(
        \frac
        {\max_{1 \< k  \< N}
        |
            V_j
            \delta_k
        |}
        {\sigma_j}
    \right)_{1\< j \< r} =
    \left(
        \max_{t \in [0,1]}
        |
            \Pi_j W_{N,t}
        |
    \right)_{1 \< j \< r}:
    \mZ^n \to \mR_+^r
 \end{equation}
whose entries are described by (\ref{love003}). For any $N \in \mN$, this map is representable as the composition $\alpha_N = \phi \circ \wt{W}_N$ of the map $\wt{W}_N: \mZ^n \to C$ defined above and a Lipschitz continuous map  $\phi: C \to \mR^r_+$ on the space (\ref{spaceC0}) given by
\begin{equation}
\label{mapphif}
    \phi(f)
    :=
    \left(
        \max_{t \in [0,1]}
        |
            \Pi_j f(t)
        |
    \right)_{1 \< j \< r} .
 \end{equation}
The  preimage $\phi^{-1}(B)$ of any Jordan measurable set
$B \subset \mR^r_+$ belongs to the algebra (\ref{tildBC0}), and
hence, application of Theorem~\ref{mainflt} leads to  $
    \lim_{N \to +\infty}
    \bF
    (
        \alpha_N^{-1}(B)
    ) =
    \cW(\phi^{-1}(B))
$. The latter convergence implies (\ref{love000}) in view of (\ref{Pij}), (\ref{alphaN}) and (\ref{mapphif}). Here, we have also used the property that
 the projections of the $2r$-dimensional standard Wiener process onto
$r$ pair-wise orthogonal two-dimensional subspaces are mutually
independent two-dimensional standard Wiener processes).
\end{proof}

\subsection{An application to the rounded-off planar rotations}
\label{ERPR}

We will now apply the above results to a dynamical system on the two-dimensional lattice $\mZ^2$, with which a celebrated problem on the rounded-off planar rotations \cite{Diamond1,Kuznetsov} is concerned. More precisely, consider the quantized linear $(R_*, L)$-system, where $R_*: \mR^2 \to \mZ^2$ is the roundoff quantizer with
\begin{equation}\label{invRstar}
    R_*^{-1}(0) =
    [-1/2, 1/2)^2
\end{equation}
(see Section~\ref{DQLS}),  and
\begin{equation}\label{rot}
    L :=
        \begin{bmatrix}
        \cos \theta & -\sin \theta \\
        \sin \theta & \cos \theta
        \end{bmatrix}
\end{equation}
is the matrix of rotation by angle $\theta \in (0,2\pi)$. This  dynamical system, whose transition operator is given by
$$
    T = R_*\circ L,
$$
is a particular case of the neutral quantized linear systems described in Section \ref{CQLSBC}. Indeed, the matrix $L$ in (\ref{rot}) is of the form (\ref{LUJU})--(\ref{Jk}), with $r = 1$ and $U = I_2$. Note that the assumption for the matrix $L$ to have a nondegenerate spectrum is equivalent to  $\theta \ne \pi$ which holds, for example, if $\theta \in (0, \pi/2)$.
Furthermore, it is assumed throughout this subsection that the rotation angle $\theta$ belongs to the set
\begin{equation}\label{setTheta}
    \Theta
    :=
    \big\{
        \theta \in (0,\pi/2):\
        {\rm the\ matrix}\ L\ {\rm in}\ (\ref{rot})\
        {\rm is\ iteratively\ nonresonant}
    \big\};
\end{equation}
see Section~\ref{PFFA}. As discussed in Section~\ref{CQLSBC}, the set $\Theta$ is of full Lebesgue measure, and moreover, the set $(0, \pi/2) \setminus \Theta$ of ``pathological'' values of the rotation angle is countable. Also, $\theta /\pi$ is irrational for any $\theta \in \Theta$.

The following theorem provides corollaries from the corresponding results of the previous subsection and Section~\ref{QLSGC}. For its formulation, we note that, in view of (\ref{invRstar}) and (\ref{rot}), the mean vector $\mu$ in (\ref{vectormu}) vanishes,
$$
    \mu = \int_{[-1/2, 1/2)^2} u \rd u = 0,
$$
and hence, the supporting system with the transition operator in (\ref{Tstar}) coincides with the original linear system. Also note that the matrix $\Phi$ in (\ref{matrixPhi}) and the covariance matrix $\Psi$ in (\ref{matrixPsi}) for the uniform distribution over the square in (\ref{invRstar}) take the form \begin{equation}\label{PhiPsi}
    \Phi = \Psi =
    \int_{[-1/2, 1/2)^2}
    u u^{\rT} \rd u =
    \frac{I_2}{12}.
\end{equation}

\begin{theorem}\label{cor1}
Suppose the rotation angle $\theta$ belongs to the set $\Theta$ in
(\ref{setTheta}).  Then the quantized linear $(R_*,
L)$-system, specified by (\ref{invRstar}) and (\ref{rot}), satisfies the following properties:
\begin{itemize}
\item[{\bf (a)}]
the quantization  errors $E_k$ given by (\ref{kthsteperror})
are mutually independent and uniformly distributed over the square
$[-1/2, 1/2)^2$. In  particular,
$$
    \bF
    \left(
        \bigcap_{k=1}^{N}
        E_k^{-1}(B_k)
    \right) =
    \prod_{k=1}^{N}
    \mes_2 B_k
$$
holds
 for any $N \in \mN$ and any
Jordan measurable subsets $B_1, \ldots, B_N$ of the
square;
\item[{\bf (b)}]
for any Jordan measurable set $B \subset \mR^2$, \begin{equation}\label{lim01}
    \lim_{N \to  +\infty}
    \bF
    \left\{
        x \in \mZ^2:\
        \sqrt{\frac{12}{N}}
        (
            L^N x - T^N(x)
        )
        \in
        B
    \right\} =
    \frac
    {1}
    {2\pi}
    \int_{B}
    \re^{-\frac{|u|^2}{2}}
    \rd u.
\end{equation}
In particular, for  any $\alpha \> 0$,
\begin{equation}\label{lim02}
    \lim_{N \to +\infty}
    \bF
    \left\{
        x \in \mZ^2:\
        \sqrt
        {\frac{12}{N}}
        |
            L^N x - T^N(x)
        | >
        \alpha
    \right\}
    =
    \re^{-\frac{\alpha^2}{2}};
\end{equation}
\item[{\bf (c)}]
for any $\alpha \> 0$,
$$
    \lim_{N \to +\infty}
    \bF
    \left\{
        x \in \mZ^2:\
        \sqrt
        {\frac{12}{N}}
        \max_{1 \< k \< N}
        |
            L^k x - T^k(x)
        | >
        \alpha
    \right\} =
    \int_{\alpha}^{+\infty}
    \tau(u)
    \rd u,
 $$
where $\tau:\mR_+ \to \mR_+$ is the PDF of the largest Euclidean deviation of the two-dimensional standard Wiener process from the origin over the time interval $[0,1]$;

\item[{\bf (d)}]
the cardinality of the preimage  of a point under the
transition operator $T$ takes values $0, 1, 2$ with the following frequencies
\begin{equation}\label{freq03}
    \bF
    \left\{
        x \in \mZ^2:\
        \# T^{-1}(x) = k
    \right\} =
    \left\{
        \begin{matrix}
        \beta    & {\rm for} & k = 0\\
        1-2\beta & {\rm for} & k = 1\\
        \beta    & {\rm for} & k = 2\\
        \end{matrix}
    \right.,
\end{equation}
where
\begin{equation}\label{beta00}
    \beta :=
    1-
    \bF
    (
        T
        (
            \mZ^2
        )
    ) =
    (
        \cos \theta + \sin \theta -1
    )^2
\end{equation}
is the frequency of ``holes'' in the  lattice $\mZ^2$ which are not reachable for the system.
\end{itemize}
\end{theorem}
\begin{proof}
The assertion (a) of the theorem is a corollary from
Theorem~\ref{distrerrors}. The assertion (b) is established by noting
that (\ref{lim01}) follows from (\ref{PhiPsi})  and
Theorem~\ref{clt}.  The relation (\ref{lim02}) can be
obtained from (\ref{lim01}) by using the fact that
the squared Euclidean norm of a two-dimensional  Gaussian random
vector with zero mean and identity covariance matrix has
the $\chi^2$-distribution with two degrees of freedom:
$$
    \frac{1}{2\pi}
    \int_{u \in \mR^2:\ |u|^2 > \alpha}
    \re^{
        -\frac{|u|^2}{2}
    }
    du  =
    \re^{
        -\frac{\alpha}{2}
    }
    \quad
    {\rm for\ all}\
    \alpha \> 0.
$$
The assertion (c) follows from Theorem~\ref{fltcor}. The
assertion (d) is a  corollary from
Theorems~\ref{markovbasin} and~\ref{holes}. Indeed, by using
(\ref{DelBA}), it follows that
\begin{align}
    \nonumber
    \bF &
    \left\{
        x \in \mZ^2:\
        \# T^{-1}(x) = k
    \right\} \\
    \label{fr04}
    & =
    \mes_2
    \left\{
        u \in [-1/2, 1/2)^2:\
        \#
        \left(
            \mZ^2
            \bigcap
            (
                u + L^{-1} [-1/2, 1/2)^2
            )
        \right)
        = k
    \right\}.
\end{align}
In combination with (\ref{FTZ}) and (\ref{Del00}) (see also Fig.~\ref{fig3}), the relation (\ref{fr04}) leads to
(\ref{freq03}) and (\ref{beta00}). \end{proof}

\begin{figure}[htbp]
\begin{center}
\unitlength=1.5mm
\linethickness{0.4pt}
\begin{picture}(55.00,50.00)

\put(5.00,20.00){\line(3,4){9.00}}
\put(5.00,20.00){\line(4,-3){12.00}}
\put(17.00,11.00){\line(3,4){9.00}}
\put(14.00,32.00){\line(4,-3){12.00}}
\put(15.50,21.50){\circle*{1.00}}
\put(13.50,19.50){\makebox(0,0)[cc]{$_0$}}
\put(15.50,21.50){\line(1,0){15.00}}
\put(15.50,21.50){\line(0,1){15.00}}

\put(20.00,20.00){\line(3,4){9.00}}
\put(20.00,20.00){\line(4,-3){12.00}}
\put(32.00,11.00){\line(3,4){9.00}}
\put(29.00,32.00){\line(4,-3){12.00}}
\put(30.50,21.50){\circle*{1.00}}
\put(30.50,19.50){\makebox(0,0)[cc]{$_1$}}
\put(5.00,35.00){\line(3,4){9.00}}
\put(5.00,35.00){\line(4,-3){12.00}}
\put(17.00,26.00){\line(3,4){9.00}}
\put(14.00,47.00){\line(4,-3){12.00}}
\put(15.50,36.50){\circle*{1.00}}
\put(13.50,36.50){\makebox(0,0)[cc]{$_1$}}

\put(20.00,35.00){\line(3,4){9.00}}
\put(20.00,35.00){\line(4,-3){12.00}}
\put(32.00,26.00){\line(3,4){9.00}}
\put(29.00,47.00){\line(4,-3){12.00}}
\put(30.50,36.50){\circle*{1.00}}
\put(30.50,36.50){\line(-1,0){15.00}}
\put(30.50,36.50){\line(0,-1){15.00}}

\put(5.00,36.00){\line(0,1){14.00}}
\put(9.50,48.00){\vector(-1,0){4.00}}
\put(9.50,48.00){\vector(4,-3){3.00}}
\put(10.50,51.00){\makebox(0,0)[cc]{$\theta$}}

\put(31.30,20.90){\line(4,-3){13.00}}
\put(32.80,10.40){\line(4,-3){7.00}}
\put(31.40,10.20){\line(-3,-4){5.50}}
\put(29.90,20.70){\line(-3,-4){10.00}}
\put(38.40,8.70){\vector(3,4){3.00}}
\put(40.50,11.50){\vector(-3,-4){3.00}}
\put(41.50,8.50){\makebox(0,0)[cc]{$\frac{1}{2}$}}
\put(24.60,8.30){\vector(4,-3){3.00}}
\put(25.40,7.70){\vector(-4,3){3.00}}
\put(23.50,6.00){\makebox(0,0)[cc]{$\frac{1}{2}$}}
\put(23,21.50){\vector(1,0){21}}
\put(15.50,29){\vector(0,1){21}}
\put(23,21.50){\line(-1,0){21}}
\put(15.50,29){\line(0,-1){21}}

\multiput(18.55,27.9)(0.4,-0.3){13}{$\cdot$}%
\multiput(18.85,28.3)(0.4,-0.3){13}{$\cdot$}%
\multiput(19.15,28.7)(0.4,-0.3){13}{$\cdot$}%
\multiput(19.45,29.1)(0.4,-0.3){13}{$\cdot$}%
\multiput(19.75,29.5)(0.4,-0.3){13}{$\cdot$}%
\multiput(20.05,29.9)(0.4,-0.3){13}{$\cdot$}%
\multiput(20.35,30.3)(0.4,-0.3){13}{$\cdot$}%
\multiput(20.65,30.7)(0.4,-0.3){13}{$\cdot$}%
\multiput(20.95,31.1)(0.4,-0.3){13}{$\cdot$}%
\multiput(21.25,31.5)(0.4,-0.3){13}{$\cdot$}%
\multiput(21.55,31.9)(0.4,-0.3){13}{$\cdot$}%
\multiput(21.85,32.3)(0.4,-0.3){13}{$\cdot$}%
\end{picture}
\end{center}\vskip-10mm
\caption{Calculation of the frequency $\bF(\mZ^2 \setminus T(\mZ^2))$ of ``holes'' in the two-dimensional lattice $\mZ^2$ which are not reachable for the rounded-off planar rotation with a generic angle $\theta$. This frequency is equal to the area $(\cos\theta+\sin\theta-1)^2$ of the shaded  square.} \label{fig3}
\end{figure}
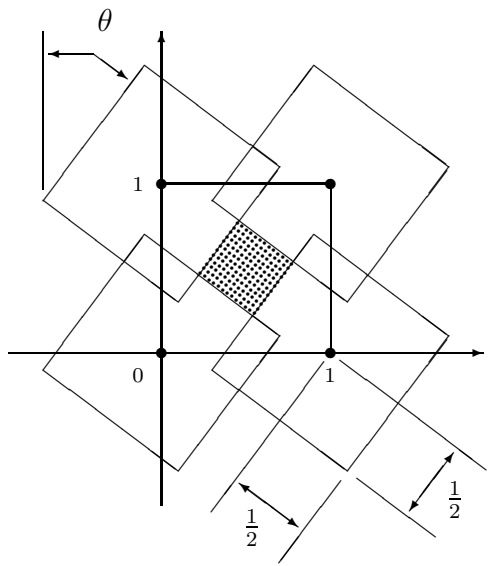

Unlike the statements (a)--(c) of Theorem~\ref{cor1}, its assertion (d) involves only the first iterate of the transition operator $T$ and  remains valid for nonresonant (that is, not necessarily \emph{iteratively} nonresonant) matrices $L$ in (\ref{rot}); see Section~\ref{DDNSQLSS}. Since for the orthogonal $(2\x2)$-matrices, the nonresonance property is equivalent to irrationality of $\re^{i\theta}$, then Theorem~\ref{cor1}(d) holds, in particular, for $\theta = \frac{\pi}{6}$. In this case, the rational independence of the rows of the corresponding matrix
$$
    \begin{bmatrix}
        I_2\\
        L
    \end{bmatrix}
    =
    \begin{bmatrix}
        1 & 0\\
        0 & 1\\
        \frac{\sqrt{3}}{2} & -\frac{1}{2}\\
        \frac{1}{2} & \frac{\sqrt{3}}{2}
    \end{bmatrix}
$$
can also be verified directly. For this rotation angle, a fragment of the reachability set $T(\mZ^2)$ is shown in  Fig.~\ref{fig4}.
\begin{figure}[htbp]
\begin{center}
\includegraphics[width=9cm]{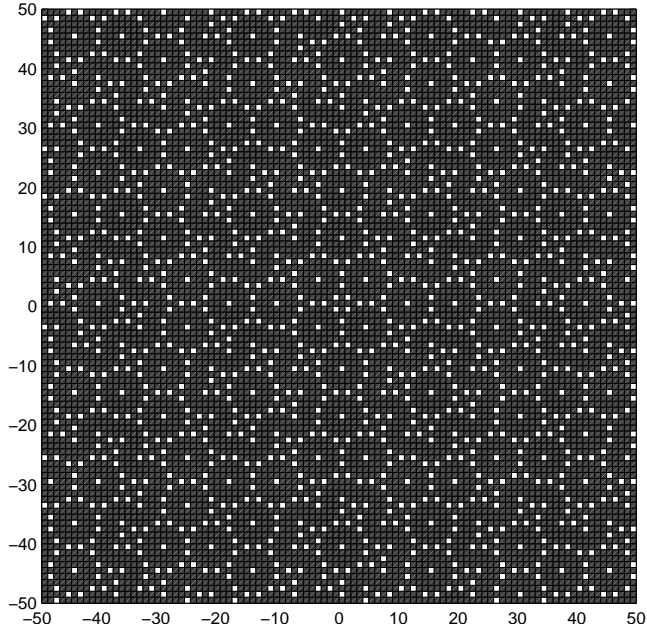}
\end{center}\vskip-4mm
\caption{A fragment $\{-50,\ldots, 50\}^2$ of the lattice $\mZ^2$, with the grey pixels depicting the
image $T(\mZ^2)$ of the lattice under the rounded-off planar rotation with angle $\theta = \frac{\pi}{6}$. The white pixels represent the
complement $\mZ^2 \setminus T(\mZ^2)$ consisting of ``holes'' in the lattice which are not reachable for the discretized system.} \label{fig4}
\end{figure}
This is an example of a frequency measurable $L^{-1}$-quasiperiodic subset of the lattice, which is not spatially periodic. According to (\ref{beta00}) of Theorem~\ref{cor1}(d),   the frequency of the reachability set is
\begin{equation}
\label{theory}
    \bF(T(\mZ^2))
    =
    1-
    \left(
        \cos \frac{\pi}{6} + \sin \frac{\pi}{6} -1
    \right)^2
    =
    \frac{\sqrt{3}}{2} = 0.8660...
\end{equation}
The relative fraction of reachable points in the moderately large fragment of the lattice under consideration is $0.8659...$, with the relative error of the theoretical prediction in (\ref{theory}) being $0.015\%$.

Also note that, by Theorem~\ref{cor1}(d), the frequency $\beta$  of ``holes'' for the rounded-off planar rotation in (\ref{beta00})  asymptotically achieves its largest value $ (\sqrt{2}-1)^2 = 0.1716...$ for generic rotation angles $\theta$ approaching $\frac{\pi}{4}$. For such rotation angles, the transition operator $T$ manifests ``minimal surjectivity''.

A comparison of the  above discussed
theoretical predictions with experimental results can also be found in \cite{Kozyakin}, where the relative fraction of points in sufficiently large rectangular fragments of $\mZ^2$ were calculated (along with their frequencies) for other events  relevant to the phase portraits of the rounded-off planar rotations. That comparison also demonstrates close proximity of the numerical experiment and the predictions provided by the frequency-based approach.

\section*{Acknowledgements}

The author is grateful to Victor S. Kozyakin, Alexei V. Pokrovskii, Eugene A. Asarin and Nikolai A. Kuznetsov for inspiring discussions on spatially  discretized dynamics in 1994--1996. Insightful discussions of the results of this report in November 1996   at Alexander Yu. Veretennikov's research seminar at the Institute for Information Transmission Problems, the Russian Academy of Sciences, are deeply appreciated. Linguistic corrections in the present version of the report have benefited from Anna I. Vladimirov's comments which  are also gratefully acknowledged.


\end{document}